\documentclass[a4paper, 12pt]{amsart}
\usepackage[margin=2.5cm]{geometry}
\usepackage{tikz-cd}
\usepackage{amssymb,amsrefs, mathrsfs}
\usepackage[draft=false]{hyperref}
\usepackage[noabbrev,capitalise,nameinlink]{cleveref}
\usepackage{bookmark}
\hypersetup{                                                         
  pdfauthor={},
  pdftitle={},
  colorlinks=true, 
  linkcolor={},
    citecolor={},
    urlcolor={},
    linktocpage=true                                                      % only page numbers are linked in toc
}

\newtheorem{thm}{Theorem}[section]
\newtheorem*{thmA}{Theorem A}
\newtheorem*{thmB}{Theorem B}
\newtheorem*{thmC}{Theorem C}
\newtheorem*{thmD}{Theorem D}
\newtheorem*{thmE}{Theorem E}
\newtheorem*{thmF}{Theorem F}
\newtheorem*{thmG}{Theorem G}
\newtheorem*{thmH}{Theorem H}
\newtheorem*{thmI}{Theorem I}
\newtheorem*{thmJ}{Theorem J}
\newtheorem*{thmK}{Theorem K}

\newtheorem{cor}[thm]{Corollary}

\newtheorem{lem}[thm]{Lemma}
\newtheorem{prp}[thm]{Proposition}

\theoremstyle{definition}
\newtheorem{dfn}[thm]{Definition}

\theoremstyle{remark}
\newtheorem{rmk}[thm]{Remark}
\newtheorem{qtn}[thm]{Question}

\numberwithin{equation}{section}

\makeatletter
\@namedef{subjclassname@2020}{%
  \textup{2020} Mathematics Subject Classification}
\makeatother

\begin{document}

\date{\today}
\title[The singular ideal and compressibility]{Characterizations of zero singular ideal in \'etale groupoid C*-algebras via compressible maps}

\author[Hume]{Jeremy B. Hume}

\address{School of Mathematics and Statistics  \\
Carleton University\\ 4302 Herzberg Laboratories\\ Ottawa, ON\\ K1S 5B6, Canada}
\email{jeremybhume@gmail.com}

\begin{abstract}

We show the singular ideal in a non-Hausdorff \'etale groupoid C*-algebra is zero if and only if every unit is contained, at the level of group representation theory, in the collection of subgroups of the unit's isotropy group obtained as limit sets of nets in the ``Hausdorff part'' of the unit space. This is achieved through a study of the interplay between the Hausdorff cover and restriction maps on C*-algebras of groupoids to reductions by closed locally invariant subsets, which we show are compressible to *-homomorphisms and therefore have many of the same properties. We also prove a simpler algebraic characterization of zero singular ideal that holds whenever the isotropy group C*-algebras satisfy a certain ideal intersection property. We prove this property holds for all direct limits of virtually torsion free solvable groups.
\end{abstract}



\maketitle

\section{Introduction}\label{sec:intro}
\label{s:intro}
Many interesting C*-algebras arise from \'etale groupoids which are not necessarily Hausdorff. For instance, Nekrashevych algebras associated to self-similar group actions \cite{N09}, C*-algebras from semigroups \cite{SS20} and the C*-algebras from germs of discrete group actions (see \cite{E11}) all have natural locally compact \'etale groupoid models with basic examples failing to be Hausdorff (although the unit spaces are always Hausdorff).

While the relationship between properties of Hausdorff etale groupoids and their C*-algebras is well understood (\cite{DR00}, \cite{T99}, \cite{D97}, \cite{BCFS14}) there has been difficulties extending these links to non-Hausdorff groupoids. An interesting obstruction to characterizing simplicity and the ideal intersection property in this case is the potential existence of elements in the reduced groupoid C*-algebra (viewed as functions on the groupoid) with zero sets dense in $G$, but are nonetheless non-zero. This is in contrast to the Hausdorff case, where all functions in the reduced C*-algebra are continuous. 

The set of these ``singular'' functions is a closed two-sided ideal in the reduced groupoid C*-algebra which turns out to be the only obstruction to generalizing the important results for Hausdorff groupoids on simplicity and the ideal intersection property (see \cite{CEPSS19} and \cite{KKLRU21}), and therefore the quotient by this \emph{singular ideal}, known as the \emph{essential groupoid C*-algebra} is well understood in these matters and in others (e.g., see \cite{EP22}).

However, none of the three important example classes of C*-algebras above are naturally modeled by the essential groupoid C*-algebras except when the obstruction vanishes. Thus, understanding when the singular ideal is zero, ideally characterizing when in terms of a groupoid property, is of great importance to the theory of \'etale groupoid C*-algebras. 

This is evidenced by the number of recent research works considering this problem; see the characterizations \cite{SS20} for groupoids from certain inverse semigroups, \cite{SS24} and \cite{GNSV25} for groupoids from self-similar groups, \cite{NS23} for groupoids with torsion free isotropy and \cite{BGHL25} for groupoids satisfying a certain finiteness condition on the ``non-Hausdorffness''. See also \cite{CEPSS19} for a sufficient condition for vanishing and \cite{CN24}, \cite{ACJM25} for investigations into the structure of the singular ideal. 

In this paper, we characterize when the singular ideal of an arbitrary \'etale groupoid is zero in terms of a groupoid property which is both topological and geometric in nature. Our work extends and simplifies the characterizations for the classes in \cite{NS23} and \cite{BGHL25} (the other interesting characterizations mentioned are example specific). 

Our characterization is achieved through an extensive study of restriction maps on C*-algebras of \'etale groupoids to reductions by closed \emph{locally invariant} subsets of the unit space. This is a new concept that is a relaxation of usual notion of invariance. Although these restriction maps are not *-homomorphisms, we show that they are at least \emph{compressible to *-homomorphisms} and therefore enjoy many of the same properties. In contrast to closed invariant sets, closed locally invariant sets are abundant in any \'etale groupoid - every finite set of units is an example.

In particular, our results on these maps and their interplay with the \emph{Hausdorff cover} as in \cite{BGHL25} allow us to calculate the image of the singular ideal in the C*-algebras of the isotropy groups (which are reductions to locally invariant sets). An application of this calculation shows groupoids whose isotropy groups satisfy a certain C*-ideal intersection property with the group rings have a simpler algebraic characterization for when the singular ideal is zero. We show many groups have this intersection property, including all groups of polynomial growth and amenable matrix groups (over characteristic zero fields).

We explain our results below.

First, we only need to consider the case of \'etale groupoids which are covered by countably many open bisections. Every $\sigma$-compact \'etale groupoid satisfies this assumption and it is easy to see the singular ideal is non-zero if and only if it is non-zero for some open subgroupoid C*-algebra generated by countably many open bisections. Therefore, any characterization of the singular ideal vanishing we establish in this paper with this assumption extends naturally to all \'etale groupoids.

Let $G$ be an \'etale groupoid with Hausdorff unit space $G^{0}$ (with the above assumption). We say $x\in G^{0}$ is \emph{Hausdorff} if every net in $G^{0}$ converging to $x$ has no other limit points as a net in $G$, and denote the collection of Hausdorff units by $C$. Since $G$ is covered by countably many bisections, $C$ is dense in $G^{0}$ (\cite[Lemma~7.15]{KM21}). This is the only reason we need the above assumption.

If $(x_{\lambda})\subseteq C$ is a net converging in $G^{0}$ to $x$, then, since $x_{\lambda}x_{\lambda} = x^{-1}_{\lambda} = x_{\lambda}$, continuity of the groupoid operations implies the set of limit points of $(x_{\lambda})$ is a subgroup $X$ of the isotropy group $G_{x}^{x}$. We say $X\in \mathcal{X}(x)$ if $X$ is ``maximal'' in the sense that any subnet $(x_{\lambda_{\mu}})$ has limit set $X$. By Corollary \ref{cor:compactness}, any net $(x_{\lambda})\subseteq C$ converging in $G^{0}$ to $x$ has a subnet with maximal limit set, so $\mathcal{X}(x)\neq\emptyset$. As we will see, $\mathcal{X}(x)$ is invariant under conjugation by elements in $G^{x}_{x}$ and is a compact Hausdorff space when equipped with the subspace topology arising from $\{0,1\}^{G^{x}_{x}}$, viewing $X\in\mathcal{X}(x)$ as identified with its indicator function $(1_{X}:G^{x}_{x}\to \{0,1\})\in \{0,1\}^{G^{x}_{x}}$.

For each subgroup $X\in \mathcal{X}(x)$, denote by $\lambda_{G^{x}_{x}/X}$ the \emph{quasi-regular} unitary representation of $G^{x}_{x}$ on $\ell^{2}(G^{x}_{x}/X)$, defined for $g\in G^{x}_{x}$ as $g\cdot \delta_{hX} = \delta_{ghX}$, for all cosets $hX\in G^{x}_{x}/X$ of $X$. Let $\lambda_{G^{x}_{x}/\mathcal{X}(x)}$ denote the representation $\oplus_{X\in\mathcal{X}(x)}\lambda_{G^{x}_{x}/X}.$ 

We motivate our characterization by first stating a special case. Denote by $J$ the \emph{singular ideal}, which is the set of functions in $f\in C^{*}_{r}(G)$ with zero set $f^{-1}(0)$ dense in $G$ (this is a non-standard but equivalent definition; see \cite[Lemma~4.1]{BGHL25}).

\begin{thmA}[\ref{thm:vsingidealcharwkcont}]
\label{thmA}
    Let $G$ be an \'etale groupoid with amenable isotropy groups. Then, $J = \{0\}$ if and only if for every $x\in G^{0}$, the left regular representation $\lambda_{G_{x}^{x}}$ is weakly contained in $\lambda_{G^{x}_{x}/\mathcal{X}(x)}$. 
\end{thmA}
More generally, the above theorem holds when the (\'etale) groupoid of cosets $G^{x}_{x}\cdot\mathcal{X}(x)$ of the subgroups $\mathcal{X}(x)$ (see Section \ref{ss:charvsing}) is amenable when $\{x\}\notin \mathcal{X}(x)$ (Theorem \ref{thm:vsingidealcharwkcont}).

Weak containment in the above case is equivalent to the property that for every $\epsilon > 0$ and finite set $F\subseteq G^{x}_{x}\setminus \{x\}$, there are vectors $\psi_{i}\in \ell^{2}(G^{x}_{x}/X_{i})$, where $X_{i}\in\mathcal{X}(x)$ for $i\leq n$ such that $$\sum^{n}_{i=1}\langle \psi_{i}, \psi_{i}\rangle =1\text{ and }|\sum^{n}_{i=1}\langle g\cdot\psi_{i}, \psi_{i}\rangle|\leq \epsilon,$$ for all $g\in F$. Essentially, it means $\{x\}\in\mathcal{X}(x)$ at the level of group representation theory of $G^{x}_{x}$.

In the general case, our characterization says $\{x\}\in\mathcal{X}(x)$ at the level of \emph{local groupoid representation theory} about $G^{x}_{x}$. We present it now.

If $g\in G^{x}_{x}$, $X\in\mathcal{X}(\tilde{x})$ and $U_{g}$ is an open bisection containing $g$, we can define a partial isometry on $\ell^{2}(G_{\tilde{x}}/X)$ by $U_{g}\cdot \delta_{hX} = 0$  when $r(h)\notin s(U_{g})$ and $U_{g}\cdot \delta_{hX} = \delta_{\tilde{g}hX}$ when $r(h)\in s(U_{g})$, where $\tilde{g}\in U_{g}$ is the unique element such that $s(\tilde{g}) = r(h)$.

Now, choose a bisection $U_{g}$ for every $g\in G^{x}_{x}$. We say $\lambda_{G^{x}_{x}}$ is \emph{$G$-weakly contained in $\lambda_{G^{x}_{x}/\mathcal{X}(x)}$} (Definition \ref{dfn:G-wkcomtain}) if for every $\epsilon > 0$, finite set $F\subseteq G^{x}_{x}\setminus \{x\}$ and open neighbourhood $U$ of $x$ with $U\subseteq \bigcap_{g\in F}r(U_{g})\cap s(U_{g})$, there are vectors $\psi_{i}\in \ell^{2}(G^{U}_{x_{i}}/X_{i})$, where $X_{i}\in\mathcal{X}(x_{i})$, $x_{i}\in U$ for $i\leq n$, such that 

    $$\sum^{n}_{i=1}\langle \psi_{i}, \psi_{i}\rangle = 1\text{ and }|\sum^{n}_{i=1}\langle U_{g}\cdot\psi_{i},\psi_{i}\rangle|\leq\epsilon,$$ for all $g\in F$.
This definition is independent of the bisections chosen and it
is dependent only on the germ isomorphism class of $G$ about $G^{x}_{x}$; see Definition \ref{dfn:hull}, \ref{dfn:germiso} and Proposition \ref{prp:dependsongerm}.

\begin{thmB}[\ref{thm:vsingidealcharlocwkcont_ess}]
    Let $G$ be an \'etale groupoid. Then, $J = \{0\}$ if and only if for every $x\in G^{0}$, $\lambda_{G^{x}_{x}}$ is $G$-weakly contained in $\lambda_{G^{x}_{x}/\mathcal{X}(x)}$.
\end{thmB}
By an application of \cite[Lemma~1.9]{NS23}, it is easy to see  $\lambda_{G^{x}_{x}}$ is never $G$-weakly contained in $\lambda_{G^{x}_{x}/\mathcal{X}}$ if $G^{x}_{x}$ is torsion free and $\{x\}\notin \mathcal{X}(x)$ (see Remark \ref{rmk:torfree}). Moreover, $\{x\}\notin \mathcal{X}(x)$ if and only if $x$ is \emph{extremely dangerous} in the sense of \cite{NS23} (Corollary \ref{cor:edangptchar}). Thus, for groupoids with torsion free isotropy groups, $J=\{0\}$ if and only if $G$ has no extremely dangerous points, recovering the characterization in \cite{NS23}. Our simplification of the characterization in \cite{BGHL25} follows from Theorem H and I, as discussed later.

Now, we discuss the proof of Theorem A and B. By \cite{CN22}, there are c.p.c. maps $\eta_{x}:C^{*}_{r}(G)\to C^{*}_{e}(G^{x}_{x})$ obtained by restriction, where $``e"$ denotes a potentially exotic $C^{*}$-norm. This family of maps is faithful in the sense that a positive element $a\in C^{*}_{r}(G)$ is zero if and only if $\eta_{x}(a) = 0$ for all $x\in G^{0}$. Hence, $J = \{0\}$  if and only if $J_{x}:= \eta_{x}(J) = 0$ for all $x\in G^{0}$. Our characterization follows from 
the calculation of these isotropy fibres $J_{x}$, which are ideals by \cite{CN24}. To do this, we determine the commuting square 

$$\begin{tikzcd}
C^{*}_{r}(G) \arrow[r] \arrow[d, "\eta_{x}"] & {C^{*}_{r,ess}(G)} \arrow[d, dashed] \\
C^{*}_{e}(G^{x}_{x}) \arrow[r, dashed]       & ?                                 
\end{tikzcd}$$ where $C^{*}_{r,ess}(G):= C^{*}_{r}(G)/J$, and show the kernel of the top map, the singular ideal, surjects onto the kernel of the bottom.

As the quotient $C^{*}_{r,ess}(G):= C^{*}_{r}(G)/J$ is not spatially implemented, it is not clear what ? should be. However, following the philosophy developed in \cite{BGHL25} that a non-Hausdorff groupoid C*-algebra can be understood via its embedding into the C*-algebra of its \emph{Hausdorff cover} $\tilde{G}$, we show the above diagram can be embedded into the commuting square
$$
\begin{tikzcd}
C^{*}_{r}(\tilde{G}) \arrow[r] \arrow[d] & C^{*}_{r}(\tilde{G}_{ess}) \arrow[d]   \\
C^{*}_{e}(\tilde{G}|_{\pi^{-1}(x)}) \arrow[r]                & C^{*}_{e}(\tilde{G}|_{\pi^{-1}_{ess}(x)})
\end{tikzcd}$$ where all maps are restrictions, $\pi:\tilde{G}^{0}\to G^{0}$ is the natural projection $X\mapsto X\cap G^{0}$ (defined above \cite[Lemma~3.5]{BGHL25}) and $\pi_{ess} = \pi|_{\tilde{G}^{0}_{ess}}$. Again, we use the notation $``e"$ for a possibly exotic C*-norm. 

The groupoids on the bottom line have a very specific form; each element in $\pi^{-1}(x)$ is a subgroup of $G^{x}_{x}$, and $\tilde{G}|_{\pi^{-1}(x)}$ is the groupoid $G^{x}_{x}\cdot\pi^{-1}(x)$ of cosets of subgroups in $\pi^{-1}(x)$ (Proposition \ref{prp:rest=cosetgpd}). Moreover, $\pi^{-1}_{ess}(x) = \mathcal{X}(x)$. The embedding $C^{*}_{e}(G^{x}_{x})\subseteq C^{*}_{e}(\tilde{G}|_{\pi^{-1}(x)})$ is induced from $\mathbb{C}[G^{x}_{x}]\subseteq C_{c}(G^{x}_{x}\cdot\pi^{-1}(x))$ via $a\in\mathbb{C}[G]\mapsto a(gX) = \sum_{h\in gX}a(h)$, $g\in G^{x}_{x}$, $X\in \pi^{-1}(x)$. The embedding then determines $?$ and the maps into it.

Now, there are two problems remaining. It is not immediately clear that the restriction $C_{c}(\tilde{G})\to C_{c}(\tilde{G}|_{\pi^{-1}(x)})$ extends to C*-completions. An arbitrary restriction would not necessarily, as the reduction might not be \'etale! Moreover, we still need to know the kernel of the top map in the first diagram surjects onto the kernel of the bottom, which is not true for a general commuting square. This fact holds if all maps were *-homomorphisms and the left vertical map had an approximate unit for its kernel that maps to an approximate unit for the kernel of the right vertical map.

What we found in the process of solving these two problems is a class of c.p.c. maps between (pre-)C*-algebras that behave like *-homomorphisms in many respects, and are abundant amongst \'etale groupoid C*-algebras.

A linear map $\eta:\mathcal{A}\to \mathcal{B}$ between pre-C*-algebras is \emph{compressible to a *-homomorphism} $\eta:\mathcal{C}\to \mathcal{B}$ if $\mathcal{C}\subseteq \mathcal{A}$ is a *-sub-algebra and for every $a\in\mathcal{A}$ and $\epsilon > 0$ there is $c\in\mathcal{C}$ and $\phi\in \tilde{\mathcal{C}}$ (the unitization of $\mathcal{C}$) such that $\|\phi\|\leq 1$, $\eta(\phi) = 1$, $\eta(a) = \eta(c)$ and $\|\phi^{*}a\phi - c\|\leq \epsilon $.

We prove that compressible maps behave like *-homomorphisms; they are c.p.c maps (Corollary \ref{cor:cpc}), send ideals to ideals in the image (Corollary \ref{cor:imidealisideal}), and satisfy the important and useful norm equation below.

\begin{thmC}[\ref{thm:normeqn}]

If $\eta:\mathcal{A}\to\mathcal{B}$ is compressible to a bounded *-homomorphism $\eta|_{\mathcal{C}}$, then for any approximate unit $(u_{\lambda})$ for the kernel of the completion $\eta: C\to B$ and $a\in\mathcal{A}$, we have

\begin{equation*}
    \|\eta(a)\| = \lim_{\lambda}\|(1-u_{\lambda})a(1-u_{\lambda})\|.
\end{equation*}

\end{thmC}

This result was inspired by the paper \cite{CN22}, where a similar norm equation is proven for the restriction map $C_{r}^{*}(G)\to C_{e}^{*}(G^{x}_{x})$ (see \cite[Equation~2.3]{CN22}). In fact, our definition of compressibility arose from a desire to understand \cite{CN22} using a representation-theoretic-free approach.

Most importantly for our discussion, in a commuting square, the kernel of the top map surjects to the kernel of the bottom with the same hypothesis as for *-homomorphisms.

\begin{thmD}[\ref{thm:kernelsurjects}]
    Suppose 

    $$
\begin{tikzcd}
A_{1} \arrow[r, "i"] \arrow[d, "\eta_{1}"] & A_{2} \arrow[d, "\eta_{2}"] \\
B_{1} \arrow[r, "j"]                       & B_{2}                      
\end{tikzcd}$$ is a commutative diagram of C*-algebras, where $i$ and $j$ are *-homomorphisms and $\eta_{1},\eta_{2}$ are compressible to *-homomorphisms $\eta_{1}:C_{1}\to B_{1}$, $\eta_{2}:C_{2}\to B_{2}$ with $\eta_{1}$ surjective. Assume that there is an approximate unit $(u_{\lambda})$ for $\ker(\eta_{1}:C_{1}\to B_{1})$ such that $(i(u_{\lambda}))$ is an approximate unit for $\ker(\eta_{2}:C_{2}\to B_{2})$. 

Then, $\eta_{1}(\text{ker}(i)) = \text{ker}(j)$. Additionally, if $i(C_{1})\subseteq C_{2}$, then $\eta_{1}(\text{ker}(i)\cap C_{1}) = \ker(j)$.
\end{thmD}

Moreover, compressibility is a natural notion. It is preserved by taking completions of pre-C*-algebras (Corollary \ref{cor:completion}), matrix amplifications (Corollary \ref{cor:cpc}), and quotients (Corollary \ref{cor:comppasstoquotient}).

Now, to determine the isotropy fibres, it suffices to show the restriction maps on the groupoid pre-C*-algebras $\mathscr{C}_{c}(G)$ and $C_{c}(\tilde{G})$ are compressible to *-homomorphisms. In fact, we show much more.

Say for an \'etale groupoid $G$, a closed set $X\subseteq G^{0}$ is \emph{locally invariant} (Definition \ref{dfn:locinvar}) if for every $g\in G$ with $r(g), s(g)\in X$, there is an open neighbourhood $U$ of $g$ such that, for all $\tilde{g}\in U$, we have $r(\tilde{g})\in X$ if and only if $s(\tilde{g})\in X$ (Definition \ref{dfn:locinvar}). 
It is easy to see that every finite subset $F\subseteq G^{0}$ is locally invariant and if $X$ is locally invariant, then the pre-images $\pi^{-1}(X)$ and $\pi_{ess}^{-1}(X)$ are locally invariant inside the Hausdorff and essential Hausdorff cover (Proposition \ref{prp:locinvar&cover}).

\begin{thmE}[\ref{thm:locinvar&compressdense}, \ref{thm:locinvar&compress} and \ref{cor:quotientnorm}]
    Let $G$ be an \'etale groupoid, $X\subseteq G^{0}$ locally invariant and $\rho$ a pre-C*-norm for $\mathscr{C}_{c}(G)$. Then, the quotient norm $\rho(X)$ on $\mathscr{C}_{c}(G|_{X})$ induced from $\mathscr{C}_{c}(G)\to \mathscr{C}_{c}(G|_{X})$ is a C*-norm. The restriction map $C^{*}_{\rho}(G)\to C^{*}_{\rho(X)}(G|_{X})$ is compressible to the *-homomorphism $C^{*}_{\rho}(H)\to C^{*}_{\rho(X)}(G|_{X})$, where $H$ is any open subgroupoid such that $H^{0} = G^{0}$, $H|_{X} = G|_{X}$ and $X$ is $H$-invariant (which exists).
\end{thmE}

We also show the norm $\rho(X)$ as above depends only on the ``germ'' of $G$ about $G|_{X}$ (Proposition \ref{prp:dependsongerm}).

This concludes our discussion on our determination of the isotropy fibres. Theorem A follows almost immediately, and it is a further application of our results on compressible *-homomorphisms to obtain the vanishing characterization in Theorem B, but we leave that for the proofs (see Section \ref{ss:charvsing}).

Since the restrictions $\mathscr{C}_{c}(G)\to \mathscr{C}_{c}(G|_{F})$ to the \emph{discrete} groupoids $G|_{F}$,  where $F\subseteq G^{0}$ is finite, seperate points and behave like *-homomorphisms, it is possible to study \'etale groupoids as ``residually discrete groupoids'', in the same way residually finite C*-algebras can be studied using their finite dimensional representations. Our characterization of when $J=0$ in terms of a property of the isotropy groups is a specific instance of this method, which we advance further in upcoming joint work with Julian Gonzales \cite{GH25}.

In \cite[Theorem~4.2]{BGHL25}, it was characterized when $J\cap\mathscr{C}_{c}(G) = \{0\}$ in terms of a groupoid property of $G$. We provide an alternate characterization which is in principal easier to check as it is more algebraic.

\begin{thmF}[\ref{cor:valgidealcharspan}]
    Let $G$ be an \'etale groupoid. Then, $J\cap\mathscr{C}_{c}(G) = \{0\}$ if and only if for every $x\in G^{0}$ the set of linear equations
    $$\sum_{h\in gX}a_{h} = 0,\text{ }g\in G^{x}_{x},\text{ }X\in\mathcal{X}(x)$$
    has no non-zero finitely supported integer solution $(a_{h})_{h\in G^{x}_{x}}$.
\end{thmF}

For clarity, finitely supported means the integers $a_{h} = 0$ for all but finitely many $h\in G^{x}_{x}$. We get several other characterizations of $J\cap\mathscr{C}_{c}(G) = \{0\}$ by applying Lemma \ref{lem:alggroupidealchar} to Theorem \ref{thm:valgidealfibres}, but we do not present these. The proof (of the contrapositive) of the ``if'' direction provides a new way to construct elements in the singular ideal.

Following a similar proof to Corollary \ref{cor:valgidealcharspan} and applying \cite[Theorem~4.2]{BGHL25}, it is easy to see the \emph{Steinberg algebra} of a ring $R$ and ample \'etale groupoid $G$ (see \cite{SS20}) has zero singular ideal if and only if the above equations in Theorem F have no non-zero solution in the group ring $\mathbb{Z}_{t}[G^{x}_{x}]$ for every $x\in G^{0}$ and order $t\in\mathbb{N}\cup\{0\}$ of some non-zero element in $R$, where $\mathbb{Z}_{t} = \mathbb{Z}/t\mathbb{Z}$ for $t > 0$ and $\mathbb{Z}_{0} = \mathbb{Z}$.

The algebraic nature of the characterization of $J\cap \mathscr{C}_{c}(G) = \{0\}$ means that it is easier to check in practice than our characterization of when $J = \{0\}$, which motivates the following.

\begin{qtn}
\label{qtn1}
    Let $G$ be an \'etale groupoid. Does $J\cap\mathscr{C}_{c}(G) = \{0\}$ imply $J = \{0\}$?
\end{qtn}

This question was shown in \cite[Theorem~4.7]{BGHL25} to have a positive answer for \'etale groupoids $G$ with $|\overline{G^{0}}|_{x}| < \infty$, for all $x\in G^{0}$. This is equivalent to $\pi^{-1}(x)$ consisting of a finite set of finite subgroups for all $x\in G^{0}$.

We introduce a related question for discrete groups. Let $\Gamma$ be a discrete group and $\mathcal{X}$ a set of subgroups invariant under conjugation and closed in $\{0,1\}^{\Gamma}$. Let $J_{\Gamma,\mathcal{X}}= \text{ker}(\lambda_{\Gamma/\mathcal{X}})$ inside the group $C^{*}$-algebra of $\Gamma$ with norm determined by the left regular representation and the quasi-regular representations associated to $X\in\mathcal{X}$ (Definition \ref{dfn:gpideal}). The question is the following.

\begin{qtn}
\label{qtn2}
    Let $\Gamma$ be a discrete group and $\mathcal{X}$ a closed and invariant set of subgroups. Does $J_{\Gamma,\mathcal{X}} \neq \{0\}$ imply  $J_{\Gamma,\mathcal{X}}\cap\mathbb{C}[\Gamma] \neq \{0\}$?
\end{qtn}

See Lemma \ref{lem:alggroupidealchar} for different characterizations of this question. We show the above two questions always have the same answer.

\begin{thmG}[\ref{cor:equalqs}]
    A positive answer to Question \ref{qtn1} is equivalent to a positive answer to Question \ref{qtn2}.
\end{thmG}
To prove that a positive answer to Question \ref{qtn1} implies a positive answer to Question \ref{qtn2} we build, for each pair of discrete group $\Gamma$ and closed invariant set of subgroups $\mathcal{X}$, a non-Hausdorff groupoid $G_{(\Gamma,\mathcal{X})}$ with exactly one non-Hausdorff point $\infty$ in the unit space, such that $J_{\infty} = J_{\Gamma,\mathcal{X}}$. One can interpret this construction as a generalization of Willet's HLS groupoid construction \cite{W15}  to the non-Hausdorff case (see Section \ref{s:nonHDconstruct}).

In the absence of a positive answer to Question \ref{qtn2}, we can speak of the \emph{class} $\mathcal{I}$ of discrete groups $\Gamma$ and closed invariant set of subgroups $\mathcal{X}$ such that the question holds true for $(\Gamma,\mathcal{X})$. So far we have shown that this class is quite large.

\begin{thmH}[Section \ref{s:groupswint}]
\label{thmH}
    $(\Gamma,\mathcal{X})\in \mathcal{I}$ if any of the following conditions hold.
    \begin{itemize}
        \item $\Gamma$ is a direct limit of virtually torsion free solvable groups,
        \item $\mathcal{X}$ is finite,
        \item every $X\in\mathcal{X}$ is finite,
        \item every $X\in\mathcal{X}$ is normal and torsion free.
    \end{itemize}
\end{thmH}
To prove this, we show the class $\mathcal{I}$ satisfies a variety of permanence results. See Section \ref{s:groupswint} for more details. We can apply Theorem H to the below result to show specific classes of \'etale groupoids satisfy an algebraic characterization for $J = \{0\}$.

\begin{thmI}[\ref{thm:valgideal=videal} and \ref{thm:linearchar}]
\label{thmI}
    Let $G$ be an \'etale groupoid such that $(G^{x}_{x}, \mathcal{X}(x))\in\mathcal{I}$ for every $x\in G^{0}$. Then, $J = \{0\}$ if and only if $J\cap\mathscr{C}_{c}(G) = \{0\}$, if and only if for every $x\in G^{0}$, the set of linear equations

    $$\sum_{h\in gX}a_{h} = 0,\text{ }g\in G^{x}_{x},\text{ }X\in\mathcal{X}(x)$$ has no non-zero finitely supported integer solution $(a_{h})_{h\in G^{x}_{x}}$.
\end{thmI}

By Theorem H, the hypothesis for Theorem I is satisfied by the class of groupoids with the finiteness condition in \cite[Theorem~4.7]{BGHL25} and our characterization improves that in \cite{BGHL25} as it is more algebraic.

When a group $\Gamma$ satisfies $(\Gamma, \mathcal{X})\in\mathcal{I}$ for every closed and invariant set of subgroups $\mathcal{X}$, we say $\Gamma$ has \emph{Property $I$}, or the \emph{Intersection Property} (Definition \ref{dfn:I&AI}). The above hypothesis is therefore satisfied whenever a groupoid's isotropy consists of Property $I$ groups. By Theorem H, we know Property $I$ is satisfied for every group with polynomial growth and every matrix group over characteristic zero fields (see Theorem \ref{thm:propertyAI&I}).

Similarly, if $\{e\}\notin\mathcal{X}$ implies $J_{\Gamma,\mathcal{X}}\cap\mathbb{C}[\Gamma]\neq \{0\}$, we say $\Gamma$ has \emph{Property $AI$}, or the \emph{Automatic Intersection Property}.

\begin{thmJ}[\ref{thm:propertyAI&I}]
Every direct limit of torsion free virtually solvable groups has Property $AI$. In particular, every torsion free group with polynomial growth and amenable torsion free matrix group over a characteristic zero field satisfies Property $AI$.
\end{thmJ}

\begin{thmK}[\ref{thm:abeliancharAI}]
    A countable abelian group $\Gamma$ satisfies Property $AI$ if and only if for every prime $p$, there is at most one element $g\in \Gamma$ with cyclic order $p$.
\end{thmK}

Therefore, the class of Property $AI$ groups is quite large. The next result shows a groupoid containing an extremely dangerous point with a Property $AI$ group always has a non-zero ``algebraic'' singular ideal.

\begin{thmH}[\ref{thm:valgideal=videal}]
    Let $G$ be an \'etale groupoid such that $G^{x}_{x}$ has Property $AI$ and $\{x\}\notin\mathcal{X}(x)$ for some $x\in G^{0}$. Then, $J\cap\mathscr{C}_{c}(G)\neq \{0\}$.
\end{thmH}

\section*{Acknowledgements}

Hume would like to thank the Isaac Newton Institute for Mathematical
Sciences for the support and hospitality during the programme ‘Topological
groupoids and their C*-algebras’ when work on this paper was undertaken.
This work was supported by EPSRC Grant Number EP/V521929/1 and by NSERC Discovery Grant RGPIN-2021-03834. Hume would also like to express thanks to Charles Starling and Julian Gonzales for helpful discussions about the ideas in this paper.

\section{Background}
\label{s:bckgrnd}
\subsection{\'Etale groupoids}
A \emph{groupoid} is a set $G$ equipped with the structure of the invertible morphisms of a category. For $g\in G$, its source object $s(g)$ and range object $r(g)$ can (and will) be identified with the identity morphisms $\text{id}_{s(g)} = g^{-1}g$ and $\text{id}_{r(g)} = gg^{-1}$, respectively. Then, composition becomes a map $G{}_{s}\times_{r}G\to G$, $(g,h)\mapsto gh$ which we call the \emph{product map}. The inverse of $g\in G$ is denoted as usual by $g^{-1}$. For $X, Y$ subsets of the object set, we let $G_{X} = s^{-1}(X)$, $G^{Y} = r^{-1}(Y)$ and $G^{Y}_{X} = G_{X}\cap G^{Y}$. If $Y =X$ we will sometimes denote $G^{X}_{X} = G|_{X}$, and this is also a groupoid (with structure inherited from $G$) known as the \emph{reduction of $G$ to $X$}. A subset $X\subseteq G^{0}$ is called \emph{invariant} if $G|_{X} = G^{X}$ (or equivalently $ G|_{X} = G_{X}$).

A \emph{topological groupoid} $G$ is equipped with a topology such that the product $G{}_{s}\times_{r}G\to G$ and inverse $^{-1}:G\to G$ are continuous. We will call the set of objects $G^{0}\subseteq G$, equipped with the relative topology, the \emph{unit space} and we will always assume $G^{0}$ is a locally compact Hausdorff space.

An \emph{\'etale groupoid} is a topological groupoid such that the range map $r:G\to G^{0}$ (or equivalently the source) is a local homeomorphism. By our assumption that $G^{0}$ is a locally compact Hausdorff space, $G$ must be locally compact and at least locally Hausdorff. We say $U\subseteq G$ is a \emph{bisection} if $r|_{U}$ and $s|_{U}$ are injections. Let $\mathcal{B}$ denote the collection of open bisections. Note the \'etale assumption implies the open bisections form a basis for the topology on $G$, and that each $U\in\mathcal{B}$ is a locally compact \emph{Hausdorff} space. Moreover, $UV\in\mathcal{B}$ and $U^{-1}\in \mathcal{B}$ if $U,V\in\mathcal{B}$.

\textbf{Standing assumption: we assume there are open bisections $\{U_{n}\}_{n\in\mathbb{N}}$ such that $\bigcup_{n}U_{n} = G$. In words, $G$ is covered by countably many open bisections.}

For a locally compact Hausdorff space $X$, denote by $C_{c}(X)$ the continuous and compactly supported functions into $\mathbb{C}$.

For $U\in\mathcal{B}$ and $f\in C_{c}(U)$, we view $f$ as a function on $G$ by extending $f$ to be zero off $U$. Note that, without assuming $G$ is Hausdorff, $f:G\to \mathbb{C}$ is not necessarily continuous. We define $\mathscr{C}_{c}(G) := \text{span}\{f:G\to \mathbb{C}: f\in C_{c}(U), U\in\mathcal{B}\}$. 

For $f_{1},f_{2}\in \mathscr{C}_{c}(G)$ we define their product $f_{1}*f_{2}:G\to \mathbb{C}$ and involution $f_{1}^{*}$ as 

\begin{equation*}
    f_{1}*f_{2}(g) = \sum_{h\in G_{s(g)}}f_{1}(gh^{-1})f_{2}(h) \text{ and }f_{1}^{*}(g) = \overline{f_{1}}(g^{-1}), \text{ for all }g\in G.
\end{equation*} Since $C_{c}(U)*C_{c}(V)\subseteq C_{c}(UV)$ and $C_{c}(U)^{*} = C_{c}(U^{-1})$, by linearity $\mathscr{C}_{c}(G)$ is closed under the product and involution. Moreover, it is straightforward to see (using the axioms of a groupoid) that these operations turn $\mathscr{C}_{c}(G)$ into a $*$-algebra.

Since $s:G\to G^{0}$ is a local homeomorphism, the set $G_{x}$, for any $x\in G^{0}$, is discrete in $G$. We view the Hilbert space $\ell^{2}(G_{x})$ as functions on $G$ via their extension to zero off $G_{x}$. Then, for each $f\in \mathscr{C}_{c}(G)$ and $\psi\in \ell^{2}(G_x)$ define $\lambda_{x}(f)(\psi) = f*\psi$. Then, $f*\psi\in \ell^{2}(G_{x})$ and moreover if $f\in C_{c}(U)$ for $U\in \mathcal{B}$, we have $\|f*\psi\|_{2}\leq \|f\|_{\infty}\|\psi\|_{2}$. Hence, every $f\in\mathscr{C}_{c}(G)$ defines a bounded operator $\lambda_{x}(f)$ and it is easy to see $\lambda_{x}:\mathscr{C}_{c}(G)\to B(\ell^{2}(G_{x}))$ is a *-algebra homomorphism. Define $\|f\|_{x} := \|\lambda_{x}(f)\|$ and $\|f\|_{r}:=\text{sup}_{x\in G^{0}}\|f\|_{x}$, which is a C*-norm for $\mathscr{C}_{c}(G)$ called the \emph{reduced norm}. The C*-completion of $\mathscr{C}_{c}(G)$ under $\|\cdot\|_{r}$ is called the \emph{reduced groupoid C*-algebra of $G$} and is denoted $C^{*}_{r}(G)$. Every element $f\in C^{*}_{r}(G)$ can be viewed as a function $f:G\to\mathbb{C}$ via the assignment $g\in G\mapsto f(g):= \langle \lambda_{s(g)}(f)*\delta_{s(g)}, \delta_{g}\rangle$. When $G$ is Hausdorff, every function in $C^{*}_{r}(G)$ is continuous and vanishes at $\infty$.

Following \cite{KM21}, we say a function $f\in C^{*}_{r}(G)$ is \emph{singular} if $s(\{g\in G: f(g)\neq 0\})$ is meagre (This is equivalent to a variety of other definitions, see \cite[Proposition~7.18]{KM21}). In \cite[Lemma~4.1]{BGHL25} it is shown to be equivalent to density of $f^{-1}(0)$ in $G$. The collection of all singular functions forms a closed two-sided ideal which we call the \emph{singular ideal}.

\subsection{Pre-C*-algebras}

We will have cause to consider other norm completions of $\mathscr{C}_{c}(G)$ in this paper. A \emph{pre-C*-algebra} $\mathcal{A}$ is a normed *-algebra satisfying all the axioms of a C*-algebra except norm-completeness \cite{Lance}. Of course, $\mathcal{A}$ sits inside its completion $A$ (a C*-algebra) as a dense *-sub-algebra. 

If $\mathcal{A}$ is a pre-C*-algebra, we will denote by $\tilde{\mathcal{A}} = \mathcal{A}\oplus\mathbb{C}$ its unitization, which is also a pre-C*-algebra. Moreover, we will extend a linear map $\eta:\mathcal{A}\to \mathcal{B}$ between pre-C*-algebras to the unitizations by setting $\eta(a + 1) = \eta(a) + 1$, for $a\in\mathcal{A}$. If $\mathcal{C}$ is a *-sub-algebra of $\mathcal{A}$, then we will view $\tilde{\mathcal{C}}$ as a *-sub-algebra of $\tilde{\mathcal{A}}$ by identifying it with its image under the unitization of the inclusion $\mathcal{C}\subseteq \mathcal{A}$.

If $\|\cdot\|_{\rho}$ is a pre-C*-norm for $\mathscr{C}_{c}(G)$, we will denote by $C^{*}_{\rho}(G)$ its C*-completion. If $\mathcal{A}$, $\mathcal{B}$, $\mathcal{C}$ are pre-C*-algebras, we will denote by $A$, $B$, $C$ their norm completions, respectively.

\subsection{Hausdorff cover of non-Hausdorff \'etale groupoid}
\label{ss:cover}

Given a locally compact (but not necessarily Hausdorff) space $\mathcal{X}$, denote its set of closed subsets by $\mathcal{C}(\mathcal{X})$. In \cite{F62}, Fell equips $\mathcal{C}(\mathcal{X})$ with a topology (the \emph{Fell topology}) whose basic open sets are of the form $\mathcal{U}(C,\mathcal{F}) = \{X\in\mathcal{C}(\mathcal{X}): X\cap C = \emptyset\text{ and } X\cap U\neq 0\text{ for all }U\in\mathcal{F}\}$, where $C$ is compact and $\mathcal{F}$ is a finite collection of open sets, and shows this is a compact Hausdorff topology. Moreover, if points are closed in $\mathcal{X}$, then each $x\in \mathcal{X}$ embeds into $\mathcal{C}(\mathcal{X})$ as the singleton $\{x\}$. Denote this map $\iota:\mathcal{X}\to\mathcal{C}(\mathcal{X})$. The \emph{Fell Hausdorffication} $\mathcal{H}(\mathcal{X})$ (see \cite{F62}) is defined as the closure of $\iota(\mathcal{X})$ in $\mathcal{C}(\mathcal{X})\setminus \{\emptyset\}$.

Let's describe the topology on $\mathcal{X}$ in terms of nets in a special case. First, for a net $(X_{\lambda})\subseteq \mathcal{C}(\mathcal{X})\setminus \{\emptyset\}$, define its \emph{limit set} as $$\lim(X_{\lambda}) = \{x\in X: \lim_{\lambda} x_{\lambda} = x\text{ for some net }x_{\lambda}\in X_{\lambda}\}$$ and its \emph{accumulation set}
$$\text{Acc}(X_{\lambda}) = \{x\in X: \lim_{\mu} x_{\lambda_{\mu}} = x\text{ for some subnet } x_{\lambda_{\mu}}\in X_{\lambda_{\mu}}\}.$$

Note that $\text{Acc}(X_{\lambda})$ is always a closed set and (by the axiom of choice) $x\in \lim(X_{\lambda})$ if and only if there is $\lambda_{0}$ and a net $x_{\lambda}\in X_{\lambda}$, for $\lambda\geq \lambda_{0}$, such that $\lim_{\lambda}x_{\lambda} = x$.

\begin{prp}
\label{prp:fell}
    Suppose $\mathcal{X}$ is a locally compact space that is locally Hausdorff. Then, a net $(X_{\lambda})\subseteq \mathcal{H}(\mathcal{X})$ converges to $X$ if and only if $\lim(X_{\lambda}) = \text{Acc}(X_{\lambda}) = X$.
\end{prp}
\begin{proof}
    Suppose $(X_{\lambda})$ is a net such that $\lim(X_{\lambda}) = \text{Acc}(X_{\lambda}) = X$, and let $C\subseteq X$ be compact and $\mathcal{F}$ a finite family of open sets in $X$ such that $X\in \mathcal{U}(C,\mathcal{F})$. If there is a subnet $(X_{\lambda_{\mu}})$ such that $(X_{\lambda_{\mu}})\cap C\neq \emptyset$ for all $\mu$, then by compactness, there is $c\in \text{Acc}(X_{\lambda}) = X$, a contradiction. Similarly, if there is $U\in\mathcal{F}$ and a subnet $(X_{\lambda_{\mu}})$ such that $X_{\lambda_{\mu}}\cap U=\emptyset$ for all $\mu$, then $U\cap X = U\cap \lim(X_{\lambda})\subseteq U\cap\text{Acc}(X_{\lambda_{\mu}}) = \emptyset$, a contradiction. Therefore, a net satisfying $\lim(X_{\lambda}) = \text{Acc}(X_{\lambda})$ converges to $X = \text{Acc}(X_{\lambda})$ in the Fell topology.

    To prove the converse, let $(X_{\lambda})\subseteq \mathcal{H}(X)$ be a net converging to $X$ in the Fell topology. If $U$ is an open and Hausdorff subset of $X$, then we claim that $|Y\cap U|\leq 1$ for any $Y\in \mathcal{H}(X)$. To prove, this, suppose that there is $Y$ such that $y_{1},y_{2}\in Y\cap U$ with $y_{1}\neq y_{2}$. Since $Y$ is the limit of a net $(\iota(y_{\lambda}))$, it follows that for any open neighbourhood $U_{1}\subseteq U$ of $y_{1}$ and $U_{2}\subseteq U$ of $y_{2}$, we have $y_{\lambda}\in U_{1}\cap U_{2}\subseteq U$ eventually. But this would imply $U$ is not Hausdorff, a contradiction. This proves the claim.

    For $x\in X$, let $U_{x}$ be an open and Hausdorff neighbourhood of $x$. Then, by the definition of the Fell topology, there is $\lambda_{0}$ such that $X_{\lambda}\cap U_{x}\neq \emptyset$ for all $\lambda\geq \lambda_{0}$. By the above claim, we have $(X_{\lambda}\cap U)_{\lambda\geq \lambda_{0}} = (\{x_{\lambda}\})_{\lambda\geq \lambda_{0}}$ for some $x_{\lambda}\in X_{\lambda}$. A further application of the Fell topology implies $(x_{\lambda})_{\lambda\geq \lambda_{0}}$ converges to $x$. Hence, $X\subseteq \lim(X_{\lambda})$.

    We now show $\text{Acc}(X_{\lambda})\subseteq X$. Suppose for the sake of contradiction that there is $x\in \mathcal{X}\setminus X$ such that $x = \lim x_{\lambda_{\mu}}$ for some subnet $x_{\lambda_{\mu}}\in X_{\lambda_{\mu}}$. Since $\mathcal{X}$ is locally Hausdorff and $X$ is closed, there is a Hausdorff neighbourhood $U$ of $x$ such that $U\cap X = \emptyset$. Since $U$ is locally compact and Hausdorff, there is a an open subset $V\subseteq U$ whose closure $C = \overline{V}^{U}$ relative to $U$ is compact. By construction, we have $X_{\lambda_{\mu}}\cap C\neq\emptyset$ eventually but $X\cap C = \emptyset$, which contradicts the convergence of $(X_{\lambda})$ in the Fell topology. Therefore, $\text{Acc}(X_{\lambda})\subseteq X$. We have proven $\text{Acc}(X_{\lambda})\subseteq X\subseteq \lim(X_{\lambda})$ and hence $\lim(X_{\lambda}) = \text{Acc}(X_{\lambda}) = X$.
\end{proof}
Note that this proof also shows every $X\in\mathcal{H}(\mathcal{X})$ is a discrete subset of $\mathcal{X}$ when the space is locally compact and locally Hausdorff.

\begin{cor}
\label{cor:compactness}
    Let $\mathcal{X}$ be a locally compact and locally Hausdorff space. If $(X_{\lambda})\subseteq \mathcal{H}(\mathcal{X})$ is a net with $x\in \text{Acc}(X_{\lambda})$, then there is subnet $(X_{\lambda_{\mu}})$ with $x\in  \text{Acc}(X_{\lambda_{\mu}}) = \lim(X_{\lambda_{\mu}})$.
\end{cor}

\begin{proof}
    Let $x_{\lambda_{\gamma}}\in X_{\lambda_{\gamma}}$ be a subnet with $\lim x_{\lambda_{\gamma}} = x$. By compactness of $\mathcal{C}(\mathcal{X})$, there is a subnet $(X_{\lambda_{\mu}})$ of $(X_{\lambda_{\gamma}})$ (and hence of $(X_{\lambda})$) converging to $X$ in $\mathcal{C}(\mathcal{X})$. By the definition of the Fell topology, we have $x\in X$, and hence $X\in\mathcal{C}(\mathcal{X})\setminus\{\emptyset\}$. Since $\mathcal{H}(\mathcal{X})$ is closed in $\mathcal{C}(X)\setminus\{\emptyset\}$, it follows that $X\in\mathcal{H}(\mathcal{X})$. By Proposition \ref{prp:fell}, we have $x\in X = \text{Acc}(X_{\lambda_{\mu}}) = \lim(X_{\lambda_{\mu}})$
\end{proof}

For an \'etale groupoid $G$, its \emph{Hausdorff cover} $\tilde{G}$ is defined to be $\mathcal{H}(G)$ equipped with the subspace topology arising from $\mathcal{C}(G)\setminus \{\emptyset\}$. Since $G$ is locally Hausdorff, the topology of $\tilde{G}$ can be described as in Proposition \ref{prp:fell}. The Hausdorff cover was first introduced by Timmermann in \cite{T11} and was rediscovered and studied further in \cite{BGHL25} We now describe some of the structure of $\tilde{G}$. More details can be found in \cite{BGHL25}. 

Define $\tilde{G}^{0}$ to be the closure of $\iota(G^{0})$ in the Fell topology. Since $G^{0}$ is Hausdorff, for every $X\in \tilde{G}^{0}$, $X\cap G^{0}$ is a singleton, and we denote the point as $\pi(X)$. The map $\pi:\tilde{G}^{0}\to G^{0}$ is easily seen to be continuous. Moreover, $X$ is a subgroup of the group $G^{x}_{x}$, $x = \pi(X)$; if $(x_{\lambda})\subseteq G^{0}$ is a net converging to $X$ in the Fell topology, then $x_{\lambda}^{-1} = x_{\lambda} =  x_{\lambda}x_{\lambda}$ for all $\lambda$ and continuity of $r,s:G\to G^{0}$ imply $r(X) = s(X) = \{x\}$ and $X^{-1} = X = XX$.

 For any $X, Y\in \tilde{G}^{0}$ and $g\in G$ such that $\pi(X) = s(g)$ and $\pi(Y) = r(g)$, we have $gX, Yg\in \tilde{G}.$ To see this, choose nets $(x_{\lambda})\subseteq G^{0}$ and $(y_{\lambda'})\subseteq G^{0}$ that converge to $X$ and $Y$ in the Fell topology. Since $r:G\to G^{0}$ and $s:G\to G^{0}$ are open maps, there are nets $(g_{\lambda})_{\lambda\geq \lambda_{0}}$ and $(g_{\lambda'})_{\lambda'\geq \lambda'_{0}}$ converging to $g$ such that $s(g_{\lambda}) = x_{\lambda}$ for $\lambda\geq \lambda_{0}$ and $y_{\lambda'} = r(g_{\lambda'})$ for $\lambda'\geq \lambda'_{0}$. As we know $X$ and $Y$ are groups, it is easy to verify that $(g_{\lambda})$ converges to $gX$ and $(g_{\lambda'})$ converges to $Yg$ in the Fell topology.

Conversely, for $\underline{g}\in \tilde{G}$, there are $Y,X\in \tilde{G}^{0}$ such that, for any $g\in \underline{g}$, we have 

$$\underline{g} = gX = Y g.$$ $X$ and $Y$ are explicitly defined as the limits in the Fell topology of $(x_{\lambda} = g^{-1}_{\lambda}g_{\lambda})$ and $(y_{\lambda} = g_{\lambda}g^{-1}_{\lambda})$ for any net $(g_{\lambda})\subseteq G$ converging in the Fell topology to $\underline{g}$. To define $r,s:\tilde{G}\to \tilde{G}^{0}$, we set $r(\underline{g}) = Y$ and $s(\underline{g}) = X$.

Now, for $\underline{g},\underline{h}\in \tilde{G}$ such that $s(\underline{g})= X = r(\underline{h})$, the pointwise product $\underline{g}\cdot \underline{h}$ makes sense, and for $g\in\underline{g}$, $h\in\underline{h}$, we have

$$\underline{g}\cdot \underline{h} = gX\cdot Xh  = gXh = ghZ\in \tilde{G},$$ where $Z = s(\underline{h})$. Inversion $\underline{g}\mapsto \underline{g}^{-1}$ is also defined pointwise. With these operations, $\tilde{G}$ is an \'etale groupoid.

There is a $*$-homomorphism $\iota:\mathscr{C}_{c}(G)\to C_{c}(\tilde{G})$ defined for $f\in \mathscr{C}_{c}(G)$ as 

$$\iota(f)(\underline{g}) = \sum_{g\in \underline{g}}f(g), \text{ for all }\underline{g}\in \tilde{G}.$$ Since $\lambda_{\{s(g)\}}$ restricted to $\iota(\mathscr{C}_{c}(G))$ is unitarily equivalent to $\lambda_{s(g)}$ and $\|\cdot\|_{x}$ for a dense set $X\subseteq \tilde{G}^{0}$ determines $\|\cdot\|_{r}$ (because $\tilde{G}$ is Hausdorff), it follows that $\iota$ extends to a *-homomorphism $\iota:C^{*}_{r}(G)\to C^{*}_{r}(\tilde{G})$. Moreover, for all $f\in C^{*}_{r}(G)$ and $g\in G$, we have $\iota(f)(\{g\}) = f(g)$, so that $\iota$ is injective. See \cite[Lemma~3.8]{BGHL25} for more details.

An element $g\in G$ is \emph{Hausdorff} if for every $h\in G$, $h\neq g$, there is a neighbourhood $U$ of $g$ and $V$ of $h$ such that $U\cap V = \emptyset$. This is obviously equivalent to every net converging to $g$ has a distinct limit point. Therefore, continuity of the groupoid operations imply $g$ is Hausdorff if and only if $s(g)$ (or $r(g)$) is Hausdorff, showing that the Hausdorff points $C := \{x\in G^{0}:x \text{ is Hausdorff}\}$ are an invariant set, both in $G$ and in $\tilde{G}$ (embedded as singletons). We use the notation $C$ here since $x\in G^{0}$ is Hausdorff if and only if the embedding $\iota:G^{0}\to \tilde{G}^{0}$ is continuous at $x$. This follows from the fact that $\iota$ is a section for $\pi$ with dense image and Proposition \cite[Proposition~3.20]{BHL24}. Note that this is equivalent to $x\notin r(\partial G^{0})$, and since  $r:G\to G^{0}$ maps closed sets with empty interior to empty interior sets (by our standing assumption), $C$ is a dense set.

We denote the closure of $\iota(C)$ in the Fell topology by $\tilde{G}^{0}_{ess}$, which is again an invariant set of $\tilde{G}$ (this follows from the fact that $r,s:\tilde{G}\to \tilde{G}^{0}$ are open mappings). The reduction $\tilde{G}|_{\tilde{G}^{0}_{ess}} =:\tilde{G}_{ess}$ is called the \emph{essential Hausdorff cover}. If we let $\tilde{J} = \{f\in C^{*}_{r}(G): f|_{\tilde{G}_{ess}} = 0\}$, then we have a short exact sequence 

$$\begin{tikzcd}
0 \arrow[r] & \tilde{J} \arrow[r] & C^{*}_{r}(\tilde{G}) \arrow[r, "q"] & C^{*}_{r}(\tilde{G}_{ess}) \arrow[r] & 0,
\end{tikzcd}$$ where $q(f) = f|_{\tilde{G}_{ess}}$, for $f\in C^{*}_{r}(\tilde{G}_{ess})$ (\cite[Definition~4.14]{BGHL25}).

Note that $\iota^{-1}(\tilde{J}) = \{f\in C^{*}_{r}(G): f|_{(G|_{C})} = 0\}$. Since $C$ is dense in $G^{0}$, we have $\iota^{-1}(\tilde{J}) = J$ (\cite[Proposition~4.15]{BGHL25}) and $\pi_{ess} = \pi:G^{0}_{ess}\to G^{0}$ is surjective.

\section{Compressibility of maps to *-homomorphisms}
\label{s:compress}
In this section we introduce a new type of map between (pre-)C*-algebras which are seemingly abundant (at least for groupoid C*-algebras) and are useful in determining the isotropy fibres of the singular ideal. The results in this section (except Theorem \ref{thm:kernelsurjects}) are inspired by those in \cite{CN22} and \cite{CN24} for the restriction of groupoid C*-algebras to isotropy group C*-algebras.

\begin{dfn}
\label{dfn:comp}
    Let $\eta:\mathcal{A}\to\mathcal{B}$ be a linear map between pre-C*-algebras and $\mathcal{C}$ a *-sub-algebra of $\mathcal{A}$ such that $\eta:\mathcal{C}\to\mathcal{B}$ is a *-homomorphism. We say $\eta$ is \emph{compressible to $\mathcal{C}$} if for every $a\in\mathcal{A}$ and $\epsilon > 0$, there is $\phi\in\widetilde{\mathcal{C}}$ and  $c\in\mathcal{C}$ such that $\|\phi\|\leq 1$, $\eta(\phi) = 1$, $\eta(a) = \eta(c)$, and $\|\phi^{*} a\phi  - c\|\leq \epsilon$.
\end{dfn}

We now show the norm of an element $a\in\mathcal{A}$ under a compressible map $\eta:\mathcal{A}\to\mathcal{B}$ satisfies a formula which is immediate in the case when $\eta$ is a bounded *-homomorphism. This result is inspired from \cite[Theorem~2.4]{CN22}.
\begin{thm}
\label{thm:normeqn}
If $\eta$ is compressible to $\mathcal{C}$ and $\eta|_{\mathcal{C}}$ is bounded, then for any approximate unit $(u_{\lambda})$ for the kernel of the completion $\eta: C\to B$ and $a\in\mathcal{A}$, we have

\begin{equation}
\label{eq:norm1}
    \|\eta(a)\| = \lim_{\lambda}\|(1-u_{\lambda})a(1-u_{\lambda})\|\text{ and }
\end{equation}

\end{thm}

\begin{proof}
Let $\epsilon > 0$. For $a\in \mathcal{A}$, let $\phi\in\tilde{\mathcal{C}}$ and $c\in C$ be such that $\|\phi\|\leq 1$, $\eta(\phi) = 1$, $\eta(a) = \eta(c)$, and $\|\phi^{*} a\phi  - c\|\leq \frac{\epsilon}{3}$. Since $\ker(\eta:C\to B) = \ker(\eta: \tilde{C}\to\tilde{B})$ and $\eta:\tilde{C}\to\tilde{B}$ is a *-homomorphism, there is $\lambda_{0}$ such that, for all $\lambda\geq \lambda_{0}$, we have $$\|(1-\phi)(1-u_{\lambda})\| \leq \frac{\epsilon}{6\|a\|}\text{ and }\big| \|(1-u_{\lambda})c(1-u_{\lambda}))\| - \|\eta(c)\|\big|\leq \frac{\epsilon}{3}.$$

From the first inequality, we have $\|(1-u_{\lambda})a(1-u_{\lambda}) - (1-u_{\lambda})\phi^{*} a \phi (1-u_{\lambda})\|\leq\frac{\epsilon}{3}$ for all $\lambda\geq \lambda_{0}$. The second inequality, along with $\eta(a) = \eta(c)$ and $\|\phi^{*}a\phi - c\|\leq \frac{\epsilon}{3}$ implies $\big| \|(1-u_{\lambda})\phi^{*}a\phi(1-u_{\lambda}))\| - \|\eta(a)\|\big|\leq \frac{2\epsilon}{3}$ for all $\lambda\geq \lambda_{0}$. Therefore,$$\big| \|(1-u_{\lambda})a(1-u_{\lambda}))\| - \|\eta(a)\|\big|\leq \epsilon\text{ for all }\lambda\geq \lambda_{0}.$$ As $\epsilon > 0$ and $a\in\mathcal{A}$ are arbitrary, this completes the proof of the equality in Equation \ref{eq:norm1}.
\end{proof}

We now embark on proving the many corollaries of these equations. The first shows a compressible map is c.p.c. and it is moreover compressible to its multiplicative domain.

For a linear map $\eta:\mathcal{A}\to\mathcal{B}$ between pre-C*-algebras $\mathcal{A}$, $\mathcal{B}$, its \emph{multiplicative domain} is $$\mathcal{M} = \{m\in \mathcal{A}: \eta(ma) = \eta(m)\eta(a)\text{ and } \eta(am) = \eta(a)\eta(m) \text{ }\forall a\in\mathcal{A}\}.$$
\begin{cor}
\label{cor:cpc}
    If $\eta$ is compressible to a bounded *-homomorphism $\eta:\mathcal{C}\to \mathcal{B}$, then $\eta$ is a completely positive and completely contracting map. Consequently, $\mathcal{C}$ is contained in the multiplicative domain $\mathcal{M}$ and $\eta$ is compressible to $\eta:\mathcal{M}\to\mathcal{B}$.
\end{cor}

\begin{proof}
    Let $(u_{\lambda})$ be an approximate unit for $\ker(\eta:C\to B)$. Then, by Theorem \ref{thm:normeqn}, for every $\epsilon > 0$, and $a\in\mathcal{A}$ there is $\lambda_{0}$ such that $\|\eta(a)\| - \epsilon \leq \|(1-u_{\lambda_{0}})a(1-u_{\lambda_{0}})\|\leq \|a\|$. Therefore, $\eta:\mathcal{A}\to\mathcal{B}$ is norm contracting.

    We first show that $\eta:\mathcal{A}\to\mathcal{B}$ is self-adjoint, i.e. $\overline{\eta}(a) := \eta(a^{*})^{*} = \eta(a)$ for every $a\in \mathcal{A}$. Since $\eta:\mathcal{C}\to \mathcal{B}$ is a *-homomorphism, it is self-adjoint. If $d\in\mathcal{C}$ and $\phi\in \tilde{\mathcal{C}}$ is such that $\|\phi\|\leq 1$, $\eta(\phi) = 1$, $\eta(a^{*}) = \eta(d)$ and $\|\phi^{*} a^{*}\phi - d\|\leq \epsilon$, then $d^{*}$ is such that $\overline{\eta}(a) = \eta(d^{*}) = \overline{\eta}(d^{*})$ and $\|\phi^{*} a\phi - d^{*}\|\leq \epsilon.$ Therefore, $\overline{\eta}$ is compressible to $\eta:\mathcal{C}\to\mathcal{B}$.

    Now, for $a\in\mathcal{A}$ self-adjoint, let $c\in\mathcal{C}$ be such that $\eta(a) = \eta(c)$. Then, $\overline{\eta}(a) = \eta(a)^{*} = \eta(c^{*})$. By Theorem \ref{thm:normeqn}, for every $\epsilon > 0$ there is $\lambda_{0}$ such that for all $\lambda\geq \lambda_{0}$, we have $\|(1- u_{\lambda}) (a - c)(1-u_{\lambda})\|\leq \epsilon$ and $\|(1- u_{\lambda}) (a - c^{*})(1-u_{\lambda})\|\leq \frac{\epsilon}{2}$. Therefore, $\|(1- u_{\lambda} (c - c^{*})(1-u_{\lambda})\|\leq \frac{\epsilon}{2}$. Since $\eta:\mathcal{C}\to\mathcal{B}$ is a *-homomorphism and norm contraction, it follows that $\|\eta(c - c^{*})\|\leq \epsilon$. Since $\epsilon > 0$ was arbitrary, it follows that $\eta(a) = \eta(c) = \eta(c)^{*} = \overline{\eta}(a)$. Therefore, $\eta$ is self-adjoint.

    Now, we show $\eta$ is positive. Since $\eta$ is self-adjoint, for $a\in\mathcal{A}$ positive, there is a self-adjoint $c\in\mathcal{C}$ such that $\eta(a) = \eta(c)$. By Theorem \ref{thm:normeqn}, for every $\epsilon > 0$, there is $\lambda_{0}$ such that $\|(1-u_{\lambda})a(1- u_{\lambda}) - (1-u_{\lambda})c(1- u_{\lambda})\|\leq \epsilon$. Therefore, $(1-u_{\lambda})c(1- u_{\lambda})$ is a self-adjoint element with spectrum distance $\epsilon > 0$ away from $[0,\infty)$. Since, $\eta:\mathcal{C}\to\mathcal{B}$ is a bounded *-homomorphism, it follows that the spectrum of  $\eta(a) = \eta(c) = \eta((1-u_{\lambda})c(1- u_{\lambda}))$ is distance $\epsilon > 0$ from $[0,\infty)$. Since $\epsilon > 0$ was arbitrary, it follows that $\eta(a)$ is a self-adjoint element with spectrum contained in $[0,\infty)$ and is therefore positive.

    Let $\eta:A\to B$ denote the extension of $\eta$ to the respective completions, which is a norm contraction by what we have shown above. To show $\eta$ is completely positive and contracting, it suffice to show that the matrix amplification $\eta^{(n)}:M_{n}(A)\to M_{n}(B)$ is compressible to $\eta^{(n)}:M_{n}(C)\to M_{n}(A)$. Since $\eta:A\to B$ is norm-contracting, Equation \ref{eq:norm1} extends to all $a\in A$.

    Since $\eta:C\to \mathcal{B}$ is a *-homomorphism and $\eta(\mathcal{A}) = \eta(\mathcal{C})$, we have $\eta(A)\subseteq \overline{\eta(\mathcal{A})} = \eta(C).$ Therefore, $\eta(A) = \eta(C)$.
    
    So, for $(a_{i,j})\in M_{n}(A)$, let $(c_{i,j})\in M_{n}(C)$ be such that $\eta(a_{i,j}) = \eta(c_{i,j})$ for all $i,j\leq n$. Using Equation \ref{eq:norm1}, there is $\lambda_{0}$ such that $\phi = (1-u_{\lambda_{0}})1_{n}\in \widetilde{M_{n}(C)}$ and $c = (1-u_{\lambda_{0}})((c_{i,j}))(1-u_{\lambda_{0}})\in M_{n}(C)$ such that $\|\phi\|\leq 1$, $\eta^{(n)}(\phi) = 1$, $\eta^{(n)}((a_{i,j})) = \eta^{(n)}(c)$ and $\|\phi^{*}(a_{i,j})\phi - c\| < \epsilon$. This proves $\eta^{(n)}$ is compressible to $\eta^{(n)}:M_{n}(C)\to M_{n}(B)$.

    Since $\eta:\mathcal{A}\to \mathcal{B}$ is completely positive and contracting (and extends to the C*-completions as thus), it extends to a unital completely positive and contracting map between the unitizations $\eta:\tilde{\mathcal{A}}\to \tilde{\mathcal{B}}$ by (\cite[Proposition~2.2.1]{BO08}. Therefore,  by  \cite[Theorem~3.18]{P03}, the multiplicative domain $\mathcal{M} = \{a\in \mathcal{A}: \eta(a^{*}a) = \eta(a)^{*}\eta(a)\text{ and }\eta(aa^{*}) = \eta(a)\eta(a)^{*}\}$. Therefore, $\mathcal{C}\subseteq \mathcal{M}$, which proves $\eta$ is compressible to $\eta:\mathcal{M}\to \mathcal{B}$.
\end{proof}

Being compressible to a larger domain is a weaker property, as we have less control over the approximate unit in the norm equation (\ref{eq:norm1}). This is why compressibility was not formulated in terms of the multiplicative domain.

\begin{cor}
\label{cor:quotientnorm}
    If $\eta:\mathcal{A}\to\mathcal{B}$ is compressible to the bounded *-homomorphism $\eta:\mathcal{C}\to\mathcal{B}$ and $\ker(\eta:\mathcal{C}\to\mathcal{B})$ contains an approximate unit for $\ker(\eta:C\to B)$, then for any $a\in\mathcal{A}$, we have
\begin{equation}
\label{eq:norm2}
    \|\eta(a)\| = \inf\{\|b\|:b\in\mathcal{A}\text{ and }\eta(a) = \eta(b)\}.
\end{equation}
\end{cor}
\begin{proof}
    Let $(u_{\lambda})$ be an approximate unit for $\ker(\eta:C\to B)$. For $b\in \mathcal{A}$ such that $\eta(a) = \eta(b)$, since $\eta$ is a norm contraction (Corollary \ref{cor:cpc}) we have $\|\eta(a)\|\leq \|b\|$. Hence, $\|\eta(a)\|\leq \inf\{\|b\|:\eta(a) = \eta(b)\}$. For every $\epsilon > 0$, there is $\lambda_{0}$ such that $\|(1-u_{\lambda_{0}})b(1-u_{\lambda_{0}})\|\leq \|\eta(a)\| + \epsilon $. Since $u_{\lambda}$ is in the multiplicative domain $\eta$ (Corollary \ref{cor:cpc}), we have $\eta(a) = \eta((1-u_{\lambda_{0}})b(1-u_{\lambda_{0}}))$ and hence $\inf\{\|b\|:\eta(a) = \eta(b)\}\leq \|(1-u_{\lambda_{0}})b(1-u_{\lambda_{0}})\|\leq \|\eta(a)\| + \epsilon$. Since $\epsilon $ is arbitrary, we have $\inf\{\|b\|:\eta(a) = \eta(b)\}\leq \|\eta(a)\|$, proving the corollary.
\end{proof}

It is important to know that compressibility is preserved under taking completions of pre-C*-algebras.

\begin{cor}
\label{cor:completion}
    If $\eta:\mathcal{A}\to\mathcal{B}$ is compressible to the bounded *-homomorphism $\eta:\mathcal{C}\to \mathcal{B}$, then $\eta:A\to B$ is compressible to $\eta:C\to B$.
\end{cor}
\begin{proof}
    The proof is contained in the three paragraphs before the last in the proof of Corollary \ref{cor:cpc}.
\end{proof}

\begin{qtn}
    When is a c.p.c. map $\eta:A\to B$ between C*-algebras compressible to its multiplicative domain?
\end{qtn}

The proof of the next result follows \cite[Corollary~2.7]{CN22} closely.

\begin{cor}
\label{cor:prekerdense}
    If $\eta$ is compressible to a bounded *-homomorphism $\eta:\mathcal{C}\to\mathcal{B}$ such that $\ker(\eta:\mathcal{C}\to\mathcal{B})$ contains an approximate unit for $\ker(\eta:C\to B)$, then $\ker(\eta:\mathcal{A}\to\mathcal{B})$ is dense in $\ker(\eta:A\to B)$
\end{cor}

\begin{proof}
    Let $(u_{\lambda})\subseteq \ker(\eta:\mathcal{C}\to\mathcal{B})$ be an approximate unit. For every $a\in \ker(\eta:A\to B)$ and $\epsilon > 0$, we can choose $a_{0}\in\mathcal{A}$ such that $\|a - a_{0}\| < \frac{\epsilon}{2}$. Since $\eta$ is norm contracting (Corollary \ref{cor:cpc}), we have $\|\eta(a_{0})\| = \|\eta(a - a_{0})\| < \frac{\epsilon}{2}$. Now, by Theorem \ref{thm:normeqn}, there is $\lambda_{0}$ such that $a_{1}:= (1-u_{\lambda_{0}})a_{0}(1-u_{\lambda_{0}})$ satisfies $\|a_{1}\| < \frac{\epsilon}{2}$. By Corollary \ref{cor:cpc}, $u_{\lambda_{0}}\in \mathcal{M}$ and hence $\eta(a_{1}) = \eta(a_{0})$. So, $a_{2}:= a_{0} - a_{1}\in \ker(\eta: \mathcal{A}\to \mathcal{B})$ and $\|a - a_{2}\|\leq \|a - a_{0}\| + \|a_{1}\| < \epsilon$, proving the corollary.
\end{proof}

Maps compressible to *-homomorphisms send ideals to ideals in the image (compare with \cite[Lemma~2.1]{CN24}.

\begin{cor}
\label{cor:imidealisideal}
    If $\eta$ is compressible to a bounded *-homomorphism $\eta:\mathcal{C}\to\mathcal{B}$, then $\eta(A) = \eta(C)$ is a C*-sub-algebra of $B$. Moreover, if $J$ is a closed two-sided ideal of $A$, then $\eta(J)$ is a closed two-sided ideal of $\eta(A)$.
\end{cor}

\begin{proof}
    We have $\eta(A) = \eta(C)$ from the fact that compressibility extends to completions of pre-C*-algebras Corollary \ref{cor:completion}.

    By Corollary \ref{cor:cpc}, $C$ is contained in the multiplicative domain of $\eta$. So, for $j\in J$ and $b\in \eta(A)$, choose $c\in C$ with $\eta(c) = b$. We have $b\eta(j) = \eta(cj)$ and $\eta(j)b = \eta(jc)$, proving that $\eta(J)$ is a two-sided ideal of $\eta(A)$. To finish the proof, it suffices to show $\eta(J)$ is complete. Let $(\eta(j_{n}))\subseteq \eta(J)$ be a sequence such that $\sum^{\infty}_{n=1}\|\eta(j_{n})\| < \infty$. Let $(u_{\lambda})$ be an approximate unit for $\ker(\eta:C\to B)$. By Theorem \ref{thm:normeqn}, for every $n\in\mathbb{N}$, there is $\lambda_{n}$ such that $j'_{n} = (1-u_{\lambda_{n}})j_{n}(1-u_{\lambda_{n}})$ satisfies $\eta(j'_{n}) = \eta(j_{n})$ and $\|j'_{n}\|\leq \|\eta(j_{n})\| + \frac{1}{2^{n}}$. Therefore, $\sum^{\infty}_{n=1} \|j'_{n}\| < \infty$, so that $\sum^{\infty}_{n=1}j'_{n}\in J$. By continuity of $\eta$, it follows that $\sum_{n=1}^{\infty}\eta(j_{n}) = \eta(\sum^{\infty}_{n=1}j'_{n})\in \eta(J)$.
\end{proof}

\begin{rmk}
    The proof (in Corollary \ref{cor:imidealisideal}) that $\eta(J)$ is closed only required that $J$ is a closed vector space such that $u_{\lambda} J\subseteq J$ and $Ju_{\lambda}\subseteq J$ for all $\lambda$, for some approximate unit $(u_{\lambda})$ for $\ker(\eta:C\to B)$.
\end{rmk}

We show compressibility is preserved under quotients.

\begin{cor}
\label{cor:comppasstoquotient}
    If $\eta:\mathcal{A}\to \mathcal{B}$ is compressible to a surjective bounded *-homomorphism $\eta:\mathcal{C}\to \mathcal{B}$ and $J$ is a closed two-sided ideal of $A$, then $\eta:\mathcal{A}/J\to \mathcal{B}/\eta(J)$ is compressible to $\eta:\mathcal{C}/J\to \mathcal{B}/\eta(J)$. If $(u_{\lambda})$ is an approximate unit for $\ker(\eta:C\to B)$, then  $(u_{\lambda} + J)$ is an approximate unit for $\ker(\eta: C/J\to B/\eta(J))$.
\end{cor}

\begin{proof}
    Let $(u_{\lambda})$ be an approximate unit for $\ker(\eta:C\to B)$. For $c\in C$ such that $\eta(c)\in \eta(J)$, let $j\in J$ be such that $\eta(cc^{*}) = \eta(j)$. Then, by Theorem \ref{thm:normeqn}, we have $\limsup_{\lambda}\| (1 - u_{\lambda})c + J\|^{2}\leq  \limsup_{\lambda}\|(1-u_{\lambda})(cc^{*} - j)(1-u_{\lambda})\| = 0$, proving that $(u_{\lambda} + J)$ is an approximate unit for $\ker(\eta: C/J\to B/\eta(J))$.

    Now, if $a\in\mathcal{A}$ and $\epsilon > 0$, choose $\phi\in\widetilde{\mathcal{C}}$ and  $c\in\mathcal{C}$ such that $\|\phi\|\leq 1$, $\eta(\phi) = 1$, $\eta(a) = \eta(c)$, and $\|\phi^{*} a\phi  - c\|\leq \epsilon$. Then, $\phi + J$ and $c + J$ satisfy $\eta(\phi + J) = 1$, $\eta(a + J) = \eta (c + J)$ and $\|\phi^{*}a\phi - c + J\|\leq \|\phi^{*} a\phi  - c\|\leq \epsilon,$ proving that $\eta:\mathcal{A}/J\to \mathcal{B}/\eta(J)$ is compressible to $\eta:\mathcal{C}/J\to \mathcal{B}/\eta(J)$.
\end{proof}

The following result will play a major role in determining the isotropy fibres of the singular ideal.

\begin{thm}
\label{thm:kernelsurjects}
    Suppose 

    $$
\begin{tikzcd}
A_{1} \arrow[r, "i"] \arrow[d, "\eta_{1}"] & A_{2} \arrow[d, "\eta_{2}"] \\
B_{1} \arrow[r, "j"]                       & B_{2}                      
\end{tikzcd}$$ is a commutative diagram of C*-algebras, where $i$ and $j$ are *-homomorphisms and $\eta_{1},\eta_{2}$ are compressible to *-homomorphisms $\eta_{1}:C_{1}\to B_{1}$, $\eta_{2}:C_{2}\to B_{2}$ with $\eta_{1}$ surjective.

Assume that there is an approximate unit $(u_{\lambda})$ for $\ker(\eta_{1}:C_{1}\to B_{1})$ such that $(i(u_{\lambda}))$ is an approximate unit for $\ker(\eta_{2}:C_{2}\to B_{2})$. Then, $\eta_{1}(\text{ker}(i)) = \text{ker}(j)$. Additionally, if $i(C_{1})\subseteq C_{2}$, then $\eta_{1}(\ker(i)\cap C_{1}) = \ker(j)$.
\end{thm}

\begin{proof}
Let $A'_{1} = A_{1}/\ker(i)$ and $B'_{1} = B_{1}/\eta_{1}(\ker(i))$, which is a C*-algebra by Corollary \ref{cor:imidealisideal}. By commutativity of the diagram, we have $\eta_{1}(\ker(i))\subseteq \ker(j)$, and so $i$ and $j$ pass to well defined *-homomorphisms $i' = i:A'_{1}\to A_{2}$ and $j' = j:B'_{1}\to B_{2}$, making the diagram 

$$
\begin{tikzcd}
A'_{1} \arrow[r, "i'"] \arrow[d, "\eta'_{1}"] & A_{2} \arrow[d, "\eta_{2}"] \\
B'_{1} \arrow[r, "j'"]                        & B_{2}                      
\end{tikzcd}$$ commute, where $\eta'_{1} = \eta_{1}:A'_{1}\to B'_{1}$. By Corollary \ref{cor:imidealisideal}, $\eta'_{1}$ is compressible to $\eta:C_{1}/\ker(i)\to B_{1}/\eta_{1}(\ker(i))$ and $(u'_{\lambda} = u_{\lambda} + \ker(i))$ is an approximate unit for $\ker(\eta'_{1}:C_{1}/\ker(i)\to B_{1}/\eta_{1}(\ker(i)))$. Moreover, $(i'(u'_{\lambda}) = i(u_{\lambda}))$ is an approximate unit for $\ker(\eta_{2}:C_{2}\to B_{2})$ by the hypothesis.

For $b'\in B'_{1}$, let $a'\in A'_{1}$ be such that $\eta'_{1}(a') = b'$. Then, by commutativity of the above diagram, $i'(a')$ satisfies $\eta_{2}(i'(a')) = j'(b')$. By Theorem \ref{thm:normeqn}, it follows that $\|j'(b')\| = \lim_{\lambda}\|i'(1-u'_{\lambda})i'(a')i'(1-u'_{\lambda})\|$. Since $i'$ is an injective *-homomorphism, we have $$\lim_{\lambda}\|i'(1-u'_{\lambda})i'(a')i'(1-u'_{\lambda})\| = \lim_{\lambda}\|(1-u'_{\lambda})a'(1-u'_{\lambda})\| = \|\eta'_{1}(a')\| = \|b'\|.$$ Therefore, $\|j'(b')\| = \|b'\|$. We have proven $j'$ is injective, so that $\eta_{1}(\ker(i)) = \ker(j)$.

Now, if $i(C_{1})\subseteq C_{2}$, we can apply the theorem as proven to the case where $A_{1} = C_{1}$ and $A_{2} = C_{2}$ to obtain $\eta_{1}(\ker(i)\cap C_{1}) = \ker(j)$.
\end{proof}

\begin{section}{Compressibility of restrictions to locally invariant subgroupoids of \'etale groupoid C*-algebras}

We introduce the notion of a locally invariant set $X$ of units in an \'etale groupoid, and show every the restriction $\mathscr{C}_{c}(G)\to \mathscr{C}_{c}(G|_{X})$ is compressible to a *-homomorphism. Every \'etale groupoid admits a rich and interesting collection of locally invariant sets (containing every finite set). As a result, these restriction maps provide a new tool to study \'etale groupoid C*-algebras. In this paper, we will use them to determine the isotropy fibres of the singular ideal.

\begin{dfn}
\label{dfn:locinvar}
    Let $G$ be an \'etale groupoid. We say a closed set $X\subseteq G^{0}$ is \emph{locally invariant} if for every $g\in G$ with $r(g),s(g)\in X$, there is an open neighbourhood $U$ of $g$ such that, for $\tilde{g}\in U$, $s(\tilde{g})\in X$ if and only if $r(\tilde{g})\in X$.
\end{dfn}

Any finite set $F\subseteq G$ is locally invariant: for $g\in G|_{F}$, let $U$ be an open bisection containing $g$ such that $r(U)\cap F = \{r(g)\}$ and $s(U)\cap F = \{s(g)\}$. It is also easy to see a finite intersection or arbitrary union of locally invariant sets is locally invariant. We will now prove an alternative characterization of this property.

\begin{prp}
\label{prp:germexists}
    Let $G$ be an \'etale groupoid. Then, a set $X$ is locally invariant if and only if there is an open subgroupoid $H\subseteq G$ such that $G|_{X} = H|_{X}$ and $X$ is an invariant subset of $H$. Moreover, $H$ can be chosen such that $H^{0} = G^{0}$.
\end{prp}

\begin{proof}
    The `if'' direction is trivial, so we prove the ``only if'' direction. Suppose $X$ is locally invariant. For every $g\in G|_{X}$, let $U_{g}$ be an open set such that, for $\tilde{g}\in U_{g}$, $s(\tilde{g})\in X$ if and only if $r(\tilde{g})\in X$. Let $H$ be the union of all finite products of the open sets $U_{g}$, $U_{g}^{-1}$. Since the product and inverse maps are open, $H$ is open. By construction, $H$ is closed under products and inverses, so $H$ is an open subgroupoid. Write $V = U_{1}\cdot...\cdot U_{n}$, where $U_{k}\in \{U_{g},U^{-1}_{g}\}_{g\in G|_{X}}$ for $k\leq n$. By induction on $n\in\mathbb{N}$, we see that for $\tilde{g} = \tilde{g}_{1}\cdot...\cdot\tilde{g}_{n}\in V$, $r(\tilde{g})\in X$ if and only if $s(\tilde{g})\in X$. Hence, $X$ is $H$-invariant and by construction $H|_{X} = G|_{X}$. To construct $H$ with $H^{0} = G^{0}$, let $U_{x} = G^{0}$ for $x\in X$.
\end{proof}

We will be using open subgroupoids of the above form quite a bit in this paper, so they deserve a name.

\begin{dfn}
\label{dfn:hull}
    Let $G$ be an \'etale groupoid and $X\subseteq G^{0}$ a locally invariant set. An open subgroupoid $H$ of $G$ for which $X$ is $H$-invariant and $H|_{X} = G|_{X}$ will be called a \emph{local groupoid} about $G|_{X}$.
\end{dfn}

Let's establish a relationship between locally invariant sets of $G$ and of its Hausdorff cover $\tilde{G}$.

\begin{prp}
\label{prp:locinvar&cover}
    Let $G$ be an \'etale groupoid. If $X\subseteq G^{0}$ is locally invariant, then $\pi^{-1}(X)$, $\pi_{ess}^{-1}(X)$ are locally invariant in the Hausdorff cover $\tilde{G}$ and essential Hausdorff cover $\tilde{G}_{ess}$, respectively. Moreover, if $H$ is a local groupoid about $G|_{X}$, then $\hat{H} := H\cdot \tilde{G}^{0}$ and $\hat{H}_{ess} := H\cdot \tilde{G}^{0}_{ess}$ are local groupoids about $\tilde{G}|_{\pi^{-1}(X)}$ and  $\tilde{G}_{ess}|_{\pi_{ess}^{-1}(X)}$.
\end{prp}

\begin{proof}
    For $\underline{x}\in \tilde{G}^{0}$ and $h\in H$ such that $\pi(\underline{x}) = r(h)$, we have $\underline{x}h = h(h^{-1}\underline{x}h)$, proving that $\tilde{G}^{0}\cdot H = H\cdot \tilde{G}^{0}$. Consequently $\hat{H}$ is a groupoid. 
    
    If $h\underline{x}\in \hat{H}$ and $s(h\underline{x}) = \underline{x}\in \pi^{-1}(X)$, then $s(h) = \pi(\underline{x})\in X$. By $H$-invariance of $X$ we have $\pi(r(h\underline{x})) = r(h)\in X$ and therefore $r(h\underline{x})\in \pi^{-1}(X)$, proving that $\pi^{-1}(X)$ is $\hat{H}$-invariant.

    Now, we show $\hat{H}$ is an open subgroupoid of $\tilde{G}$. Suppose $(\underline{h}_{\lambda})$ is a net in $\tilde{G}$ converging to $\underline{h} = h\underline{x}\in \hat{H}$. By Proposition \ref{prp:fell}, there is a net $h_{\lambda}\in \underline{h}_{\lambda}$ converging to $h\in H$. Since $H$ is open in $G$, $h_{\lambda}\in H$ eventually. We can write $(\underline{h}_{\lambda} = h_{\lambda}\underline{x}_{\lambda})$, where $\underline{x}_{\lambda} = h^{-1}_{\lambda}\underline{h}_{\lambda}\in \tilde{G}^{0}$ and hence $\underline{h}_{\lambda}\in \hat{H}$ eventually, proving that $\hat{H}$ is open in $\tilde{G}$.

    We have proven $\hat{H}$ is a local groupoid about $\tilde{G}|_{\pi^{-1}(X)}$. Since $\hat{H}\cap \tilde{G}_{ess} = \hat{H}_{ess}$ and $\pi^{-1}(X)\cap \tilde{G}^{0}_{ess} = \pi^{-1}_{ess}(X)$, it follows that $\hat{H}_{ess}$ is a local groupoid about $\tilde{G}_{ess}|_{\pi_{ess}^{-1}(X)}$
\end{proof}
We begin our investigations of restrictions to locally invariant sets by proving a slight improvement of \cite[Theorem~3.13]{BGHL25} that will streamline some of the arguments in this section.
\begin{lem}
\label{lem:PO1}
Let $G$ be an \'etale groupoid. Suppose $Y$ is a closed set of $G$ and $f\in\mathscr{C}_{c}(G)$ is such that $f|_{Y} = 0$. Then, for any compact set $K\subseteq G$ and bisections $\{U_{i}\}^{n}_{i=1}$ such that $\{g\in G:f(g)\neq 0\}\subseteq K\subseteq \bigcup^{n}_{i=1}U_{i}$, there are functions $f_{i}\in C_{c}(U_{i})$ such that $f_{i}|_{Y} = 0$ for all $1\leq i\leq n$ and $f = \sum^{n}_{i=1}f_{i}$.
\end{lem}
\begin{proof}

We prove the lemma by induction on $n$. The case $n = 1$ is settled because $f\in C_{c}(U_{1})$ (see the proof of \cite[Theorem~3.13(i)]{BGHL25}). Now suppose the lemma is true for $n -1 \geq 1$ and let's settle it for $n$.

For each $k\leq n$, choose open bisections $V_{k}$ and compact sets $K_{k}$ with $K_{k}\subseteq V_{k}\subseteq U_{k}$ such that $K\subseteq \bigcup^{n}_{k=1}K_{k}$ and the closure $\overline{V_{k}}^{U_{k}}$  of $V_{k}$ relative to $U_{k}$ is compact in $U_{k}$, for all $k\leq n$. As in the proof of \cite[Theorem~3.13(i)]{BGHL25} the function $f$ is continuous on $\check{V}_{n}:= V_{n}\setminus \bigcup_{k < n} V_{k}$ with $\{g\in \check{V}_{n}: f(g)\neq 0\}$ contained in the compact set $\check{K}_{n}:= K_{n}\setminus \bigcup_{k < n} V_{k}$. Moreover, $f|_{Y\cap \check{V}_{n}} = 0$. By the Tietze extension theorem, there is $f_{n}\in C_{c}(V_{n})\subseteq C_{c}(U_{n})$ such that $f_{n}|_{\check{V}_{n}} = f|_{\check{V}_{n}}$ and $f_{n}|_{Y\cap V_{n}} = 0$. Then, $f':= f - f_{n}$ has $f'|_{Y} = 0$ and $\{g\in G: f'(g)\neq 0\}\subseteq K' \subseteq \bigcup^{n-1}_{k=1} U_{k}$, with $K':= \bigcup^{n-1}_{k=1} \overline{V_{k}}^{U_{k}}$ and we may invoke the induction hypothesis to $f'$ to get $f' = \sum_{k=1}^{n-1}f_{k}$ for some $f_{k}\in C_{c}(U_{k})$ with $f_{k}|_{Y} = 0$. Therefore the lemma is proved for $f = \sum^{n}_{k=1}f_{k}$, thus proving the lemma by induction.
\end{proof}

\begin{thm}
\label{thm:locinvar&compressdense}
    Let $G$ be an \'etale groupoid and $X$ a locally invariant subset of $G^{0}$. Equip $\mathscr{C}_{c}(G)$ and $\mathscr{C}_{c}(G|_{X})$ with any pre-C*-norms and let $H$ be a local groupoid about $G|_{X}$ with $H^{0} = G^{0}$. Then, the restriction map $r_{X}:\mathscr{C}_{c}(G)\to \mathscr{C}_{c}(G|_{X})$ is compressible to the surjective *-homomorphism $r_{X}:\mathscr{C}_{c}(H)\to \mathscr{C}_{c}(G|_{X})$. Moreover, any approximate unit $(u_{\lambda})\subseteq C_{c}(G^{0}\setminus X)$ is an approximate unit for $\ker(r_{X}:\mathscr{C}_{c}(H)\to \mathscr{C}_{c}(G|_{X}))$.
\end{thm}

\begin{proof}
    We first show $r_{X}:\mathscr{C}_{c}(H)\to \mathscr{C}_{c}(G|_{X})$ is surjective. Let $b\in C_{c}(V)$, where $V$ is an open bisection of $G|_{X}$. Let $\{U_{i}\}^{n}_{i=1}$ be bisections of $H$ that cover a compact set $K\subseteq V$ such that $\{g\in G|_{X}: b(g)\neq 0\}\subseteq K$. Choose functions $\phi_{i}\in C_{c}(U_{i}\cap V)$ such that $\sum^{n}_{i=1}\phi_{i}|_{K} = 1$ and let $b_{i} = b\phi_{i}\in C_{c}(U_{i}\cap V)$. Since $C_{c}(U_{i}\cap V)\subseteq C_{c}(U_{i}\cap G|_{X})$ and $G|_{X}$ is closed, by the Tietze extension theorem, there is $a_{i}\in C_{c}(U_{i})$ such that $a_{i}|_{G|_{X}} = b_{i}$. Hence, $a = \sum_{i=1}^{n}a_{i}\in \mathscr{C}_{c}(H)$ satisfies $r_{X}(a) = b$.

    Now, let $(u_{\lambda})$ be an approximate unit for $C_{c}(G^{0}\setminus X)$. We first show for any $a\in \mathscr{C}_{c}(G)$ with $a|_{G|_{X}} = 0$, we have $\limsup_{\lambda}\|(1-u_{\lambda})a(1-u_{\lambda})\|_{\rho} = 0$. Since $G|_{X}$ is closed, by Lemma \ref{lem:PO1} and the triangle inequality, it suffices to show this for $a\in C_{c}(U)$ for some open bisection and with $\|a\|_{\infty}\leq 1$, where $\|\cdot\|_{\infty}$ denotes the sup-norm. 
    
    Let $K\subseteq U$ be a compact set such that $\{g\in G: a(g)\neq 0\}\subseteq K$. Let $ 1 > \epsilon > 0$ and $U'\subseteq U$ an open set such that  $K\cap G|_{X}\subseteq U'$ and $|a(g)| < \epsilon$ for all $g\in U'$. There is an open set $W\subseteq G^{0}$ such that $(r(K)\cup s(K))\cap X\subseteq W$ and $r^{-1}(W)\cap s^{-1}(W)\cap K \subseteq U'$; for if not, then we can extract a net $(g_{\lambda})\subseteq K\setminus U'$ converging to $g\in K\setminus U'\subseteq K\setminus G|_{X}$ with $r(g_{\lambda})$ and $s(g_{\lambda})$ converging to elements in $X$, a contradiction. Set $L = r(K)\cup s(K)$. By applying the approximate unit to $\phi\in C_{c}(G^{0}\setminus X)$ satisfying $\phi|_{L\setminus W} = 1$, we see that there is $\lambda_{0}$ such that $|(1-u_{\lambda})(x)|\leq \epsilon$ for all $\lambda\geq \lambda_{0}$ and $x\in L\setminus W$.

    For $g\in U'$, we have $|(1-u_{\lambda})(r(g))a(g)(1-u_{\lambda})(s(g))|\leq |a(g)| < \epsilon$ and for $g\notin r^{-1}(W)\cap s^{-1}(W)\cap K\subseteq U'$, we have $|(1-u_{\lambda})(r(g))a(g)(1-u_{\lambda})(s(g))|\leq \epsilon^{2}|a(g)|\leq \epsilon^{2} < \epsilon $ for all $\lambda \geq \lambda_{0}$.

    Since $f_{\lambda} = (1-u_{\lambda})a(1-u_{\lambda})$ is supported on the open bisection $U$, we have $\|f_{\lambda}\|_{\rho} = \|f_{\lambda}\|_{\infty}  < \epsilon $ for all $\lambda\geq \lambda_{0}$. This norm equality follows from the C*-identity $\|f\|_{\rho} = \sqrt{\|f^{*}f\|_{\rho}}$,  $|f|^{2} = (f^{*}f)\circ s$ and that elements $f^{*}f\in C_{c}(s(U))$ have a unique C*-norm provided by the sup-norm. We have proven $\limsup_{\lambda}\|(1-u_{\lambda})a(1-u_{\lambda})\|_{\rho} = 0$. Now, suppose $f\in $

    Note that this shows $(u_{\lambda})$ is an approximate unit for $\ker(r_{X}:\mathscr{C}_{c}(H)\to \mathscr{C}_{c}(G|_{X}))$ (apply the above to $a = c^{*}c$ where $c\in \ker(r_{X}:\mathscr{C}_{c}(H)\to \mathscr{C}_{c}(H|_{X}))$).

    Now, we show compressibility. Let $a\in \mathscr{C}_{c}(G)$ and choose $c'\in \mathscr{C}_{c}(H)$ such that $r_{X}(a - c') = 0$. From above, for every $\epsilon > 0$, there is $\lambda_{0}$ such that $\phi = 1-u_{\lambda_{0}}$ and $c = (1-u_{\lambda_{0}})c'(1- u_{\lambda_{0}})\in\mathscr{C}_{c}(H)$ satisfy $\|\phi^{*}a\phi - c\| < \epsilon$, along with $\eta(\phi) = 1$, $\eta(a) = \eta(c)$. Therefore, $r_{X}$ is compressible to $r_{X}:\mathscr{C}_{c}(H)\to \mathscr{C}_{c}(G|_{X})$.
\end{proof}

If $\|\cdot\|_{\rho}$ is a pre-C*-norm on $\mathscr{C}_{c}(G)$, then denote by  $\|\cdot\|_{\rho(X)}$ the quotient norm on $\mathscr{C}_{c}(G|_{X})$ induced from $\mathscr{C}_{c}(H)$, where $H$ is a local groupoid about $G|_{X}$ - that is, for $b\in \mathscr{C}_{c}(G|_{X})$, let $\|b\|_{\rho(X)} := \text{inf}\{\|a\|_{\rho}:a\in\mathscr{C}_{c}(H), \text{ }r_{X}(a) = b\}$. Since $r_{X}:\mathscr{C}_{c}(H)\to \mathscr{C}_{c}(G|_{X})$ is a *-homomorphism, we have that the closure $C^{*}_{\rho}(H\setminus H|_{X})$ of $\ker(\mathscr{C}_{c}(H)\to \mathscr{C}_{c}(G|_{X}))$ is an ideal in $C^{*}_{\rho}(H)$ and for any $a\in \mathscr{C}_{c}(H)$, we have $$\|\eta(a)\|_{\rho(X)} = \text{inf}\{ \|a + j\|:j\in C^{*}_{\rho}(H\setminus H|_{X})\}.$$ Therefore, $\|\cdot\|_{\rho(X)}$ is the quotient norm by the ideal $C^{*}_{\rho}(H\setminus H|_{X})$, and is hence a pre-C*-norm. Before we show that this norm is independent of $H$, let's extend the results in Theorem \ref{thm:locinvar&compressdense} to the completions under $\rho$ and $\rho(X)$.

\begin{thm}
    \label{thm:locinvar&compress}
     Let $G$ be an \'etale groupoid, $X\subseteq G^{0}$ a locally invariant set and $\rho$ a pre-C*-norm for $\mathscr{C}_{c}(G)$. For any local groupoid $H$ about $G|_{X}$ with $H^{0} = G^{0}$, the $*$-homomorphism $\eta_{X}:\mathscr{C}_{c}(H)\to \mathscr{C}_{c}(G|_{X})$ is bounded relative to the norms $\rho$ and $\rho(X)$. Consequently, the restriction map $\eta_{X}:\mathscr{C}_{c}(G)\to \mathscr{C}_{c}(G|_{X})$ extends to $\eta_{X}:C^{*}_{\rho}(G)\to C^{*}_{\rho(X)}(G|_{X})$ and is compressible to $C^{*}_{\rho}(H)\subseteq C^{*}_{\rho}(G)$. Furthermore, any approximate unit $(u_{\lambda})$ for $C_{c}(G^{0}\setminus X)$ is an approximate unit for $\ker(C^{*}_{\rho}(H)\to C^{*}_{\rho(X)}(G|_{X}))$.
\end{thm}

\begin{proof}
    By Theorem \ref{thm:locinvar&compressdense}, $\eta_{X}:\mathscr{C}_{c}(G)\to \mathscr{C}_{c}(G|_{X})$ is compressible to $\mathscr{C}_{c}(H)$, and $\eta_{X}:\mathscr{C}_{c}(H)\to \mathscr{C}(G|_{X})$ is norm contracting with $\mathscr{C}(G|_{X})$ equipped with the quotient norm relative to $\rho$ and this *-homomorphism. Therefore, Corollary \ref{cor:completion} implies $\eta_{X}$ extends to the completion $\eta_{X}:C^{*}_{\rho}(G)\to C^{*}_{\rho}(G|_{X})$ and is compressible to $\eta_{X}:C^{*}_{\rho}(H)\to C^{*}_{\rho(X)}(G|_{X})$.

    By Theorem \ref{thm:locinvar&compress}, $(u_{\lambda})$ is an approximate unit for $\ker(r_{X}:\mathscr{C}_{c}(H)\to \mathscr{C}_{c}(G|_{X}))$. Since $\ker(r_{X}:\mathscr{C}_{c}(H)\to \mathscr{C}_{c}(G|_{X}))$ is dense in $\ker(r_{X}:C^{*}_{\rho}(H)\to C^{*}_{\rho(X)}(G|_{X}))$ by construction, we have that $(u_{\lambda})$ is also an approximate unit for $\ker(r_{X}:C^{*}_{\rho}(H)\to C^{*}_{\rho(X)}(G|_{X}))$.
\end{proof}

The below corollary to Theorem \ref{thm:locinvar&compress} amounts to saying this norm is independent of the local groupoid $H$ about $G|_{X}$ inside $G$ chosen.

\begin{cor}
\label{cor:restrictnorm}
Let $G$ be an \'etale groupoid and $X$ a closed locally invariant set. Let $H$ and $K$ be local groupoids about $G|_{X}$. Equip $\mathscr{C}_{c}(G)$ with any pre-C*-norm and let $\mathscr{C}_{c}(H)$ and $\mathscr{C}_{c}(K)$ inherit this norm as *-sub-algebras.
 
The quotient norm on $\mathscr{C}_{c}(G|_{X})$ induced from $\mathscr{C}_{c}(H)$ is equal to that induced from $\mathscr{C}_{c}(K)$.
\end{cor}

\begin{proof}
    When $H^{0} = K^{0} = G^{0}$ this follows from Theorem \ref{thm:locinvar&compress} and Equation \ref{eq:norm1}, using an approximate unit $(u_{\lambda})\subseteq C_{c}(G^{0}\setminus X)$.

    In general, let $H' = H\cup G^{0}$, and denote by $\|\cdot\|'$ the quotient norm on $\mathscr{C}_{c}(G|_{X})$ induced from $\mathscr{C}_{c}(H')$. For $b\in \mathscr{C}_{c}(G|_{X})$ choose $a\in \mathscr{C}_{c}(H)$ such that $r_{X}(a) = b$. For an approximate unit $(u_{\lambda})\subseteq C_{c}(G^{0}\setminus X)$, we have $(1-u_{\lambda})a(1-u_{\lambda})\in \mathscr{C}_{c}(H)$ for all $\lambda$ (since $C^{*}_{r}(H)$ is an ideal in $C^{*}_{r}(H')$). Therefore, from Equation \ref{eq:norm1}, we have $$\inf\{\|c\|: c\in C^{*}_{r}(H)\text{ and } r_{X}(c) = b\}\leq \lim_{\lambda}\|(1-u_{\lambda})a(1-u_{\lambda})\| = \|b\|'.$$ By Equation \ref{eq:norm2} we have the reverse inequality $$\|b\|' = \inf\{\|c\|: c\in C^{*}_{r}(H')\text{ and } r_{X}(c) = b\}\leq \inf\{\|c\|: c\in C^{*}_{r}(H)\text{ and } r_{X}(c) = b\}.$$ Hence, $\|b\|' = \inf\{\|c\|: c\in C^{*}_{r}(H)\text{ and } r_{X}(c) = b\}$. The right hand side is the quotient norm induced from $\mathscr{C}_{c}(H)$, thus proving the corollary.
\end{proof}

We record, for convenience, the norm equation we have proved.

\begin{cor}
\label{cor:normeqngpd}
       Let $G$ be an \'etale groupoid, $X\subseteq G^{0}$ a locally invariant set and $\rho$ a pre-C*-norm for $\mathscr{C}_{c}(G)$.
       
       Then, for any approximate unit $(u_{\lambda})$ for $C_{0}(G^{0}\setminus X)$ and any $a\in C^{*}_{\rho}(G)$, we have 
       
       \begin{equation}
       \label{eq:norm3}
           \lim_{\lambda}\|(1-u_{\lambda})a(1-u_{\lambda})\|_{\rho} = \|\eta_{X}(a)\|_{\rho(X)}.
       \end{equation}
\end{cor}

\begin{proof}
    This is immediate from Theorem \ref{thm:locinvar&compress} and Equation \ref{eq:norm1} in Theorem \ref{thm:normeqn}.
\end{proof}

\begin{rmk}
    In the special case that $X = \{x\}$, $x\in G^{0}$, the norm $\rho(x)$ is equal to the exotic norm on $\mathbb{C}[G^{x}_{x}]$ defined by Christensen and Neshveyev in \cite{CN22}. Equation \ref{eq:norm3} and \cite[Equation~2.3]{CN22} are the same.
\end{rmk}

Corollary \ref{cor:restrictnorm} says the norm on a reduction $G|_{X}$ depends only on the ``germ'' of local groupoids. Let's make this precise.

\begin{dfn}
\label{dfn:germiso}
    Let $G_{1}$ and $G_{2}$ be \'etale groupoids, $X_{1}, X_{2}$ locally invariant sets in $G_{1}$, $G_{2}$, respectively. We say $G_{1}$ about $G_{1}|_{X_{1}}$ is \emph{germ isomorphic} to $G_{2}$ about $G_{2}|_{X_{2}}$ if there are local groupoids $H_{1}$ about $G_{1}|_{X_{1}}$ in $G_{1}$, $H_{2}$ about $G_{2}|_{X_{2}}$ in $G_{2}$ and a groupoid isomorphism $\alpha:H_{1}\to H_{2}$ which restricts to a homeomorphism $\alpha:H^{0}_{1}\to H^{0}_{2}$ such that $\alpha(X_{1}) = X_{2}$. We write $(G_{1},X_{1})\simeq_{\alpha}(G_{2}, X_{2})$.
\end{dfn}

\begin{rmk}
    Note that if $H_{1}$ and $H_{2}$ are open groupoids in $G$ containing $G|_{X}$, then $H_{1}$ about $G|_{X}$ is germ isomorphic to $H_{2}$ about $G|_{X}$; consider the identity map on a local groupoid contained in $H_{1}\cap H_{2}$. Similarly, if $(G_{1},X_{1})\simeq_{\alpha}(G_{2}, X_{2})$ and $(G_{2}, X_{2})\simeq_{\beta}(G_{3}, X_{3})$, then $(G_{1},X_{1})\simeq_{\alpha\circ\beta}(G_{3}, X_{3})$, with $\alpha\circ\beta$ defined on the suitable intersections of local groupoids used in the definitions of $\alpha,\beta$.
\end{rmk}

Now, the following corollary to Corollary \ref{cor:restrictnorm} is immediate.

\begin{cor}
\label{cor:germnormiso}
    Let $G_{1}$, $G_{2}$ be \'etale groupoids and $X_{1}, X_{2}$ locally invariant sets in $G_{1}, G_{2}$, respectively. If $(G_{1}, X_{1})\simeq_{\alpha}(G_{2}, X_{2})$, then $\alpha:G_{1}|_{X_{1}}\to G_{2}|_{X_{2}}$ induces a $C^{*}$-algebra isomorphism $\alpha^{*}:C_{r(X_{2})}^{*}(G_{2}|_{X_{2}})\to C^{*}_{r(X_{1})}(G_{1}|_{X_{1}})$ defined as $\alpha^{*}(f) = f\circ\alpha$ for $f\in C_{c}(G_{2}|_{X_{2}})$.
\end{cor}

This corollary is the main motivation to introduce germ isomorphisms, with the aim of aiding in calculating the norms $r(X)$ in examples. We will also show later in Section \ref{ss:charvsing} that a property central to our characterization of vanishing of the singular ideal is invariant under germ isomorphism. For this we will need to know how germ isomorphism behaves with the Hausdorff cover construction.

\begin{lem}
\label{lem:germ&cover}
    Let $G$ be an \'etale groupoid, $X$ a locally invariant set and $H$ a local groupoid about $G|_{X}$. Then, $\tilde{G}|_{\pi_{G}^{-1}(X)} = \tilde{H}|_{\pi_{H}^{-1}(X)}$ (identified as sets of discrete subsets of $G$) with equal topologies. Moreover, $\pi_{G,ess}^{-1}(X) = \pi^{-1}_{H,ess}(X)$.
\end{lem}
\begin{proof}
    The map $\kappa:\tilde{G}\to \mathcal{C}(H)$ given by $Y\mapsto Y\cap H$ is continuous in the Fell topology, and this is easy to see from its definition in terms of basic open sets. For the remainder of the proof, we use the description of the Fell topology as in Proposition \ref{prp:fell}.

    Suppose $Y\in \tilde{G}|_{\pi^{-1}(X)}$. Choose a net $(y_{\lambda})\subseteq G$ such that $(\{y_{\lambda}\})$ converges to $Y$ in the Fell topology of $G$. Since $Y\subseteq G|_{X}\subseteq H$ and $H$ is open, we have $y_{\lambda}\in H$ eventually. Therefore, $Y\cap H = Y$ and $Y\in \tilde{H}$ and hence $\kappa(\tilde{G}|_{\pi_{G}^{-1}(X)})\subseteq \tilde{H}|_{\pi_{H}^{-1}(X)}$. Moreover, suppose $(\{y_{\lambda}\})$ is a net in $H$ converging in the Fell topology of $H$ to $Y\in \tilde{H}|_{\pi^{-1}_{H}(X)}$, then every accumulation point $y$ of $(\{y_{\lambda}\})$ considered as a net in $G$ satisfies $r(y),s(y)\in X$. Therefore, $Y$ is equal to the accumulation set of $(\{y_{\lambda}\})$ considered as a net in $G$, proving that $(\{y_{\lambda}\})$ converges in the Fell topology of $G$ to $Y$. Hence, $Y\in \tilde{G}|_{\pi^{-1}(X)}$, $\kappa(\tilde{G}|_{\pi_{G}^{-1}(X)}) = \tilde{H}|_{\pi_{H}^{-1}(X)}$ and therefore $\tilde{G}|_{\pi_{G}^{-1}(X)} = \tilde{H}|_{\pi_{H}^{-1}(X)}$ (by $Y\cap H = Y$, $Y\in \tilde{G}|_{\pi_{G}^{-1}(X)}$).
    
    Similarly, if $(Y_{\lambda})$ is a net in $\tilde{H}|_{\pi_{H}^{-1}(X)}$ converging to $Y$ in the Fell topology of $H$, then every accumulation point $y$ of $(Y_{\lambda})$ in $G$ satisfies $r(y),s(y)\in X$ and therefore is an accumulation point of $(Y_{\lambda})$ in $H$. Therefore, $(Y_{\lambda})$ converges (by Proposition \ref{prp:fell}) to $Y$ as a net in $\tilde{G}|_{\pi_{G}^{-1}(X)}$. It follows that $\text{id} = \kappa:\tilde{G}|_{\pi_{G}^{-1}(X)} \to \tilde{H}|_{\pi_{H}^{-1}(X)}$ is open. We have proven the first part of the proposition.

    Now, we prove $\pi_{G,ess}^{-1}(X) = \pi^{-1}_{H,ess}(X)$. Let $C_{H}$, $C_{G}$ denote the Hausdorff units of the respective groupoids. Since $H\subseteq G$ we have $C_{G}\cap H^{0}\subseteq C_{H}$. The map $\iota:C_{H}\to \tilde{H}^{0}_{ess}$ is continuous and has dense image. Since $C_{G}\cap H^{0}$ is dense in $C_{H}$ it follows from continuity that $\iota:C_{G}\cap H^{0}\to \tilde{H}^{0}_{ess}$ also has dense image. Therefore, every $\underline{x}\in \pi^{-1}_{ess, H}(X)$ is the limit (in the Fell topology of $H$) of a net $(\{u_{\lambda}\})\subseteq C_{G}\cap H^{0}$. As we have seen, if $(u_{\lambda})\subseteq H^{0}$ is a net converging to some $x\in X$, then $\{u_{\lambda}\}$ converges in the Fell topology of $H$ to $\underline{x}$ if and only if it converges in the Fell topology of $G$ to $\underline{x}$. Hence, $\underline{x}\in \pi^{-1}_{ess, G}(X)$. The other containment $\pi^{-1}_{ess, G}(X)\subseteq \pi^{-1}_{ess, H}(X)$ follows immediately from $C_{G}\cap H^{0}\subseteq C_{H}$ and $X\subseteq H^{0}$.
\end{proof}

\begin{prp}
\label{prp:germ&cover}
    Let $G_{1}, G_{2}$ be \'etale groupoids and $X_{1}$, $X_{2}$ locally invariant sets such that $(G_{1}, X_{1})\simeq_{\alpha} (G_{2}, X_{2})$. Then, $\tilde{\alpha}:\tilde{G}_{1}|_{\pi^{-1}(X_{1})}\to \tilde{G}_{2}|_{\pi^{-1}(X_{2})}$ is a groupoid isomorphism and homeomorphism, where $\tilde{\alpha}(Y) = \alpha(Y)$ (viewing $Y$ as a closed set in $G_{1}|_{X_{1}}$). Moreover, $\alpha(\pi^{-1}_{ess}(X_{1})) = \pi^{-1}_{ess}(X_{2})$.
\end{prp}
\begin{proof}
    let $H_{1}$, $H_{2}$ be local groupoids about $G_{1}|_{X_{1}}, G_{2}|_{X_{2}}$, respectively, such that $\alpha:H_{1}\to H_{2}$ is a groupoid isomorphism (and homeomorphism). Then, it is easy to see $\tilde{\alpha}:\tilde{H}_{1}\to \tilde{H}_{2}$ is a groupoid isomorphism (and homeorphism), where $\tilde{H}_{1},\tilde{H}_{2}$ are the Hausdorff covers of $H_{1}, H_{2}$ and $\tilde{\alpha}$ is $\alpha$ extended to closed sets of $H_{1}$, $H_{2}$. Moreover, $\tilde{\alpha}(\tilde{H}_{1,ess}) = \tilde{H}_{2,ess}$ and $\alpha(\pi_{H_{1}}^{-1}(X_{1})) = \pi^{-1}_{H_{2}}(X_{2})$. By Lemma \ref{lem:germ&cover}, we have $\tilde{G_{i}}|_{\pi_{G_{i}}^{-1}(X_{i})} = \tilde{H}_{i}|_{\pi_{H_{i}}^{-1}(X_{i})}$ with equal topologies and $\pi_{G_{i},ess}^{-1}(X_{i}) = \pi^{-1}_{H_{i},ess}(X_{i})$ for $i=1,2$, proving the proposition.
\end{proof}

Now, we consider the functorial properties of locally invariant sets.

\begin{prp}
    Suppose $X$ and $Y$ are closed locally invariant sets such that $X\subseteq Y$, then $X$ is locally invariant in $G|_{Y}$, the norm $\rho(Y)(X) = \rho(X)$ and the diagram

    $$
\begin{tikzcd}
C^{*}_{\rho}(G) \arrow[rd, "\eta_{X}"] \arrow[r, "\eta_{Y}"] & C^{*}_{\rho(Y)}(G|_{Y}) \arrow[d, "\eta_{X}"] \\
                                                             & C^{*}_{\rho(X)}(G|_{X})                      
\end{tikzcd}$$ commutes.
\end{prp}

\begin{proof}
    By the construction in Proposition \ref{prp:germexists}, we can choose a local groupoid $H$ about $G|_{Y}$ that contains a local groupoid  $K$ about $G|_{X}$ with $H^{0} = K^{0} = G^{0}$. Then, $K|_{Y}\subseteq G|_{Y}$ is a local groupoid about $G|_{X}$. Hence, $X$ is locally invariant in $G|_{Y}$ and for $c\in G|_{X}$, we have 
    \begin{align*}
        \|c\|_{\rho(X)} = \inf\{\|a\|_{\rho}:a\in\mathscr{C}_{c}(K),\text{ }\eta_{X}(a) = c\} =
        \\\inf\{\|a\|_{\rho}:a\in \mathscr{C}_{c}(K), b\in \mathscr{C}_{c}(K|_{Y}),\text{ }\eta_{Y}(a) = b,\text{ }\eta_{X}(b) = c\} = 
        \\\inf\{\|b\|_{\rho(Y)}:b\in \mathscr{C}_{c}(K|_{Y}),\text{ }\eta_{X}(b) = c\} = \|c\|_{\rho(Y)(X)}.
    \end{align*}

    Hence, $\rho(Y)(X) = \rho(X)$. Since the diagram commutes on the canonical pre-C*-algebras and $\rho(Y)(X) = \rho(X)$ it commutes on the C*-completions as claimed.
\end{proof}

\begin{cor}
\label{cor:commdiagrest}
    Let $G$ be an \'etale groupoid, $X\subseteq G^{0}$ a closed locally invariant set and $Y$ a closed invariant set. Then, $X\cap Y$ is a closed locally invariant set in $G|_{Y}$ and an invariant set in $G|_{X}$. The diagram

    $$
\begin{tikzcd}
C^{*}_{r}(G) \arrow[r, "q_{Y}"] \arrow[d, "\eta_{X}"] & C^{*}_{r}(G|_{Y}) \arrow[d, "\eta_{X\cap Y}"] \\
C_{r(X)}^{*}(G|_{X}) \arrow[r, "q_{X\cap Y}"]           & C_{r(X\cap Y)}^{*}(G|_{X\cap Y})          
\end{tikzcd}$$ commutes, where the horizontal mappings are *-homomorphisms. 

Moreover, we have $\eta_{X}(\text{ker}(q_{Y})) = \text{ker}(q_{X\cap Y}).$
\end{cor}

\begin{proof}
    By Theorem \ref{thm:locinvar&compress}, the restrictions of $\mathcal{I} = \ker(\eta_{X}:\mathscr{C}_{c}(G)\to \mathscr{C}_{c}(G|_{X}))$ and $\mathcal{J} = \ker(\eta_{X\cap Y}:\mathscr{C}_{c}(G|_{Y})\to \mathscr{C}_{c}(G|_{X\cap Y}))$ to the appropriate compressible domains contain approximate units, and are therefore dense in $\ker(\eta_{X})$ and $\ker(\eta_{X\cap Y})$ by Corollary \ref{cor:prekerdense}. Therefore, $q_{Y}(\mathcal{I})\subseteq \mathcal{J}$ implies $q_{Y}(\ker(\eta_{X}))\subseteq \ker(\eta_{X\cap Y})$.
    
    By this containment and the fact that the norms $\|\cdot\|_{r(X)}$ and $\|\cdot\|_{r(X\cap Y)}$ are the Banach quotient norms by the closed subspaces $\ker(\eta_{X}))$ and $\ker(\eta_{X\cap Y})$ (Equation \ref{eq:norm2} and the approximate unit fact in Theorem \ref{thm:locinvar&compress}), it follows that $q_{X\cap Y}$ (defined on the dense sub-algebras) extends to a *-homomorphism $q_{X\cap Y}: C^{*}_{r(X)}(G|_{X})\to C^{*}_{r(X\cap Y)}(G|_{X\cap Y})$, and the proposed diagram commutes (it commutes on the canonical dense sub-algebras). If $(u_{\lambda})$ is an approximate unit for $C_{c}(G^{0}\setminus X)$, then it is easy to see $(q_{Y}(u_{\lambda}) = u_{\lambda}|_{Y})$ is an approximate unit for $C_{c}(Y\setminus X\cap Y)$. Hence, Theorem \ref{thm:kernelsurjects} shows $\eta_{X}(\ker(q_{Y})) = \ker(q_{X\cap Y})$.
\end{proof}

\begin{prp}
\label{prp:commdiagcover}
    Let $G$ be an \'etale groupoid and $X\subseteq G^{0}$ a closed locally invariant set. Let $\iota:C^{*}_{r}(G)\to C^{*}_{r}(\tilde{G})$ be the inclusion. Then, the diagram

    $$ \begin{tikzcd}
C^{*}_{r}(G) \arrow[r, "\iota"] \arrow[d, "\eta_{X}"] & C^{*}_{r}(\tilde{G}) \arrow[d, "\eta_{\mathcal{X}}"] \\
C_{r(X)}^{*}(G|_{X}) \arrow[r, "\iota_{X}"]           & C^{*}_{r(\mathcal{X})}(\tilde{G}|_{\mathcal{X}})    
\end{tikzcd}$$ commutes, where $\mathcal{X} = \pi^{-1}(X)$ and $\iota_{X}$ is the *-homomorphism defined, for $f\in \mathscr{C}_{c}(G|_{X})$, as  $\iota_{X}(f)(\underline{g}) = \sum_{g\in\underline{g}}f(g)$, for $\underline{g}\in \tilde{G}|_{\mathcal{X}} = (G|_{X})\mathcal{X}$.

Moreover, $\iota_{X}$ is injective.
    
\end{prp}

\begin{proof}

By a similar argument as in Corollary \ref{cor:commdiagrest}, we have that $\iota_{X}:\mathscr{C}_{c}(G|_{X})\to C_{c}(\tilde{G}|_{\mathcal{X}})$ extends continuously to $\iota_{X}:C^{*}_{r(X)}(G)\to C^{*}_{r(\mathcal{X})}(\tilde{G}|_{\mathcal{X}})$. Therefore, the diagram in the proposition commutes.

If $(u_{\lambda})$ is an approximate unit for $C_{c}(G^{0}\setminus X)$, then $(\iota(u_{\lambda}) = u_{\lambda}\circ \pi)$ is an approximate unit for $C_{c}(\tilde{G}^{0}\setminus \pi^{-1}(X))$. By Theorem \ref{thm:locinvar&compress}, these are approximate units for $\ker(\eta_{X})$ and $
\ker(\eta_{\mathcal{X}})$, respectively, so Theorem \ref{thm:kernelsurjects} implies $ 0 = \eta_{X}(\ker(\iota)) = \ker(\iota_{X})$, proving that $\iota_{X}$ is injective.

\end{proof}

\end{section}

\section{Isotropy fibres of the singular ideal - characterization of vanishing and the intersection property}
\label{s:isofibre}
In this section we prove our characterization of vanishing singular ideal by computing its isotropy fibres in general. By Proposition \ref{prp:commdiagcover} we know that $C^{*}_{r(x)}(G^{x}_{x})$ embeds into $C^{*}_{r(\mathcal{X})}(\tilde{G}|_{\mathcal{X}})$, where $\mathcal{X} = \pi^{-1}(x)$ and it will be useful to describe $\tilde{G}|_{\mathcal{X}}$, which we do in Section \ref{ss:cosetgpds} below. Section \ref{ss:isofibres} contains our main theorem (the calculation of isotropy fibres) using the theory we have built in the previous sections. In Section \ref{ss:charvsing} we apply this calculation to obtain various characterizations of vanishing of the singular ideal and Section \ref{ss:isofibrealgsing} characterizes vanishing of $J\cap\mathscr{C}_{c}(G)$, and the singular ideal intersection property $J\neq\{0\}\implies J\cap\mathscr{C}_{c}(G)\neq \{0\}$ in terms of an ideal intersection property of the isotropy groups.
\subsection{Coset groupoids}
\label{ss:cosetgpds}
We introduce the construction of a \emph{coset groupoid}. They will be important later as they correspond to certain reductions of the Hausdorff cover groupoid as mentioned above. We will also use them in Section \ref{s:nonHDconstruct} to construct non-Hausdorff groupoids with prescribed singular ideals.

 Let $\Gamma$ be a discrete group and view a subset $X\subseteq \Gamma$ as a sequence $X\in \{0,1\}^{\Gamma}$ via its characteristic function $1_{X}:\Gamma\to \{0,1\}$. Let $\mathcal{X}\subseteq \{0,1\}^{\Gamma}$ be a closed set of subgroups invariant under conjugation. We can equip the set of cosets $\Gamma\mathcal{X} = \{\gamma X:\gamma\in \Gamma, X\in\mathcal{X}\}\subseteq \{0,1\}^{\Gamma}$ (with the subspace topology) an \'etale groupoid structure. 

First note that for any $Y\in \Gamma\mathcal{X}$, $s(Y):= y^{-1}Y$ and $r(Y):= Yy^{-1}$ are in $\mathcal{X}$ by conjugation invariance, are independent of the choice $y\in Y$ and define continuous maps $r,s:\Gamma\mathcal{X}\to \mathcal{X}$. Second,  we say $Y, Z\in \Gamma\mathcal{X}$ are composable if and only if $y^{-1}Y = Zz^{-1}$, in which case their product is defined as the pointwise product $YZ = yzs(Z)\in \Gamma\mathcal{X}$. The inverse of $Y\in G\mathcal{X}$ is defined as the pointwise inverse $Y^{-1} = yr(Y)\in \Gamma\mathcal{X}$. These operations are obviously continuous.

Also, it is easy to see from these operations that the groupoid range and source maps co-incide with $r,s$ defined above, and hence the unit space $(\Gamma\mathcal{X})^{0} = \mathcal{X}$.

We have $\mathcal{X}= \{X\in \Gamma\mathcal{X}: e\in X\}$ and therefore $\mathcal{X}$ is clopen in $\Gamma\mathcal{X}$. More generally, given a clopen set $\mathcal{U}\subseteq \mathcal{X}$, $\gamma\mathcal{U} = \{Y\in \Gamma\mathcal{X}: \gamma\in Y\text{ and } \gamma^{-1}Y\in \mathcal{U}\}$ is a clopen bisection. Hence, $\Gamma\mathcal{X}$ is a locally compact Hausdorff \'etale groupoid with totally disconnected unit space.

let $r\mathcal{X}$ denote the semi-C*-norm on $\mathbb{C}[\Gamma]$ which is the supremum of operator norms from the quasi-regular representations $\lambda_{\Gamma/X}:\mathbb{C}[\Gamma]\to B(\ell^{2}(\Gamma/X))$, for each $X\in\mathcal{X}$, defined by $\lambda_{\Gamma/X}(\delta_{\gamma}):\delta_{hX}\mapsto \delta_{\gamma hX}$, for $\gamma\in\Gamma$.

If $\mathcal{X}$ contains the identity, then $r\mathcal{X}\geq r$ obviously and so the semi-C*-norm is a C*-norm. Denote by $C^{*}_{r\mathcal{X}}(\Gamma)$ the C*-algebra induced from this semi-C*-norm. We will use the notation $\hat{\mathcal{X}} = \mathcal{X}\cup\{e\}$ throughout this paper, since the kernel of the quotient $q_{e}:C^{*}_{r\hat{\mathcal{X}}}(\Gamma)\to C^{*}_{r\mathcal{X}}(\Gamma)$ is related to the isotropy fibres of singular ideals.

Each $X\in \mathcal{X}$ defines a positive linear functional $a\in C^{*}_{r\hat{\mathcal{X}}}(\Gamma)\mapsto a(X) = \langle \delta_{X},\lambda_{X}(a)\delta_{X}\rangle$. More generally for $Y\in \Gamma\mathcal{X}$ and $y\in Y$, we can define $a\mapsto a(Y) = (\delta_{y}^{-1}*a)(y^{-1}Y)$, which is a continuous linear functional on $C^{*}_{r\hat{\mathcal{X}}}(\Gamma)$ independent of the representative $y\in Y$. For $a\in \mathbb{C}[\Gamma]$, we have $a(Y) = \sum_{y\in Y}a(y)$.

Clearly $q(a):Y\mapsto a(Y)$ is continuous for $a\in\mathbb{C}[\Gamma]$ (it is locally constant). Its support lies in $\bigcup_{\gamma:a(\gamma)\neq 0}\gamma\mathcal{X}$ and thus $q(a)\in C_{c}(\Gamma\mathcal{X})$. Moreover, it is an easy check that  $q: \mathbb{C}[\Gamma]\to C_{c}(\Gamma\mathcal{X})$ is a *-homomorphism. The left regular representation $\lambda_{X}$ of $C_{c}(\Gamma\mathcal{X})$ associated to a unit $X\in\mathcal{X}$ composed with $q$ is canonically unitarily equivalent to $\lambda_{\Gamma/X}$. Hence, there is an *-homomorphism $q:C^{*}_{r\hat{\mathcal{X}}}(\Gamma)\to C^{*}_{r}(\Gamma\mathcal{X})$ given by $q(a)(Y) = a(Y)$, for all $a\in C^{*}_{\mathcal{X}}(\Gamma)$ and $Y\in G\mathcal{X}$ and the kernel of $q$ is equal the kernel of $q_{e}: C^{*}_{r\hat{\mathcal{X}}}(\Gamma)\to C^{*}_{r\mathcal{X}}(\Gamma)$ mentioned above. We will denote this kernel in the following way.

\begin{dfn}
\label{dfn:gpideal}
    Let $\Gamma$ be a discrete group and $\mathcal{X}$ a closed set of subgroups invariant under conjugation. We denote by $J_{\Gamma,\mathcal{X}}$ the ideal of the quotient map $C^{*}_{r\hat{\mathcal{X}}}(\Gamma)\to C^{*}_{r\mathcal{X}}(\Gamma)$. Note that this is equal to $\bigcap_{X\in\mathcal{X}}\ker(\lambda_{\Gamma/X})$ inside $C^{*}_{r\hat{\mathcal{X}}}(\Gamma)$, which is equal to the kernel of $q:C^{*}_{r\hat{\mathcal{X}}}(\Gamma)\to C^{*}_{r}(\Gamma\mathcal{X})$.
\end{dfn}
Let us characterize precisely when this ideal is zero.

\begin{lem}
\label{lem:groupidealchar}
    Let $\Gamma$ be a discrete group and $\mathcal{X}$ a closed set of subgroups invariant under conjugation. Then, $J_{\Gamma,\mathcal{X}} = 0$ if and only if $\lambda_{\Gamma}$ is weakly contained in $\oplus_{X\in\mathcal{X}}\lambda_{\Gamma/X}$.
\end{lem}
\begin{proof}
    The quotient $C^{*}_{r\hat{\mathcal{X}}}(\Gamma)\to C^{*}_{r\mathcal{X}}(\Gamma)$ is injective if and only if $\ker(\lambda_{\Gamma})\cap \bigcap_{X\in\mathcal{X}} \ker(\lambda_{\Gamma/X})$ $ = \bigcap_{X\in\mathcal{X}} \ker(\lambda_{\Gamma/X})$ in the full group C*-algebra $C^{*}(\Gamma)$. This is true if and only if $\bigcap_{X\in\mathcal{X}} \ker(\lambda_{\Gamma/X})\subseteq \ker(\lambda_{\Gamma})$, which is true if and only if (by \cite[Theorem~F.4.4]{BHV07}) $\lambda_{\Gamma}$ is weakly contained in $\oplus_{X\in\mathcal{X}}\lambda_{\Gamma/X}$.
\end{proof}

Now, let's characterize when the ideal intersect its group ring vanishes.

\begin{lem}
\label{lem:alggroupidealchar}
Let $\Gamma$ be a discrete group and $\mathcal{X}$ a closed set of subgroups invariant under conjugation. Then, the following are equivalent.

\begin{enumerate}
    \item $J_{\Gamma,\mathcal{X}}\cap \mathbb{C}[\Gamma]\neq \{0\}$.
    \item There is a non-zero element $a\in \mathbb{C}[\Gamma]$ such that $\sum_{x\in X}a(\gamma x) = 0$ for all $\gamma\in \Gamma$ and $X\in\mathcal{X}$.
    \item There is a non-zero element $a\in \mathbb{Z}[\Gamma]$ such that $\sum_{x\in X}a(\gamma x) = 0$ for all $\gamma\in \Gamma$ and $X\in\mathcal{X}$.
    \item There is a finite set $F\subseteq \Gamma$ such that the vectors $\delta_{fX\cap F}:= \sum_{h\in fX\cap F}\delta_{h}$, for $f\in F$ and $X\in\mathcal{X}$, do not linearly span $\mathbb{C}[F]$.
\end{enumerate}
\end{lem}

\begin{proof}
The condition in $(2)$ is equivalent (by conjugation invariance) to $\sum_{x\in X}a(\gamma_{2} x \gamma^{-1}_{1}) = 0$ for all $\gamma_{1},\gamma_{2}\in \Gamma$ and $X\in\mathcal{X}$, which is equivalent to $\langle \lambda_{\Gamma/X}(a)\delta_{\gamma_{1}X},\delta_{\gamma_{2} X}\rangle = 0 $ for all $\gamma\in\Gamma$ and $X\in\mathcal{X}$, so the equivalence between $(1)$ and $(2)$ is immediate.

Obviously $(3)$ implies $(2)$, so let's show the converse. Let $a\in\mathbb{C}[\Gamma]$ be as in $(2)$. Set $F = \{\gamma\in \Gamma: a(\gamma)\neq 0\}$, which is finite. Let $K$ be the collection of non-empty sets of the form $\gamma X\cap F$ for some $\gamma\in \Gamma$ and $X\in\mathcal{X}$. Then, each $\gamma X\cap F\in K$ can be written as $\gamma X\cap F = fX\cap F$ for some $f\in F$. Therefore, $K$ is finite. Let $A:\mathbb{C}^{F}\to \mathbb{C}^{K}$ be the matrix specified by $A_{\gamma, fX\cap F} = 1$ if $\gamma\in fX\cap F$, and zero otherwise. Then, $a' = (a(g))_{g\in F}\in\mathbb{C}^{F}$ satisfies $A(a') = 0$. Using Gaussian elimination, it is easy to see an integer valued matrix has a non-zero complex solution to $A(x) = 0$ if and only if it has a non-zero rational solution $x = b'$. By clearing denominators we may assume $b'\in \mathbb{Z}^{F}$. Setting $b = \sum_{g\in F}b'(g)\delta_{g}$, it is easy to see $b$ satisfies $(3)$. 

Now, we show $(2)$ implies $(4)$. Let $a\in\mathbb{C}[\Gamma]$ be as in $(2)$. Set $F = \{\gamma\in \Gamma: a(\gamma)\neq 0\}$ and let $\hat{a}:\mathbb{C}[F]\to\mathbb{C}$ be the linear functional defined, for $\delta_{f}\in \mathbb{C}[F]$ as $\hat{a}(\delta_{f}) = a(f)$ and extend this definition linearly. Then, for $f\in F$ and $X\in\mathcal{X}$, we have $\hat{a}(\delta_{fX\cap F}) = \sum_{h\in fX\cap F}a(h) = \sum_{x\in X}a(\gamma x) = 0$. We have produced a non-zero linearly functional on $\mathbb{C}[F]$ that annihilates the subspace $\text{span}\{\delta_{fX\cap F}: f\in F,X\in\mathcal{X}\}$. Therefore, $\text{span}\{\delta_{fX\cap F}: f\in F,X\in\mathcal{X}\}\neq \mathbb{C}[F]$.

    The proof of $(4)$ implies $(2)$ follows similarly; let $\hat{b}:\mathbb{C}[F]\to \mathbb{C}$ be a non-zero linear functional that annihilates the subspace $\text{span}\{\delta_{fX\cap F}: f\in F,X\in\mathcal{X}\}\neq \mathbb{C}[F]$. Then, define $b\in\mathbb{C}[\Gamma]$ as $b(g) = \hat{b}(\delta_{g})$ for $g\in F$ and zero otherwise. For $\gamma\in \Gamma$ and $X\in\mathcal{X}$, if $\gamma X\cap F = \emptyset$, then $\sum_{x\in X}b(\gamma x) = 0$. If $f\in \gamma X\cap F$, then $\gamma X = f X$ and $\sum_{x\in X}b(\gamma x) = \sum_{h\in f X\cap F}b(h) = \hat{b}(\delta_{fX\cap F}) = 0$.
\end{proof}

 Given an \'etale groupoid $G$ and $x\in G^{0}$, let's identify the groupoid $\tilde{G}|_{\pi^{-1}(x)}$. 

\begin{prp}
\label{prp:rest=cosetgpd}
    Let $G$ be an \'etale groupoid and $x\in G^{0}$. Then, $\tilde{G}|_{\pi^{-1}(x)} = G_{x}^{x}\cdot \pi^{-1}(x)$, with the subspace topology arising from $\tilde{G}$ is equal to the topology of the coset groupoid defined above.
\end{prp}

\begin{proof}
    For $gX\in \tilde{G}|_{\pi^{-1}(x)}$, we have $s(g) = \pi(X) = x$ and $r(g) = \pi(gXg^{-1}) = x$, so as sets we have $\tilde{G}|_{\pi^{-1}(x)} = G^{x}_{x}\cdot \pi^{-1}(x)$. It is easy to see they are the same groupoid, so it remains to show that $\pi^{-1}(x)$ is closed in $\{0,1\}^{G^{x}_{x}}$ and the topologies of the groupoids are equal.
    
    Let's first show the subspace topology induced from $\tilde{G}^{0}$ is the same as the subspace topology induced from $\{0,1\}^{G^{x}_{x}}$. Since $\pi^{-1}(x)\subseteq \tilde{G}^{0}$ is compact ($\pi$ is proper, see \cite[Section~3]{BGHL25}) and $\pi^{-1}(x)\subseteq \{0,1\}^{\Gamma}$ is Hausdorff , it suffices to show $\text{id}:\pi^{-1}(x)\subseteq \tilde{G}^{0}\to \pi^{-1}(x)\subseteq \{0,1\}^{\Gamma}$ is continuous (this will also show $\pi^{-1}(x)$ is closed in $\{0,1\}^{\Gamma}$). We will use the description of the topology of $\tilde{G}$ as in Proposition \ref{prp:fell}. Suppose $(X_{\lambda})\subseteq \pi^{-1}(x)$ is a net converging to $X$ in the Fell topology. It follows that for every $x\in X$, there is a net $x_{\lambda}\in X_{\lambda}$ such that $(x_{\lambda})$ converges to $x$ in $G$. Since $(x_{\lambda})\subseteq G^{x}_{x}$ and $G^{x}_{x}$ is discrete, we have $x_{\lambda} = x$ eventually and therefore $1_{X_{\lambda}}(x) = 1_{X}(x) = 1$ eventually. Now, suppose $x\notin X$. Since $\text{Acc}(X_{\lambda}) = \lim(X_{\lambda}) = X$, it follows that there is no subnet $x_{\lambda_{\mu}}\in X_{\lambda_{\mu}}$ such that $x_{\lambda_{\mu}} = x$ eventually. Therefore, we must have $1_{X_{\lambda}}(x) = 0 = 1_{X}(x)$ eventually (otherwise we can extract such a subnet). Hence, $(1_{X_{\lambda}})$ converges to $1_{X}$.

    Now, In each groupoid $g\cdot \pi^{-1}(x)$ for $g\in G^{x}_{x}$ is a clopen bisection and since we know the topologies are equal on the unit space, it follows that the topologies are equal for the groupoid, proving the proposition.
\end{proof}

A similar description also holds when $x\in G^{0}$ is replaced with a finite set.

\subsection{Isotropy fibres of the Singular ideal}
\label{ss:isofibres}

Our main theorem of this section follows from the general theory of compressible maps we established previously.

\begin{thm}
\label{thm:isofibre}
    Let $G$ be an \'etale groupoid, $x\in G^{0}$, $\mathcal{X} = \pi^{-1}_{ess}(x)$ and $J$ its singular ideal. Let $H$ be a local groupoid about $G^{x}_{x}$. Then, 
    
    $$J_{x} := \eta_{x}(J) = \eta_{x}(J\cap C^{*}_{r}(H)) = \ker(q_{x}),$$ where $q_{x}:C^{*}_{r(x)}(G^{x}_{x})\to C^{*}_{r(\mathcal{X})}(G^{x}_{x}\cdot \mathcal{X})$ is the *-homomorphism defined, for $a\in \mathbb{C}[G_{x}^{x}]$ as $$q_{x}(a)(Y) = \sum_{y\in Y}a(y), \text{ for all }Y\in G^{x}_{x}\cdot\mathcal{X},$$ and extended to the completions. 
    
    Moreover, if $r(\mathcal{X})$ denotes the (semi-C*-norm) on $\mathbb{C}[G^{x}_{x}]$ determined by $q_{x}$ and $\widehat{r(\mathcal{X})}$ denotes the C*-norm $\max\{\|\cdot\|_{r(\mathcal{X})}, \|\cdot\|_{r}\}$, then 

    $$p(J_{x}) = \ker(C^{*}_{\widehat{r(\mathcal{X})}}(G^{x}_{x})\to C^{*}_{r(\mathcal{X})}(G^{x}_{x})),$$ where $p:C^{*}_{r(x)}(G^{x}_{x})\to C^{*}_{\widehat{r(\mathcal{X}})}(G^{x}_{x})$ is the projection.
\end{thm}

\begin{proof}

Let $\mathcal{Y} = \pi^{-1}(x)$ and $\mathcal{X} = \pi^{-1}_{ess}(x)$. By Proposition \ref{prp:germexists},  we can choose a local groupoid $H$ about $G^{x}_{x}$ with $H^{0} = G^{0}$, and $\hat{H} = H\cdot \tilde{G}^{0}$, $\hat{H}_{ess} = H\cdot\tilde{G}_{ess}^{0}$ are local groupoids about $\pi^{-1}(x)$ and $\pi^{-1}_{ess}(x)$ by Proposition \ref{prp:locinvar&cover}. From Proposition \ref{prp:commdiagcover} and Corollary \ref{cor:commdiagrest}, the diagrams

$$
\begin{tikzcd}
C^{*}_{r}(G) \arrow[r, "\iota"] \arrow[d, "\eta_{x}"] & C^{*}_{r}(\tilde{G}) \arrow[d, "\eta_{\pi^{-1}(x)}"] & C^{*}_{r}(\tilde{G}) \arrow[d, "\eta_{\pi^{-1}(x)}"] \arrow[r, "q_{G^{0}_{ess}}"]    & C^{*}_{r}(\tilde{G}_{ess}) \arrow[d, "\eta_{\pi^{-1}_{ess}(x)}"] \\
C^{*}_{r(x)}(G^{x}_{x}) \arrow[r, "\iota_{x}"]        & C^{*}_{r(\mathcal{Y})}(G^{x}_{x}\cdot\pi^{-1}(x))    & C^{*}_{r(\mathcal{Y})}(G^{x}_{x}\cdot\pi^{-1}(x)) \arrow[r, "q_{\pi^{-1}_{ess}(x)}"] & C^{*}_{r(\mathcal{X})}(G^{x}_{x}\cdot\pi_{ess}^{-1}(x))         
\end{tikzcd}
$$ commute. Moreover, $\iota(C^{*}_{r}(H))\subseteq C^{*}_{r}(\hat{H})$ and $q_{G^{0}_{ess}}(C^{*}_{r}(\hat{H})) = C^{*}_{r}(\hat{H}_{ess})$.

Note that $q_{x} = q_{\pi^{-1}_{ess}(x)}\circ \iota_{x}$, so letting $\iota_{ess} = q_{G^{0}_{ess}}\circ \iota$, we see from the above diagrams that the diagram

    $$
\begin{tikzcd}
C^{*}_{r}(G) \arrow[r, "\iota_{ess}"] \arrow[d, "\eta_{x}"] & C^{*}_{r}(\tilde{G}_{ess}) \arrow[d, "\eta_{\pi^{-1}_{ess}(x)}"] \\
C^{*}_{r(x)}(G^{x}_{x}) \arrow[r, "q_{x}"]                        & C^{*}_{r(\mathcal{X})}(G^{x}_{x}\cdot\pi^{-1}_{ess}(x))         
\end{tikzcd}$$ commutes. By Theorem \ref{thm:locinvar&compress}, if $(u_{\lambda})\subseteq C_{c}(G^{0}\setminus \{x\})$ is an approximate unit for $\ker(\eta_{x}|_{C^{*}_{r}(H)})$, then $(\iota_{ess}(u_{\lambda}) = u_{\lambda}\circ \pi_{ess})\subseteq C_{c}(\tilde{G}^{0}_{ess}\setminus \pi^{-1}_{ess}(x))$ is an approximate unit for $\ker(\eta_{\pi^{-1}_{ess}(x)}|_{C^{*}_{r}(\hat{H}_{ess})})$. Therefore, by Theorem \ref{thm:kernelsurjects}, we have $\eta_{x}(\ker(\iota_{ess})) = \ker(q_{x})$. By \cite[Proposition~4.15]{BGHL25}, we know $\ker(\iota_{ess}) = J$, so we have proven $J_{x} = \ker(q_{x})$. Moreover, since $\iota_{ess}(C^{*}_{r}(H))\subseteq C^{*}_{r}(\hat{H}_{ess})$, we have $\eta_{x}(J\cap C^{*}_{r}(H)) = \ker(q_{x})$ by Theorem \ref{thm:kernelsurjects}. Note that if $K$ is a local groupoid about $G^{x}_{x}$, then so is $H = K\cup G^{0}$, and $\eta_{x}(J\cap C^{*}_{r}(K)) = \eta_{x}(J\cap C^{*}_{r}(H)) = \ker(q_{x})$. This finishes the proof of first part of the Theorem.

Now, we prove the ``moreover'' clause. Set $\hat{H}'_{ess} = \hat{H}_{ess}\cup G^{x}_{x} = H\cdot (\tilde{G}_{ess}^{0}\cup \{x\})$ and let $\iota_{ess}':C^{*}_{r}(H)\to C^{*}_{r}(\hat{H}'_{ess})$ be the *-homomorphism defined for $a\in \mathscr{C}_{c}(H)$ as $\iota_{ess}'(a)(\underline{h}) = \sum_{h\in \underline{h}}a(g)$, for all $\underline{h}\in \hat{H}'_{ess}$. Similarily, define $q'_{ess}:C^{*}_{r}(\hat{H}'_{ess})\to C^{*}(\hat{H}_{ess})$ as $q'_{ess}(b) = b|_{\hat{H}_{ess}}$ for $b\in C^{*}_{r}(\hat{H}'_{ess})$. Either $\{x\}\in \tilde{G}_{ess}^{0}$, in which case $\hat{H}'_{ess} = \hat{H}_{ess}$, or $\{x\}\notin \tilde{G}_{ess}^{0}$. In the latter case, we have $C^{*}_{r}(\hat{H}'_{ess}) = C^{*}_{r}(\hat{H}_{ess})\oplus C^{*}_{r}(G^{x}_{x})$, with $q'_{ess}$ the projection onto the first direct summand. In either case, we have $\iota_{ess} = q'_{ess}\circ \iota'_{ess}$ and the diagram

$$
\begin{tikzcd}
C^{*}_{r}(H) \arrow[d, "\eta_{x}"] \arrow[r, "\iota'_{ess}"] & C^{*}_{r}(\hat{H}'_{ess}) \arrow[d, "\eta_{\hat{\mathcal{X}}}"] \arrow[r, "q'_{ess}"]      & C^{*}_{r}(\hat{H}_{ess}) \arrow[d, "\eta_{\mathcal{X}}"] \\
C^{*}_{r(x)}(G^{x}_{x}) \arrow[r, "\iota'_{x}"]               & C_{r(\hat{\mathcal{X}})}^{*}(G^{x}_{x}\cdot\hat{\mathcal{X}}) \arrow[r, "q_{\mathcal{X}}"] & C^{*}_{r(\mathcal{X})}(G^{x}_{x}\cdot\mathcal{X})       
\end{tikzcd}$$ commutes, where the *-homomorphism $\iota'_{x}$ is defined for $a\in\mathbb{C}[G^{x}_{x}]$ and $Y\in G^{x}_{x}\cdot\hat{\mathcal{X}}$ as $\iota'_{x}(a)(Y) = \sum_{y\in Y}a(y)$. The map $\iota'_{x}$ extends to the C*-completion by a similar argument to that of Corollary \ref{cor:commdiagrest}. The *-homomorphism $q_{\mathcal{X}}$ is the reduction map and it is well defined by Corollary \ref{cor:commdiagrest}. Since $q_{\mathcal{X}}\circ \iota'_{x} = q_{x}$, we have $\iota'_{x}(\ker(q_{x})) = \iota'_{x}(\iota'^{-1}_{x}(\ker(q_{\mathcal{X}}))) = \ker(q_{\mathcal{X}})\cap \text{im}(\iota'_{x})$.

 In the case that $\{x\}\in \pi^{-1}_{ess}(x)$, since $\hat{H}'_{ess} = \hat{H}$, we have  $C^{*}_{r(\hat{\mathcal{X}})}(G^{x}_{x}\cdot\mathcal{X}) = C^{*}_{r(\mathcal{X})}(G^{x}_{x}\cdot\mathcal{X})$. Now, assume $\{x\}\notin \pi^{-1}_{ess}(x)$. We have $C^{*}_{r}(\hat{H}'_{ess}) = C^{*}_{r}(\hat{H}_{ess})\oplus C^{*}_{r}(G^{x}_{x})$ and $q'_{ess}$ is the projection on the first summand. By the same argument as in the proof of the first part of the theorem, we have $\eta_{\hat{\mathcal{X}}}(\ker(q'_{ess})) = \ker(q_{\mathcal{X}})$. Combining this with the description of $q'_{ess}$, we have that $C^{*}_{r(\hat{\mathcal{X}})}(G^{x}_{x}\cdot\hat{\mathcal{X}}) = C^{*}_{r(\mathcal{X})}(G^{x}_{x}\cdot\mathcal{X})\oplus C^{*}_{r}(G^{x}_{x})$, with $q_{\mathcal{X}}$ the projection onto the first summand. In either case, the diagram.

$$
\begin{tikzcd}
C^{*}_{r(x)}(G^{x}_{x}) \arrow[r, "\iota'_{x}"] \arrow[d, "p"]                        & C_{r(\hat{\mathcal{X}})}^{*}(G^{x}_{x}\cdot\hat{\mathcal{X}}) \arrow[r, "q_{\mathcal{X}}"] & C^{*}_{r(\mathcal{X})}(G^{x}_{x}\cdot\mathcal{X}) \\
C^{*}_{\widehat{r(\mathcal{X})}}(G^{x}_{x}) \arrow[ru, "i", hook] \arrow[r] & C^{*}_{r(\mathcal{X})}(G^{x}_{x}) \arrow[ru, "j", hook]                               &                                                  
\end{tikzcd}$$ commutes, where $i$ and $j$ are the natural embeddings. Therefore, $$i(p(J_{x})) = \iota'_{x}(\ker(q_{x})) = \ker(q_{\mathcal{X}})\cap \text{im}(\iota'_{x}).$$  A simple diagram chase shows

$$\ker(q_{\mathcal{X}})\cap \text{im}(\iota'_{x})  = i(\ker(C^{*}_{\widehat{r(\mathcal{X})}}(G^{x}_{x})\to C^{*}_{r(\mathcal{X})}(G^{x}_{x}))),$$ and therefore $i(p(J_{x})) = i(\ker(C^{*}_{\widehat{r(\mathcal{X})}}(G^{x}_{x})\to C^{*}_{r(\mathcal{X})}(G^{x}_{x})))$, proving the theorem.
\end{proof}

A similar result (and proof) holds for any closed locally invariant set $X$.

Now, we will begin to prove our characterization of the singular ideal vanishing in terms of a weak containment property. First, we identify when the isotropy fibre vanishes for trivial reasons.

Neshveyev and Schwartz in \cite[Proposition~1.12]{NS23} say $x\in G^{0}$ is \emph{extremely dangerous} if there exists $\{g_{i}\}^{n}_{i=1}\subseteq G^{x}_{x}\setminus \{x\}$ and bisections $U_{i}$ containing $g_{i}$, for $i\leq n$, such that $U\setminus \bigcup^{n}_{i=1}U_{i}$ has empty interior for some open set $U\subseteq G^{0}$ containing $x$. Denote the collection of extremely dangerous points as $D_{0}$. It turns out the isotropy fibres of the singular ideal vanish off the extremely dangerous points.

\begin{cor}
\label{cor:edangptchar}
    Let $G$ be an \'etale groupoid. Then, $x\in D_{0}$ if and only if $\{x\}\notin \pi^{-1}_{ess}(x)$. Consequently, $p(J_{x}) = 0$ for $x\notin D_{0}$. 
\end{cor}
\begin{proof}
    By the paragraph below \cite[Definition~4.12]{BGHL25}, $x\in D_{0}$ if and only if $\iota(x)\in \text{int}(\pi^{-1}(D))$. Since $\tilde{G}^{0}_{ess} = \overline{\iota(C)}$ it is easy to see that $\text{int}(\pi^{-1}(D)) = \tilde{G}^{0}\setminus \tilde{G}^{0}_{ess}$. Therefore, $x\in D_{0}$ if and only if $\iota(x) = \{x\}\notin \tilde{G}^{0}_{ess}$, if and only if $\{x\}\notin \pi^{-1}(x)\cap \tilde{G}^{0}_{ess} = \pi^{-1}_{ess}(x)$.

    So, if $x\notin D_{0}$, then $\{x\}\in \pi^{-1}_{ess}(x)$ and hence the norm $r(\mathcal{X}) = \widehat{r(\mathcal{X})}$ and $p(\eta_{x}(J)) = \ker(C^{*}_{\widehat{r(\mathcal{X})}}(G^{x}_{x})\to C^{*}_{r(\mathcal{X})}(G^{x}_{x})) = 0$.
\end{proof}

\begin{cor}
\label{cor:isofibrewhenregular}
Let $G$ be an \'etale groupoid, $x\in G^{0}$, $\mathcal{X} = \pi^{-1}_{ess}(x)$ and $J$ the singular ideal. If $r(\mathcal{X}) = r$ on $C_{c}(G^{x}_{x}\cdot \pi^{-1}_{ess}(x))$ (for instance if $G^{x}_{x}\cdot\pi^{-1}_{ess}(x)$ is amenable), then $\widehat{r(\mathcal{X})} = r\hat{\mathcal{X}}$ on $\mathbb{C}[G^{x}_{x}]$ and $p(J_{x}) = J_{G^{x}_{x}, \pi^{-1}_{ess}(x)}$.
\end{cor}

We recall a simple, well-known fact we will use a few times in this section.

\begin{lem}
\label{lem:simplelem}
Let $G$ be an \'etale groupoid and $x\in G^{0}$. A function $f\in C^{*}_{r}(G)$ satisfies $f|_{G_{x}}\neq 0$ if and only if $\lambda_{G^{x}_{x}}(\eta_{x}(f^{*}f))\neq 0$.
\end{lem}
\begin{proof}
     We have $\lambda_{G^{x}_{x}}(\eta_{x}(f^{*}f))(\delta_{x}) = f^{*}f(x) = \sum_{g\in G_{x}}|f(g)|^{2}$. 
\end{proof}

\begin{cor}
\label{cor:nonzerofibre}
    Let $G$ be an \'etale groupoid. If $J\neq 0$, then there is $x\in D_{0}$ such that $J_{G_{x}^{x},\pi^{-1}_{ess}(x)}\neq 0$.
\end{cor}

\begin{proof}
From Lemma \ref{lem:simplelem}, $J\neq 0$ implies there is $x_{0}\in G^{0}$ such that $\lambda_{G^{x_{0}}_{x_{0}}}(J_{x_{0}})\neq 0$. By Theorem \ref{thm:isofibre}, we have $0\neq \lambda_{G^{x_{0}}_{x_{0}}}(J_{x_{0}}) = \lambda_{G^{x_{0}}_{x_{0}}}(p(J_{x_{0}})) \subseteq \lambda_{G^{x_{0}}_{x_{0}}}(J_{G^{x_{0}}_{x_{0}}, \pi^{-1}_{ess}(x_{0})})$. By Corollary \ref{cor:edangptchar}, we have $x_{0}\in D_{0}$.
\end{proof}

\subsection{Characterization of vanishing of the singular ideal}
\label{ss:charvsing}
Now, in the case when the isotropy fibres behave nicely, we are able to characterize vanishing of the singular ideal in terms of a weak containment property.

\begin{thm}
\label{thm:vsingidealcharwkcont}
    Let $G$ be an \'etale groupoid such that the norm $r(\pi^{-1}_{ess}(x)) = r$ on $C_{c}(G^{x}_{x}\cdot \pi^{-1}_{ess}(x))$ for every $x\in D_{0}$. 
    
    Then, $J = \{0\}$ if and only if for every $x\in D_{0}$, we have $J_{G^{x}_{x}, \pi^{-1}_{ess}(x)} = 0$, if and only if $\lambda_{G^{x}_{x}}$ is weakly contained in $\oplus_{X\in \pi^{-1}_{ess}(x)}\lambda_{G^{x}_{x}/X}$. 
    
    More generally, we have $J = \{0\}$ if and only if $\ker(C^{*}_{\widehat{r(\mathcal{X})}}(G^{x}_{x})\to C^{*}_{r(\mathcal{X})}(G^{x}_{x})) = 0$ for all $x\in D_{0}$, where $r(\mathcal{X})$ is as in Theorem \ref{thm:isofibre}.
\end{thm}

\begin{proof}
Suppose $J\neq 0$. Then, Corollary \ref{cor:nonzerofibre} implies $J_{G^{x_{0}}_{x_{0}}, \pi^{-1}_{ess}(x_{0})}\neq 0$ for some $x_{0}\in D_{0}$. 

Conversely, if $J_{G^{x_{0}}_{x_{0}}, \pi^{-1}_{ess}(x_{0})}\neq 0$, then Corollary \ref{cor:isofibrewhenregular} implies $p(J_{x_{0}}) = J_{G^{x_{0}}_{x_{0}}, \pi^{-1}_{ess}(x_{0})}\neq 0$, and therefore $J\neq 0$.

By Lemma \ref{lem:groupidealchar}, $J_{G^{x_{0}}_{x_{0}}, \pi^{-1}_{ess}(x_{0})}\neq 0$ if and only if is $\lambda_{G_{x_{0}}^{x_{0}}}$ is not weakly contained in $\oplus_{X\in \pi^{-1}_{ess}(x_{0})}\lambda_{G^{x_{0}}_{x_{0}}/X}$.

The more general statement follows from a similar argument and $p(J_{x}) = \ker(C^{*}_{\widehat{r(\mathcal{X})}}(G^{x}_{x})\to C^{*}_{r(\mathcal{X})}(G^{x}_{x}))$ (Theorem \ref{thm:isofibre}).
\end{proof}

Now, let's characterize when $\ker(C^{*}_{\widehat{r(\mathcal{X})}}(G^{x}_{x})\to C^{*}_{r(\mathcal{X})}(G^{x}_{x})) = 0$ in terms of a ``weak containment'' property as above. First, let's see this property is invariant under germ isomorphism (Definition \ref{dfn:germiso})

\begin{prp}
\label{prp:dependsongerm}
    Let $G$, $H$ be \'etale groupoids such that $(G, \{x\})\simeq (H, \{y\})$ for some $x\in G^{0}$, $y\in H^{0}$. Then, $\ker(C^{*}_{\widehat{r(\mathcal{X})}}(G^{x}_{x})\to C^{*}_{r(\mathcal{X})}(G^{x}_{x})) = 0$ if and only if $\ker(C^{*}_{\widehat{r(\mathcal{Y})}}(H^{y}_{y})\to C^{*}_{r(\mathcal{Y})}(H^{y}_{y})) = 0$, where $\mathcal{X} = \pi^{-1}_{ess}(x)$ and $\mathcal{Y} = \pi^{-1}_{ess}(y)$.
\end{prp}

\begin{proof}
    Let $G'$ be a local groupoid about $G^{x}_{x}$, $H'$ a local groupoid about $H^{y}_{y}$ and $\alpha:G'\to H'$ an isomorphism. This induces a $*$-isomorphism $\alpha^{*}:C_{r}^{*}(H')\to C^{*}_{r}(G')$ such that $\alpha^{*}(f) = f\circ \alpha$, for all $f\in C^{*}_{r}(H')$. Denote the singular ideals in $C^{*}_{r}(G')$ and $C^{*}_{r}(H')$ as $J'_{G}$ and $J'_{H}$, respectively. Since function belongs to a singular ideal if and only if it has dense zero set in the respective groupoid, and $\alpha$ is a homeomorphism, it follows that $\alpha^{*}(J'_{H}) = J'_{G}$.
    
     We show $C^{*}_{r}(G')\cap J_{G} = J_{G}'$. Suppose $f\in C^{*}_{r}(G')\cap J_{G}$. Then, $f^{-1}(0)$ is dense in $G$ and hence dense in $G'$ ($G'$ is open in $G$), so $f\in J_{G}'$. Conversely, if $f\in J_{G}'$, then $f^{-1}(0)$ is dense in $G'$ and $f = 0$ off $G'$, when viewed as in $C^{*}_{r}(G)$. Hence, $f\in C^{*}_{r}(G')\cap J_{G}$. We have shown $J_{G}' = J\cap C^{*}_{r}(G)$. Similarly, $C^{*}_{r}(H')\cap J_{H} = J'_{H}$.

    Therefore, $\alpha^{*}(C^{*}_{r}(H')\cap J_{H}) = C^{*}_{r}(G')\cap J_{G}$. By Corollary \ref{cor:germnormiso}, the diagram 

    $$
\begin{tikzcd}
C^{*}_{r}(H') \arrow[r, "\alpha^{*}"] \arrow[d, "\eta_{y}"] & C^{*}_{r}(G') \arrow[d, "\eta_{x}"] \\
C^{*}_{r(y)}(H^{y}_{y}) \arrow[r, "\alpha^{*}"]             & C^{*}_{r(x)}(G^{x}_{x})            
\end{tikzcd}$$ commutes, with the bottom map a *-isomorphism. Therefore, $$\alpha^{*}(\eta_{y}(C^{*}_{r}(H')\cap J_{H})) = \eta_{x}(\alpha^{*}(C^{*}_{r}(H')\cap J_{H})) = \eta_{x}(C^{*}_{r}(G')\cap J_{G}).$$

By Theorem \ref{thm:isofibre} we have $\eta_{y}(C^{*}_{r}(H')\cap J_{H}) = J_{y}$ and $\eta_{x}(C^{*}_{r}(G')\cap J_{G}) = J_{x}$, so there is an induced *-isomorphism $\alpha^{*}_{ess}C^{*}_{r(y)}(G^{y}_{y})/J_{y}\to C^{*}_{r(x)}(G^{x}_{x})/J_{x}$ making the diagram

$$
\begin{tikzcd}
{\mathbb{C}[H^{y}_{y}]} \arrow[d, "q_{y}"] \arrow[r, "\alpha^{*}"] & {\mathbb{C}[G^{x}_{x}]} \arrow[d, "q_{x}"] \\
C^{*}_{r(y)}(H^{y}_{y})/J_{y} \arrow[r, "\alpha^{*}_{ess}"]        & C^{*}_{r(x)}(G^{x}_{x})/J_{x}             
\end{tikzcd}$$ commute, where $q_{x}$, and $q_{y}$ are the quotient maps. By the proof of Theorem \ref{thm:isofibre} (in particular, by the properties proven of the third commutative diagram), the semi-C*-norms $\|q_{y}(-)\| = \|-\|_{r(\mathcal{Y})}$ and $\|q_{x}(-)\| = \|-\|_{r(\mathcal{X})}$. Hence, $\|\alpha^{*}(-)\|_{r(\mathcal{X})} = \|-\|_{r(\mathcal{Y})}$. Since $\alpha:G^{x}_{x}\to G^{y}_{y}$ is a group isomorphism, we also have $\|\alpha^{*}(-)\|_{r} = \|-\|_{r}$. Therefore, $\|\alpha^{*}(-)\|_{\widehat{r(\mathcal{X})}} = \|-\|_{\widehat{r(\mathcal{Y})}}$, proving that the diagram

$$
\begin{tikzcd}
C^{*}_{\widehat{r(\mathcal{Y})}}(H^{y}_{y}) \arrow[d] \arrow[r, "\alpha^{*}"] & C^{*}_{\widehat{r(\mathcal{X})}}(G^{x}_{x}) \arrow[d] \\
C^{*}_{r(\mathcal{Y})}(H^{y}_{y}) \arrow[r, "\alpha^{*}_{ess}"]               & C^{*}_{r(\mathcal{X})}(G^{x}_{x})                    
\end{tikzcd}$$ commutes, with the horizontal maps *-isomorphisms. This proves the proposition.

\end{proof}

Now, we prove a lemma which we can apply to obtain multiple ``weak containment'' type characterizations of when $\ker(C^{*}_{\widehat{r(\mathcal{X})}}(G^{x}_{x})\to C^{*}_{r(\mathcal{X})}(G^{x}_{x})) = 0$.

\begin{lem}
\label{lem:groupidealcharlocwkcont}
    Let $G$ be an \'etale groupoid, $x\in G^{0}$ and $\mathcal{X} = \pi^{-1}_{ess}(x)$. Let $H\subseteq G$ be a local groupoid about $G^{x}_{x}$ with $H^{0} = G^{0}$ and $\mathcal{U}$ a neighbourhood basis for $x$ in $G^{0}$. Choose a bisection $U_{g}$ for each $g\in G^{x}_{x}$.
    
    Suppose $S\subseteq \tilde{G}_{ess}^{0}$ is a subset such that the representation $\oplus_{X\in S\cap\pi_{ess}^{-1}(U)}\lambda_{\tilde{G}_{X}}$ is faithful for $C^{*}_{r}(H|_{U}\cdot \pi^{-1}(U))$ , for all $U\in\mathcal{U}$.

    Then, $\ker(C^{*}_{\widehat{r(\mathcal{X})}}(G^{x}_{x})\to C^{*}_{r(\mathcal{X})}(G^{x}_{x})) = 0$ if and only if for every $\epsilon > 0$, finite set $F\subseteq G^{x}_{x}\setminus \{x\}$ and $U\in\mathcal{U}$ with  $U\subseteq \bigcap_{g\in F} r(U_{g})\cap s(U_{g})$, there are vectors $\psi_{i}\in \ell^{2}(G_{x_{i}}^{U}/X_{i})$, where $X_{i}\in S\cap \pi^{-1}(U)$ and $\pi(X_{i}) = x_{i}$, for $i\leq n$ such that

    $$\sum^{n}_{i=1}\langle \psi_{i}, \psi_{i}\rangle = 1\text{ and }|\sum^{n}_{i=1}\langle 1_{U_{g}} * \psi_{i}, \psi_{i}\rangle |\leq \epsilon, \text{ for all }g\in F,$$
    where $1_{U_{g}}$ denotes the characteristic function on $U_{g}$, for $g\in F$. 
\end{lem}

\begin{proof}
 For $f\in \mathscr{C}_{c}(G)$ and $\psi\in \ell^{2}(G_{x}/X)$, we note that $\lambda_{X}(\iota_{ess}(f))(\psi) = f*\psi$, where the right hand side for $\psi = \delta_{hX}$ satisfies $f*\delta_{hX} = \sum_{s(g) = r(h)}f(g)\delta_{ghX}$. Throughout the proof, we make this identification without comment.

    First, we prove the ``only if'' direction. Let $\oplus_{\tau}\pi_{\tau}$ be the GNS representation of $C^{*}_{r(\mathcal{X})}(G^{x}_{x})$, where $\tau$ denotes a state. Then, $\ker(C^{*}_{\widehat{r(\mathcal{X})}}(G^{x}_{x})\to C^{*}_{r(\mathcal{X})}(G^{x}_{x})) = 0$ if and only if $\lambda_{G^{x}_{x}}$ is weakly contained in $\oplus_{\tau}\pi_{\tau}$. Therefore, for every finite set $F\subseteq G^{x}_{x}\setminus \{x\}$ and $\epsilon > 0$, there are states $\tau_{i}$ on $C^{*}_{r(\mathcal{X})}(G^{x}_{x})$ and $a_{i}\in\mathbb{C}[G]$, for $i\leq n$, such that $$| 1 - \sum^{n}_{i=1}\tau_{i}(a^{*}_{i}a_{i})| = |\langle\lambda_{G^{x}_{x}}(\delta_{x})\delta_{x}, \delta_{x}\rangle - \sum^{n}_{i=1}\tau_{i}(a^{*}_{i} \delta^{*}_{x}*a_{i})| = 0$$ and

    \begin{equation}
    \label{eq:wkequivalence1}
        | \sum^{n}_{i = 1}\tau_{i}(a^{*}_{i}\delta^{*}_{g} * a_{i})| = |\langle \lambda_{G^{x}_{x}}(\delta_{g})\delta_{x},\delta_{x}\rangle - \sum^{n}_{i = 1}\tau_{i}(a^{*}_{i}\delta^{*}_{g}* a_{i})|\leq \frac{\epsilon}{2},
    \end{equation}
for all $g\in F$.

Let $\{U_{g}\}_{g\in F}$ be open bisections and $U\in\mathcal{U}$ such that $U\subseteq \bigcap_{g\in F} r(U_{g})\cap s(U_{g})$. Choose $V\in\mathcal{U}$ such that $\overline{V}^{G^{0}}$ is compact and $\overline{V}^{G^{0}}\subseteq U$. Choose $\phi\in C_{c}(U)$ such that $\phi|_{V} = 1$ and define, for $g\in F$, $\phi_{g} = (\phi\circ r)(\phi\circ s)|_{U_{g}}\in C_{c}(U_{g})$.

Let $H\subseteq G$ be the local groupoid as in the hypothesis. Then, $H|_{V} = H_{1}$ and $H|_{U} = H_{2}$ are local groupoids about $G^{x}_{x}$ and $K_{1} = H_{1}\cdot\tilde{G}_{ess}^{0}$, $K_{2} = H_{2}\cdot G^{0}_{ess}$ are local groupoids about $\pi_{ess}^{-1}(x)$. with $H^{0}_{1} = \pi_{ess}^{-1}(V)$, $H^{0}_{ess} = \pi^{-1}_{ess}(U)$. Then, the restriction maps $q_{1}:C^{*}_{r}(K_{1})\to C^{*}_{r(\mathcal{X})}(G^{x}_{x}\cdot \mathcal{X})$ and $q_{2}:C^{*}_{r}(K_{2})\to C^{*}_{r}(G^{x}_{x}\cdot\mathcal{X})$ are *-homomorphisms by Theorem \ref{thm:locinvar&compress}.

Since $C^{*}_{r(\mathcal{X})}(G^{x}_{x})\subseteq C^{*}_{r(\mathcal{X})}(G^{x}_{x}\cdot\mathcal{X})$, we can extend the states $\tau_{i}$, $i\leq n$, to states $\tilde{\tau}_{i}$ on  $C^{*}_{r(\mathcal{X})}(G^{x}_{x}\cdot\mathcal{X})$. Moreover, under this embedding we have $q_{2}(\varphi_{g}) = \delta_{g}$, where $\varphi_{g} = \iota_{ess}(\phi_{g})$, for all $g\in F$. Let $b_{i}\in \mathscr{C}_{c}(K_{1})$ for $i\leq n$ be such that $q_{1}(b_{i}) = a'_{i}$. Then, $b^{*}_{i}\varphi^{*}_{g}b_{i}\in C^{*}_{r}(K_{1})$ and $q_{1}(b^{*}_{i}\varphi^{*}_{g}b_{i}) = a^{*}_{i}\delta^{*}_{g}*a_{i}$ for all $g\in F$ and $i\leq n$. Since $K_{1} = (H|_{V})\cdot \pi_{ess}^{-1}(V)$, by the hypothesis, the representation $\oplus_{X\in S\cap\pi_{ess}^{-1}(V)}\lambda_{\tilde{G}_{X}}$ is faithful for $C^{*}_{r}(K_{1})$, so we can approximate (in the weak* topology) the states $\tilde{\tau}_{i}\circ q_{1}$ by vector states in $\ell^{2}(G^{V}_{\pi(X)}/X)$, $X\in S\cap \pi^{-1}(V)$. Therefore, for every $i\leq n$ there are vectors $\psi'_{ij}\in \ell^{2}(G^{V}_{x_{ij}}/X_{ij})\subseteq \ell^{2}(G^{U}_{x_{ij}}/X_{ij})$ for $j\leq m_{i}$, where $X_{ij}\in S\cap\pi^{-1}(V)\subseteq S\cap \pi^{-1}(U)$ and $\pi(X_{ij}) = x_{ij}$, such that 
\begin{equation}
\label{eq:wkequivalence2}
    \sum_{i,j}\langle \psi'_{ij}, b_{i}^{*}b_{i}\psi'_{ij}\rangle = 1 \text{ and } |\tau_{i}(a^{*}_{i}\delta^{*}_{g}a_{i}) - \sum^{m_{i}}_{j=1}\langle \psi'_{i,j}, (b_{i}^{*}\varphi^{*}_{g}b_{i})\psi'_{i,j}\rangle|\leq \frac{\epsilon}{2n}
\end{equation}
for all $i\leq n$ and $g\in F$.

So, if we re-index $ij$, $i\leq n$, $j\leq m_{i}$ by $k\leq l$, and write $\psi_{k} = b_{i}\psi'_{ij}$, using the inequalities \ref{eq:wkequivalence1} and \ref{eq:wkequivalence2}, we have 

$$\sum^{l}_{k=1}\langle \psi_{k},\psi_{k}\rangle = 1 \text{ and } |\sum^{l}_{k=1}\langle \varphi_{g}*\psi_{k}, \psi_{k}\rangle |\leq \epsilon,$$
for all $g\in F$.

Now, since $\psi_{k}\in \ell^{2}(G^{V}_{x_{k}}/X_{k})$ and $x_{k}\in V$ we have  $\langle \varphi_{g}*\psi_{k}, \psi_{k}\rangle = \langle (\phi_{g}|_{r^{-1}(V)\cap s^{-1}(V)})*\psi_{k}, \psi_{k}\rangle = \langle (1_{U_{g}}|_{r^{-1}(V)\cap s^{-1}(V)})*\psi_{k}, \psi_{k}\rangle = \langle 1_{U_{g}}*\psi_{k}, \psi_{k}\rangle $, for all $k\leq l$. This proves the ``only if'' direction and the ``moreover'' statement.

We now prove the ``if'' direction. For $a\in \mathbb{C}[G^{x}_{x}]$ with $F = \{g\in G:a(g)\neq 0\}\setminus \{x\}$ and bisections $\{V_{g}\}_{g\in F}$, let $V,U\in\mathcal{U}$ such that $V\subseteq \bigcap_{g\in F}r(V_{g})\cap s(V_{g})$, $\overline{U}^{G^{0}}$ is compact, and $\overline{U}^{G^{0}}\subseteq V$, let $\phi\in C_{c}(V)$ be such that $\phi|_{U} = 1$. Set $\phi_{g} = (\phi\circ r)(\phi\circ s)|_{V_{g}}\in C_{c}(V_{g})$, for $g\in F$ and $b^{\phi} = a(x)\phi + \sum_{g\in F}a(g)\phi_{g}$. By the hypothesis, there are vectors $\psi_{k}\in \ell^{2}(G^{U}_{x_{k}}/X_{k})$ for $k\leq l$ such that $\sum^{l}_{k=1} \langle\psi_{k}, \psi_{k}\rangle = 1$ and $|\sum^{l}_{k=1}\langle \psi_{k}, 1_{U_{g}}*\psi_{k}\rangle| \leq \frac{\epsilon}{M}$, for all $g\in F$, where $M = \sum_{g\in F} |a_{g}|$. Since $\langle \psi_{k}, 1_{U_{g}}*\psi_{k}\rangle = \langle \psi_{k}, (1_{U_{g}}|_{r^{-1}(U)\cap s^{-1}(U)})*\psi_{k}\rangle = \langle \psi_{k}, (\phi_{g}|_{r^{-1}(U)\cap s^{-1}(U)})*\psi_{k}\rangle = \langle \psi_{k}, \phi_{g}*\psi_{k}\rangle$ for all $g\in F$ and $k\leq l$ we have, by the triangle ineqality, $|a(x) - \tau(\iota_{ess}(b^{\phi}))|\leq \epsilon$, where $\tau(c) = \sum^{l}_{k=1}\langle \psi_{k}, c*\psi_{k}\rangle$ for $c\in C^{*}_{r}(\tilde{G}_{ess})$. Since $\tau$ is a state on $C^{*}_{r}(\tilde{G}_{ess})$, it follows that $|a(x)| - \epsilon \leq \|\iota_{ess}(b^{\phi})\|$. 

Choose an approximate unit $\phi_{\lambda}\subseteq C_{c}(G^{0}\setminus \{x\})$ such that $\phi_{\lambda}|_{U_{\lambda}} = 0$ for some $U_{\lambda}$ with $\overline{U_{\lambda}}^{G^{0}}$ compact and $\overline{U_{\lambda}}^{G^{0}}\subseteq V$. Then, with $\phi^{\lambda} = (1-u_{\lambda})\phi(1-u_{\lambda})$, we have $b^{\phi^{\lambda}} = (1-u_{\lambda})b^{\phi}(1-u_{\lambda})$. Therefore, we have $|a(x)| - \epsilon \leq \|\iota_{ess}((1-u_{\lambda})b^{\phi}(1-u_{\lambda}))\|$ for all $\lambda$. By Theorem \ref{thm:locinvar&compress}, $(\iota_{ess}(u_{\lambda}) = u_{\lambda}\circ \pi_{ess})$ is an approximate unit for the kernel of some compression of $\eta_{\mathcal{X}}$ to a *-homomorphism. Therefore, by Equation \ref{eq:norm3} in Corollary \ref{cor:normeqngpd}, we have $|a(x)| - \epsilon \leq \|\eta_{\mathcal{X}}(\iota_{ess}(b^{\phi}))\|_{r(\mathcal{X})} = \|a\|_{r(\mathcal{X})}$. Since $\epsilon > 0$ was arbitrary, it follows that $|a(x)|\leq \|a\|_{r(\mathcal{X})}$. Therefore, $a\mapsto a(x)$ defines a $\|\cdot\|_{r(\mathcal{X})}$ bounded linear functional on $\mathbb{C}[G^{x}_{x}]$. As $a\mapsto a(x)$ defines a positive linear functional on $C^{*}(G^{x}_{x})$ and its GNS representation is unitarily equivalent to the left regular representation, it follows that $\ker(C^{*}(G_{x}^{x})\to C_{r(\mathcal{X})}^{*}(G^{x}_{x})))\subseteq \ker(\lambda_{G^{x}_{x}})$ and therefore $\ker(C^{*}_{\widehat{r(\mathcal{X})}}(G^{x}_{x})\to C^{*}_{r(\mathcal{X})}(G^{x}_{x})) = 0$, proving the ``if'' direction and thus the lemma. 
\end{proof}

To simplify the statement of our main result, we will make a definition.

\begin{dfn}
\label{dfn:G-wkcomtain} Fix $x\in G^{0}$, write $\mathcal{X} = \pi^{-1}_{ess}(x)$ and choose a bisection $U_{g}$ for every $g\in G^{x}_{x}$. We say $\lambda_{G^{x}_{x}}$ is \emph{$G$-weakly contained} in $\lambda_{G^{x}_{x}/\mathcal{X}}$ if for every $\epsilon > 0$, finite set $F\subseteq G^{x}_{x}\setminus \{x\}$ and open neighbourhood $U$ of $x$ with $U\subseteq \bigcap_{g\in F}r(U_{g})\cap s(U_{g})$, there are vectors $\psi_{i}\in \ell^{2}(G^{U}_{x_{i}}/X_{i})$, where $X_{i}\in\mathcal{X}(x_{i})$, $x_{i}\in U$ for $i\leq n$ such that 

    $$\sum^{n}_{i=1}\langle \psi_{i}, \psi_{i}\rangle = 1\text{ and }|\sum^{n}_{i=1}\langle 1_{U_{g}}*\psi_{i},\psi_{i}\rangle|\leq\epsilon,$$ for all $g\in F$.

We write $\lambda_{G^{x}_{x}}\prec_{G}\lambda_{G^{x}_{x}/\mathcal{X}}$.
\end{dfn}

By Lemma \ref{lem:groupidealcharlocwkcont} applied to $S = \tilde{G}_{ess}^{0}$, $\lambda_{G^{x}_{x}}\prec_{G}\lambda_{G^{x}_{x}/\mathcal{X}}$ if and only if $\ker(C^{*}_{\widehat{r(\mathcal{X})}}(G^{x}_{x})\to C^{*}_{r(\mathcal{X})}(G^{x}_{x})) = 0$, so the definition is independent of the bisections $U_{g}$ chosen.

Moreover, we have the following corollary.

\begin{cor}
    Let $G$, $H$ be \'etale groupoids such that $(G, \{x\})\simeq (H, \{y\})$ for some $x\in G^{0}$, $y\in H^{0}$. Then, $\lambda_{G^{x}_{x}}\prec_{G}\lambda_{G^{x}_{x}/\mathcal{X}}$ if and only if $\lambda_{H^{y}_{y}}\prec_{H}\lambda_{H^{y}_{y}/\mathcal{Y}}$.
\end{cor}

\begin{proof}
    This is an immediate application of Proposition \ref{prp:dependsongerm} and Lemma \ref{lem:groupidealchar} applied to $S = \tilde{G}^{0}_{ess}$.
\end{proof}

\begin{thm}
\label{thm:vsingidealcharlocwkcont_ess}
    Let $G$ be an \'etale groupoid. Then, $J = \{0\}$ if and only if $\lambda_{G^{x}_{x}}\prec_{G}\lambda_{G^{x}_{x}/\mathcal{X}}$ for all $x\in G^{0}$.
\end{thm}

\begin{proof}
    Again, this is an immediate application of Proposition \ref{prp:dependsongerm} and Lemma \ref{lem:groupidealchar} applied to $S = \tilde{G}^{0}_{ess}$.
\end{proof}

Here is another description of $G$-weak containment in terms of $G$ only.

\begin{thm}
\label{thm:vsingidealcharlocwkcont_cont}
    Let $G$ be an \'etale groupoid and $x\in G^{0}$. For each $g\in G^{x}_{x}$ choose a bisection $U_{g}$. We have $\lambda_{G^{x}_{x}}\prec_{G}\lambda_{G^{x}_{x}/\mathcal{X}}$ if and only if for every $\epsilon > 0$, finite set $F\subseteq G^{x}_{x}\setminus \{x\}$ and open neighbourhood $U\subseteq \bigcap_{g\in G}r(U_{g})\cap s(U_{g})$ of $x$, there are vectors $\psi_{i}\in \ell^{2}(G^{U}_{x_{i}})$, where $x_{i}\in C\cap U$ for all $i\leq n$ such that 

    $$\sum^{n}_{i=1}\langle \psi_{i}, \psi_{i}\rangle = 1 \text{ and } |\sum^{n}_{i=1}\langle 1_{U_{g}}*\psi_{i}, \psi_{i}\rangle |\leq \epsilon$$ for all $g\in F$.
\end{thm}

\begin{proof}
    Apply Lemma \ref{lem:groupidealcharlocwkcont} to $S = \iota(C)$.
\end{proof}

\begin{rmk}
\label{rmk:torfree}
    Using the characterization for $\lambda_{G^{x}_{x}}\prec_{G}\lambda_{G^{x}_{x}/\mathcal{X}}$ in Theorem \ref{thm:vsingidealcharlocwkcont_cont} and a similar application of \cite[Lemma~1.9]{NS23} as in \cite[Proposition~1.8]{NS23}, it is easy to see $\lambda_{G^{x}_{x}}$ is not $G$-weakly contained in $\lambda_{G^{x}_{x}/\mathcal{X}}$ when $G^{x}_{x}$ is torsion free and $\{x\}\notin \mathcal{X}$. In particular, $J_{x}\neq \{0\}$.
\end{rmk}

If $G$ is minimal, then the number of representations needed to check $G$-weak containment reduces considerably. 

\begin{thm}
\label{thm:vsingidealcharlocwkcont_min}
    Let $G$ be a minimal \'etale groupoid and $x\in G^{0}$. Choose $X\in \mathcal{X} = \pi^{-1}_{ess}(x)$ and for each $g\in G^{x}_{x}$, choose a bisection $U_{g}$.
    
    Then, $\lambda_{G^{x}_{x}}\prec_{G}\lambda_{G^{x}_{x}/\mathcal{X}}$ if and only if for every $\epsilon > 0$, finite set $F\subseteq G^{x}_{x}\setminus \{x\}$ and open neighbourhood $U\subseteq \bigcap_{g\in G}r(U_{g})\cap s(U_{g})$ of $x$, there are vectors $\psi_{i}\in \ell^{2}(G^{U}_{x}/X)$ such that 

    $$\sum^{n}_{i=1}\langle \psi_{i}, \psi_{i}\rangle = 1 \text{ and } |\sum^{n}_{i=1}\langle 1_{U_{g}}*\psi_{i}, \psi_{i}\rangle |\leq \epsilon$$ for all $g\in F$.
\end{thm}

\begin{proof}
    Since $G$ is minimal, so is $\tilde{G}_{ess}$ (\cite[Lemma~5.1]{BGHL25}) and therefore $\lambda_{G/X}$ for any $X\in \pi^{-1}_{ess}(x)$, $x\in G^{0}$ is a faithful representation for $C^{*}_{r}(\tilde{G}_{ess})$. So, apply Lemma \ref{lem:groupidealcharlocwkcont} to $S = \{X\}$.
\end{proof}

\subsection{The isotropy fibres of $J\cap\mathscr{C}_{c}(G)$ and a characterization of the singular ideal intersection property}
\label{ss:isofibrealgsing}

 Let's identify the isotropy fibres of $J\cap\mathscr{C}_{c}(G)$ using an explicit construction. We will use this later to characterize the largest class of \'etale groupoid C*-alebras (that can be defined in a certain way) for which the property $J\neq \{0\}\implies J\cap\mathscr{C}_{c}(G)\neq \{0\}$ holds in terms of a property of the isotropy group C*-algebras. To see that it is the largest class, we will need the construction in Section \ref{s:nonHDconstruct}, however.

\begin{prp}
\label{prp:algisofibre}
    Let $G$ be an \'etale groupoid, $x\in G^{0}$ and $J$ its singular ideal. Then, $\eta_{x}(J\cap\mathscr{C}_{c}(G)) = J_{G^{x}_{x}, \pi^{-1}_{ess}(x)}\cap \mathbb{C}[G^{x}_{x}]$.
\end{prp}

\begin{proof}
The diagram 

$$
\begin{tikzcd}
\mathscr{C}_{c}(G) \arrow[r] \arrow[d] & C_{c}(\tilde{G}_{ess}) \arrow[d]        \\
{\mathbb{C}[G^{x}_{x}]} \arrow[r, "q"] & C_{c}(G^{x}_{x}\cdot \pi^{-1}_{ess}(x))
\end{tikzcd}$$ commutes and from Lemma \ref{lem:alggroupidealchar}, we have $\ker(q) = J_{G^{x}_{x}, \pi^{-1}_{ess}(x)}\cap\mathbb{C}[G^{x}_{x}]$. Therefore, $\eta_{x}(J\cap\mathscr{C}_{c}(G))\subseteq J_{G^{x}_{x}, \pi^{-1}_{ess}(x)}\cap\mathbb{C}[G^{x}_{x}]$

We prove the reverse inclusion. Let $b\in J_{G^{x}_{x}, \pi^{-1}_{ess}(x)}\cap\mathbb{C}[G^{x}_{x}]$, so that (by Lemma \ref{lem:alggroupidealchar}) $\sum_{h\in gX}b(h) = 0$ for all $X\in\pi_{ess}^{-1}(x)$ and $g\in G^{x}_{x}$. Let $\{g_{i}\}^{l}_{i=1} = \{g\in G^{x}_{x}: b(g)\neq 0\}$ and $\{U_{i}\}^{l}_{i=1}$ be open bisections such that $g_{i}\in U_{i}$ for all $i\leq l$ and $s(U_{i}) = s(U_{j}) =:W$ for all $i,j\leq l$. 

By choosing the $U_{i}$ small enough, we may assume that if $I\subseteq \{1,...,l\}$ satisfies $C_{I}:= s(\bigcap_{i\in I} U_{i}\setminus \bigcup_{j\notin I}U_{j})\cap C\neq \emptyset$, then $x\in \overline{s(\bigcap_{i\in I} U_{i}\setminus \bigcup_{j\notin I}U_{j})\cap C}^{G^{0}}$.

Let $\phi\in C_{c}(W)$ be non-zero at $x$, and define $f = \sum^{l}_{i=1}b(g_{i})(\phi\circ s)|_{U_{i}}.$ To show $f\in J$, it suffices to show $f(g) = 0$ if $g\in \bigcap_{i\in I} U_{i}\setminus \bigcup_{j\notin I}U_{j}$ and $C_{I}\neq\emptyset$. Let $(u_{\lambda})\subseteq C_{I}$ be a net converging to $x$. By taking a subnet if necessary, we can assume $u_{\lambda}$ converges to $X\in \pi_{ess}^{-1}(x)$. Then, for any $i_{0}\in I$, we have $\{g_{i}\}_{i\in I} = g_{i_{0}}X\cap \{g_{i}\}^{n}_{i=1}$, so that $f(g) = \sum_{i\in I}b(g_{i})\phi(s(g)) = \sum_{h\in g_{i_{0}}X}b(h)\phi(s(g)) = 0$. Hence, $f$ is in $J$.
\end{proof}

\begin{thm}
    \label{thm:valgidealfibres}
    Let $G$ be an \'etale groupoid. Then, $J\cap\mathscr{C}_{c}(G) = \{0\}$ if and only if $J_{G^{x}_{x}, \pi^{-1}_{ess}(x)}\cap \mathbb{C}[G^{x}_{x}] = 0$ for every $x\in D_{0}$.
\end{thm}
\begin{proof}
    By Proposition \ref{prp:algisofibre}, we have $\eta_{x}(J\cap\mathscr{C}_{c}(G)) =J_{G^{x}_{x}, \pi^{-1}_{ess}(x)}\cap \mathbb{C}[G^{x}_{x}]$.  Now, apply Lemma \ref{lem:simplelem} to some non-zero positive element in $J\cap\mathscr{C}_{c}(G)$ to show $J\cap\mathscr{C}_{c}(G)\neq 0$ if and only there is $x\in D_{0}$ such that $J_{G^{x}_{x}, \pi^{-1}_{ess}(x)}\cap \mathbb{C}[G^{x}_{x}] = \eta_{x}(J\cap\mathscr{C}_{c}(G)) \neq 0$, which proves the theorem.
\end{proof}

Now, we can apply our characterization of when $J_{\Gamma,\mathcal{X}}\cap\mathbb{C}[\Gamma] = \{0\}$ in Lemma \ref{lem:alggroupidealchar} $(3)$ to immediately obtain the following corollary of Theorem \ref{thm:valgidealfibres}.

\begin{cor}
\label{cor:valgidealcharspan}
    Let $G$ be an \'etale groupoid. Then, $J\cap\mathscr{C}_{c}(G) = \{0\}$ if and only if for every $x\in D_{0}$, there is no non-zero $a\in\mathbb{Z}[G^{x}_{x}]$ such that
    $$\sum_{h\in gX}a(h) = 0,\text{ for all }g\in G^{x}_{x},\text{ }X\in\pi^{-1}_{ess}(x).$$
\end{cor}
\begin{dfn}
\label{dfn:I&AI}
    Let $\Gamma$ be a discrete group and $\mathcal{X}$ a closed and invariant set of subgroups (as in Section \ref{ss:cosetgpds}). We will say $(\Gamma, \mathcal{X})\in \mathcal{I}$ if $J_{\Gamma,\mathcal{X}}\cap\mathbb{C}[\Gamma]\neq \{0\}$ or $J_{\Gamma,\mathcal{X}} = 0$.

    If $(\Gamma, \mathcal{X})\in \mathcal{I}$ for all closed and invariant sets of subgroups $\mathcal{X}$, we will say $\Gamma$ has \emph{Property $I$}, or the \emph{Intersection Property}.

    If for all closed invariant sets of subgroups $\mathcal{X}$, we have $J_{\Gamma,\mathcal{X}}\cap\mathbb{C}[\Gamma]\neq \{0\}$ if and only if $e\notin \mathcal{X}$, then we say $\Gamma$ has \emph{Property $AI$}, or the \emph{Automatic Intersection Property}.
\end{dfn}

Note that Property $AI$ implies Property $I$, but the converse does not hold. For instance, every discrete abelian group satisifies Property $I$ (Theorem \ref{thm:propertyAI&I}) and in Theorem \ref{thm:abeliancharAI}, we characterize when a discrete abelian group satisfies Property $AI$.

In \cite[Questions~4.11]{BGHL25} it was asked whether $J =0$ implies $J\cap\mathscr{C}_{c}(G) = 0$. We now prove this in the affirmitave whenever the isotropy and essential fibres belong to the class $\mathcal{I}$.

\begin{thm}
\label{thm:valgideal=videal}
Let $G$ be an \'etale groupoid, and assume $(G^{x}_{x}, \pi^{-1}_{ess}(x))\in\mathcal{I}$ for every $x\in D_{0}$. Then, $J = \{0\}$ if and only if $J\cap\mathscr{C}_{c}(G) = \{0\}$.

Suppose $G^{x}_{x}$ has Property $AI$ for some $x\in D_{0}$. Then, $J\cap\mathscr{C}_{c}(G)\neq\{0\}$.
\end{thm}
\begin{proof}
    $J = \{0\}$ implies $J\cap\mathscr{C}_{c}(G) = \{0\}$ and so the ``only if'' direction follows from Proposition \ref{prp:algisofibre}.

    If $J\neq 0$, then Corollary \ref{cor:nonzerofibre} implies there is $x\in D_{0}$ such that $J_{G^{x}_{x}, \pi^{-1}_{ess}(x)} \neq 0$. Since $(G^{x}_{x}, \pi^{-1}_{ess}(x))\in \mathcal{I}$, we have $J_{G^{x}_{x}, \pi^{-1}_{ess}(x)}\cap \mathbb{C}[G^{x}_{x}]\neq \{0\}$. 
    
    Since the quotient map $C^{*}_{r(x)}(G^{x}_{x})\to C^{*}_{r\hat{\mathcal{X}}}(G^{x}_{x})$, where $\mathcal{X} = \pi^{-1}_{ess}(x)$, surjects $\text{ker}(q_{x})\cap\mathbb{C}[G^{x}_{x}]$ bijectively onto $J_{G^{x}_{x}, \pi^{-1}_{ess}(x)}\cap \mathbb{C}[G^{x}_{x}]\neq \{0\}$ (trivially), it follows from Proposition \ref{prp:algisofibre} that $\eta_{x}(J\cap\mathscr{C}_{c}(G))\neq \{0\}$. Hence, $J\cap\mathscr{C}_{c}(G)\neq\{0\}$, proving the ``if'' direction.

    If $x\in D_{0}$ is such that $G^{x}_{x}$ has Property $AI$, then by Corollary \ref{cor:edangptchar}, we have $\{x\}\notin \pi^{-1}_{ess}(x)$ and hence $J_{G^{x}_{x}, \pi^{-1}_{ess}(x)}\cap\mathbb{C}[G^{x}_{x}]\neq \{0\}$ by the definition of Property $AI$. Now the same argument in the last paragraph applies to show $J\cap\mathscr{C}_{c}(G)\neq\{0\}$.
    
\end{proof}

We will verify the hypothesis for the above theorem for a large class of groupoids in Section \ref{s:groupswint}. Theorem \ref{thm:valgideal=videal} combined with Corollary \ref{cor:valgidealcharspan} yields the following.

\begin{thm}
\label{thm:linearchar}
    Let $G$ be an \'etale groupoid and assume $(G^{x}_{x}, \pi^{-1}_{ess}(x))\in\mathcal{I}$ for every $x\in D_{0}$. Then, $J = \{0\}$ if and only if for every  $x\in D_{0}$, there is no non-zero $a\in\mathbb{Z}[G^{x}_{x}]$ such that
    $$\sum_{h\in gX}a(h) = 0,\text{ for all }g\in G^{x}_{x},\text{ }X\in\pi^{-1}_{ess}(x).$$
\end{thm}

\section{A non-Hausdorff groupoid construction with prescribed singular ideal isotropy fibre}
\label{s:nonHDconstruct}
In this section we prove a converse to Theorem \ref{thm:valgideal=videal}. More specifically we construct, for every discrete $\Gamma$ and closed invariant set of subgroups $\mathcal{X}$ not containing the identity, a non-Hausdorff groupoid $G$ with exactly one extremely dangerous point $x_{0}$, with $G^{x_{0}}_{x_{0}} = \Gamma$ and isotropy fibre $J_{x_{0}} = J_{\Gamma,\mathcal{X}}$.

Let $\Gamma\cdot\mathcal{X}$ be the coset groupoid introduced in Section \ref{ss:cosetgpds} and let $G$ be the groupoid $\Gamma\times\{\infty\}\sqcup (\Gamma\cdot \mathcal{X})\times \mathbb{N} = \Gamma\sqcup\bigsqcup_{n\in\mathbb{N}} \Gamma\cdot \mathcal{X} $ which we topologize using the basis consisting of open sets $U\subseteq (\Gamma\cdot\mathcal{X})\times \mathbb{N}$ in the product topology and sets $\{(\gamma,\infty)\}\sqcup \gamma\mathcal{X}\times \{k\in \mathbb{N}: k\geq n\}$ for $\gamma\in \Gamma$ and $n\in\mathbb{N}$.

Under this topology, $G^{0} = \{(e,\infty)\}\sqcup \mathcal{X}\times \mathbb{N}$ is the one point compactification of $\mathcal{X}\times \mathbb{N}$, where $\{(e,\infty)\}$ corresponds to the point at $\infty$. The sets $\{(\gamma,\infty)\}\sqcup \gamma\mathcal{X}\times \mathbb{N}$, for $\gamma\in \Gamma$, are open bisections. Therefore, $G$ is a locally compact \'etale groupoid with Hausdorff unit space. 

We calculate its Hausdorff cover $\tilde{G}$.
Let $\pi:\tilde{G}^{0}\to G^{0}$ be the natural projection. Since $\Gamma\cdot\mathcal{X}\times \mathbb{N}$ is Hausdorff and is the reduction of $G$ to $\mathcal{X}\times \mathbb{N}$, we have that $\pi^{-1}(u) = \{u\}$ for $u\in \mathcal{X}\times \mathbb{N}$. Suppose $(u_{\lambda})\in G^{0}\setminus \{(e,\infty)\}$ converges in $G^{0}$ to $\{(e,\infty)\}$ and converges in the Fell topology to $X'\times\{\infty\}\subseteq \Gamma\times\{\infty\}$. We can write $u_{\lambda} = (X_{\lambda}, n_{\lambda})$, where $n_{\lambda}\to \infty$ and $ X_{\lambda}\in \mathcal{X}$. By taking a subnet if necessary, we can assume that $X_{\lambda}$ converges in $\mathcal{X}$ to $X$. Since $u_{\lambda}$ converges to $X'\times\{\infty\}$ we must have, for every $x'\in X'$, that $X_{\lambda}\in x'\mathcal{X}$ eventually, which is equivalent to $x'\in X_{\lambda}$ eventually. It follows that $x'\in X$, so that $X'\subseteq X$. Moreover, $X_{\lambda}$ converging to $X$ implies, for all $x\in X$, that $x\in X_{\lambda}$ eventually, which is again equivalent to $X_{\lambda}\in x\mathcal{X}$, so that $x\in X'$. Hence, $X = X'$.

We have shown that $\pi^{-1}(\{e\}) = \mathcal{X}\times\{\infty\}\cup\{(e,\infty)\} = \hat{\mathcal{X}}\times\{\infty\}$, and it is easy to see (following the same argument above) that the topology on $\tilde{G}^{0}$ is homeomorphic to the disjoint union $\{(e,\infty)\}\sqcup\mathcal{X}\times (\infty\cup\mathbb{N})$. The action of $G|_{G^{0}\setminus \{(e,\infty)\}}$ on $\tilde{G}^{0}\setminus \pi^{-1}(\{(e,\infty)\}) = G^{0}\setminus \{(e,\infty)\}$ is the usual action of a groupoid on its unit space, while $G|_{(e,\infty)} = \Gamma\times\{\infty\}$ acts on $\pi^{-1}(\{(e,\infty)\}) = \{e\}\times\{\infty\}\cup \mathcal{X}\times \{\infty\}$ by conjugation. Hence, $\tilde{G}$ is isomorphic to $\Gamma\times \{\infty\}\sqcup \Gamma\cdot\mathcal{X}\times (\{\infty\}\cup\mathbb{N})$, with the essential Hausdorff cover $\tilde{G}_{ess}$ isomorphic to $\Gamma\cdot\mathcal{X}\times (\{\infty\}\cup\mathbb{N})$.

Since $(e,\infty)\notin \mathcal{X}\times\{\infty\} = \pi^{-1}_{ess}((e,\infty))$, Corollary \ref{cor:edangptchar} implies $x_{0} = (e,\infty)$ is extremely dangerous. The subgroupoid $G|_{G^{0}\setminus \{x_{0}\}}$ is Hausdorff, so this is the only extremely dangerous point. We have $G_{x_{0}}^{x_{0}} = \Gamma$ and $C^{*}_{r}(\tilde{G})\simeq C^{*}_{r}(\Gamma)\oplus C^{*}_{r}(\Gamma\cdot\mathcal{X})\otimes C(\{\infty\}\cup\mathbb{N})$. The summand on the right is $C^{*}_{r}(\tilde{G}_{ess})$, and so the norm $r(\pi^{-1}_{ess}(x_{0})) = r$ and $r(x_{0})$ is the norm determined by the homomorphisms $\mathbb{C}[\Gamma]\to C^{*}_{r}(\Gamma)$ and $\mathbb{C}[\Gamma]\to C^{*}_{r}(\Gamma\cdot\mathcal{X})$. Hence, $r(x_{0})$  is the supremum norm $r\hat{\mathcal{X}}$ of norms from the quasi-regular representations $\lambda_{\Gamma/X}$, $X\in\mathcal{X}\cup\{e\}=:\hat{\mathcal{X}}$. By Corollary \ref{cor:isofibrewhenregular}, $J_{x_{0}}$ is the kernel of $q:C^{*}_{r\hat{\mathcal{X}}}(\Gamma)\to C^{*}_{r}(\Gamma\cdot\mathcal{X})$ and thus $J_{x_{0}} = J_{\Gamma,\mathcal{X}}$.

Our construction proves the following.

\begin{thm}
\label{thm:nhdgpdconst}
If there is a discrete group $\Gamma$ with a closed set of subgroups $\mathcal{X}$ invariant under conjugation such that $(\Gamma, \mathcal{X})\notin \mathcal{I}$, then there is an \'etale groupoid $G$ such that $J\cap\mathscr{C}_{c}(G) = \{0\}$ and $J\neq \{0\}$. Moreover, $G$ has exactly one extremely dangerous point $x_{0}$, $G^{x_{0}}_{x_{0}} = \Gamma$ and $\pi_{ess}^{-1}(x_{0}) = \mathcal{X}$.
\end{thm}

Hence, we have proven a converse of sorts to Theorem \ref{thm:valgideal=videal}.

\begin{cor}
\label{cor:equalqs}
    A positive answer to Question \ref{qtn1} (about non-Hausdorff groupoid C*-algebras) is equivalent to a positive answer to Question \ref{qtn2} (about group C*-algebras). Similarly, for the appropriate restricted cases of these questions (e.x. amenable isotropy).
\end{cor}

\section{Groups satisfying the intersection properties}
\label{s:groupswint}
Now, we prove a large class of groups satisfies the intersection properties. We do so by establishing a variety of permanence properties. We note that many results similar to those for Property $AI$ groups here hold for group rings over rings more general than $\mathbb{C}$.

The first result we establish is, for torsion free groups, Property $I$ and $AI$ coincide.

\begin{prp}
    Let $\Gamma$ be a torsion free discrete group. Then, $\Gamma$ satisfies Property $I$ if and only if $\Gamma$ satisfies Property $AI$
\end{prp}

\begin{proof}
    It suffices to show that if $\mathcal{X}$ is a closed and invariant set of subgroups such that $\{e\}\notin \mathcal{X}$, then $J_{\Gamma,\mathcal{X}}\neq \{0\}$.  Let $G$ be the groupoid constructed in Section \ref{s:nonHDconstruct} from $\Gamma,\mathcal{X}$. Since the isotropy of $G$ is torsion free and $G$ has an extremely dangerous point, \cite[Proposition~1.8]{NS23} implies $J\neq \{0\}$. Since $G$ has exactly one extremely dangerous point $x_{0}$, we must have $J_{x_{0}}\neq \{0\}$. By construction, $J_{x_{0}} = J_{\Gamma,\mathcal{X}}$.
\end{proof}

We note that the above proposition can be proven directly by working with the group $\Gamma$, but the above argument displays an interesting application of groupoid theory to group theory.

Let's observe that the ideals we consider are isomorphic to ideals in reduced group C*-algebras.
\begin{prp}
\label{prp:gideal=rgideal}
    Let $\Gamma$ be a discrete group and $\mathcal{X}$ a closed set of subgroups invariant under conjugation. Then, $J_{\Gamma,\mathcal{X}}\neq\{0\}$ if and only if $\lambda_{\Gamma}(J_{\Gamma,\mathcal{X}})\neq\{0\}$.
\end{prp}

\begin{proof}
    Since $\|\cdot\|_{r\hat{\mathcal{X}}} = \sup\{\|\cdot\|_{r\mathcal{X}}, \|\cdot\|_{r}\}$ and $a\in J_{\Gamma,\mathcal{X}}$ if and only if $\|a\|_{r\mathcal{X}} = 0$, it follows that the map $\lambda_{\Gamma}:J_{\Gamma,\mathcal{X}}\to C^{*}_{r}(\Gamma)$ is injective.
\end{proof}
As a corollary, these ideals behave well under finite intersections.

\begin{cor}
\label{cor:gidealunion=int}
Let $\Gamma$ be a discrete group and suppose $\mathcal{X}_{1}$ and $\mathcal{X}_{2}$ are closed sets of subgroups invariant under conjugation. If $J_{\Gamma, \mathcal{X}_{1}\cup\mathcal{X}_{2}}\neq \{0\}$, then $J_{\Gamma, \mathcal{X}_{1}}\neq \{0\}$ and $J_{\Gamma, \mathcal{X}_{2}}\neq \{0\}$. Moreover, $J_{\Gamma,\mathcal{X}_{1}\cup\mathcal{X}_{2}}\cap\mathbb{C}[\Gamma] = J_{\Gamma, \mathcal{X}_{1}}\cap J_{\Gamma,\mathcal{X}_{2}}\cap \mathbb{C}[\Gamma]$.
\end{cor}

\begin{proof}
    By Proposition \ref{prp:gideal=rgideal}, we have $0\neq \lambda_{\Gamma}(J_{\Gamma, \mathcal{X}_{1}\cup\mathcal{X}_{2}})\subseteq \lambda_{\Gamma}(J_{\Gamma,\mathcal{X}_{1}})\cap \lambda_{\Gamma}(J_{\Gamma,\mathcal{X}_{2}})$, and therefore (by Proposition \ref{prp:gideal=rgideal}) $J_{\Gamma,\mathcal{X}_{1}}\neq 0$ and $J_{\Gamma,\mathcal{X}_{2}}\neq 0$. The ``moreover'' statement is immediate from the definitions of the ideals.
\end{proof}

They behave well under restriction to subgroup C*-algebras.

\begin{prp}
\label{prp:resttosg}
    Let $\Gamma$ be a discrete group and $\mathcal{X}$ a closed set of subgroups invariant under conjugation. Let $\Lambda\subseteq \Gamma$ be a subgroup. Then, $\Lambda\cap\mathcal{X} = \{\Lambda\cap X:X\in\mathcal{X}\}$ is a closed set of subgroups of $\Lambda$ invariant under conjugation by elements in $\Lambda$. Moreover, $C^{*}_{r\Lambda\cap\hat{\mathcal{X}}}(\Lambda)\subseteq C^{*}_{r\hat{\mathcal{X}}}(\Gamma)$, $J_{\Gamma,\mathcal{X}}\cap C^{*}_{r\Lambda\cap\hat{\mathcal{X}}}(\Lambda) = J_{\Lambda, \Lambda\cap\mathcal{X}}$ and $J_{\Gamma,\mathcal{X}}\cap\mathbb{C}[\Lambda] = J_{\Lambda, \Lambda\cap \mathcal{X}}\cap\mathbb{C}[\Lambda]$.
\end{prp}

\begin{proof}
    The representation $\lambda_{\Gamma}\oplus\bigoplus_{X\in\mathcal{X}}\lambda_{\Gamma/X}$ restricted to $\mathbb{C}[\Lambda]$ and $\ell^{2}(\Lambda)\oplus\bigoplus_{X\in\mathcal{X}}\ell^{2}(\Lambda/X)$ is unitarily equivalent to $\lambda_{\Lambda}\oplus\bigoplus_{X\in\mathcal{X}}\lambda_{\Lambda/\Lambda\cap X}$, so $C^{*}_{r\Lambda\cap\hat{\mathcal{X}}}(\Lambda)\subseteq C^{*}_{r\hat{\mathcal{X}}}(\Gamma)$. Moreover, since $a\in J_{\Gamma,\mathcal{X}}$ if and only if $\langle a\delta_{X}, a\delta_{X}\rangle = 0 $ for all $X\in\mathcal{X}$, the above fact implies $C^{*}_{r\Lambda\cap \hat{\mathcal{X}}}(\Lambda)\cap J_{\Gamma,\mathcal{X}} = J_{\Lambda,\Lambda\cap \mathcal{X}}$.
\end{proof}

\begin{cor}
\label{cor:ctbleunion}

    If $\Gamma = \bigcup_{n}\Gamma_{n}$, where $\Gamma_{n}$ are subgroups such that $\Gamma_{n}\subseteq \Gamma_{n+1}$ for all $n\in\mathbb{N}$, then $J_{\Gamma, \mathcal{X}}\neq 0$ if and only if $J_{\Gamma_{n}, \Gamma_{n}\cap\mathcal{X}}\neq 0$ for some $n\in\mathbb{N}$.

    Moreover, $J_{\Gamma, \mathcal{X}}\cap \mathbb{C}[\Gamma]\neq 0$ if and only if $J_{\Gamma_{n}, \Gamma_{n}\cap \mathcal{X}}\cap\mathbb{C}[\Gamma_{n}]\neq 0$ for some $n\in\mathbb{N}$.
\end{cor}

\begin{proof}
    The first ``if and only if'' follows from Proposition \ref{prp:resttosg} and the fact that $C_{r\hat{\mathcal{X}}}^{*}(\Gamma)$ is the inductive limit of $C^{*}_{r\Gamma_{n}\cap\hat{\mathcal{X}} }(\Gamma_{n})$ and every ideal in an inductive limit algebra must intersect an algebra in its limiting sequence (\cite[Lemma~III.4.1]{Davidson}).

    The second ``if and only if'' follows from Proposition \ref{prp:resttosg} and $\mathbb{C}[\Gamma] = \bigcup_{n}\mathbb{C}[\Gamma_{n}]$.
\end{proof}

It follows that Property $I$ and $AI$ are preserved under countable increasing sequences of groups with the same properties.

\begin{cor}
\label{cor:pAI&ctbleunions}
    Suppose $\Gamma$ is a discrete group and $\Gamma = \bigcup_{n}\Gamma_{n}$, where $\Gamma_{n}$ are subgroups such that $\Gamma_{n}\subseteq \Gamma_{n+1}$ for all $n\in\mathbb{N}$. If $\Gamma_{n}$ has Property $I$ (or $AI$) for all $n\in\mathbb{N}$, then $\Gamma$ has Property $I$ (or $AI$).
\end{cor}

\begin{proof}
    The fact that Property $I$ is invariant under increasing unions is immediate from Corollary \ref{cor:ctbleunion} and Proposition \ref{prp:resttosg}. As for Property $AI$, by compactness of $\mathcal{X}$ in $\{0,1\}^{\Gamma}$, if $\{e\}\notin \mathcal{X}$, there is a finite set $F\subseteq \Gamma\setminus \{e\}$ such that $X\cap F\neq \emptyset$, for all $X\in\mathcal{X}$. Therefore, if we choose $\Gamma_{n}$ such that $F\subseteq \Gamma_{n}$, then $\{e\}\notin \Gamma_{n}\cap\mathcal{X}$. We can then apply Property $AI$, Proposition \ref{prp:resttosg} and Corollary \ref{cor:ctbleunion} to conclude $\Gamma$ has Property $AI$.
\end{proof}

\begin{prp}
\label{prp:videalredtonsg}
    Let $\Gamma$ be a discrete group and $\mathcal{X}$ a closed set of subgroups invariant under conjugation. Let $N$ be the normal subgroup generated by $X\in\mathcal{X}$. Then, $J_{\Gamma,\mathcal{X}}\neq \{0\}$ if and only if $J_{N,\mathcal{X}}\neq \{0\}$. Moreover, $J_{\Gamma,\mathcal{X}}\cap \mathbb{C}[\Gamma]\neq\{0\}$ if and only if $J_{N,\mathcal{X}}\cap\mathbb{C}[N] \neq \{0\}$
\end{prp}

There is a special subgroup C*-algebra we can always restrict the ideal to which never affects its non-triviality.

\begin{proof}
By Proposition \ref{prp:resttosg} and the fact that $N\cap\mathcal{X} = \mathcal{X}$, we have $C^{*}_{r\hat{\mathcal{X}}}(N)\cap J_{\Gamma,\mathcal{X}} = J_{N,\mathcal{X}}$. It follows that $J_{N,\mathcal{X}}\neq 0$ implies $J_{\Gamma,\mathcal{X}}\neq 0$.

Let $\Phi:\mathbb{C}[\Gamma]\to \mathbb{C}[N]$ be the map $a\mapsto a|_{N}$, $a\in \mathbb{C}[\Gamma]$. Then, $\Phi$ extends to a c.p.c. (completely positive and completely contracting) map $\Phi:C^{*}(\Gamma)\to C^{*}(N)$ \cite{R74}. Since the induced representation on $\mathbb{C}[\Gamma]$ from $\lambda_{N/X}$ is unitarily equivalent to $\lambda_{\Gamma/X}$, for $X\in\hat{\mathcal{X}}$, it follows that $\Phi$ descends to a c.p.c. map $\Phi:C^{*}_{r\hat{\mathcal{X}}}(\Gamma)\to C^{*}_{r\mathcal{X}}(N)$. 

Let's show $\langle \Phi(a)\delta_{X}, \Phi(a)\delta_{X}\rangle \leq \langle a\delta_{X}, a\delta_{X}\rangle$, for all $a\in C^{*}_{r\hat{\mathcal{X}}}(\Gamma)$ and $X\in\mathcal{X}$. It suffices to check this for $a\in\mathbb{C}[\Gamma]$. Write $a = \sum^{n}_{i=1}a_{i}\delta_{g_{i}}$ with $g_{0} = \delta_{e}$ and let $R$ be the equivalence relation on $\{0,...,n\}$ defined by $\{(i,j): g_{i}^{-1}g_{j}\in N\}$. Let $[R]$ be the equivalence classes and for $q\in [R]$, define $a_{q} = \sum_{i\in q} a_{i}\delta_{g_{i}}$. For every $X\in\mathcal{X}$, we have $\langle  a\delta_{X},a\delta_{X}\rangle  = \sum_{(i,j):g^{-1}_{i}g_{j}\in X}\overline{a_{i}}a_{j}$. Since $X\subseteq N$, we have $g^{-1}_{i}g_{j}\in X$ implies $g_{i}, g_{j}\in q$ for some $q\in [R]$. Hence, $\sum_{(i,j):g^{-1}_{i}g_{j}\in X}\overline{a_{i}}a_{j} = \sum_{q\in [R]}\langle a^{*}_{q}a_{q}\delta_{X}, \delta_{X}\rangle
= \sum_{q\in [R]}a^{*}_{q}a_{q}(X)$. Since $\Phi(a) = a_{[0]}$, it follows that $\langle \Phi(a)\delta_{X}, \Phi(a)\delta_{X}\rangle \leq \sum_{q\in [R]}\langle a^{*}_{q}a_{q}\delta_{X}, \delta_{X}\rangle = \langle a\delta_{X}, a\delta_{X}\rangle.$

Therefore, if $a\in J_{\Gamma,\mathcal{X}}$ ($\langle a\delta_{X}, a\delta_{X}\rangle = 0$ for all $X\in\mathcal{X}$), then $\Phi(a)\in J_{N, \mathcal{X}}$ and hence $\Phi(J_{\Gamma,\mathcal{X}}) = J_{N,\mathcal{X}}$. Now, suppose $a$ is a positive and non-zero element in $J_{\Gamma,\mathcal{X}}$. From Proposition \ref{prp:gideal=rgideal}, we have $\lambda_{r}(a)\neq 0$, so that $\langle \lambda_{r}(a)\delta_{e}, \delta_{e}\rangle\neq 0$. Therefore, $\langle \lambda_{r}(\Phi(a))\delta_{e}, \delta_{e}\rangle = \langle \Phi(\lambda_{r}(a))\delta_{e}, \delta_{e}\rangle = \langle \lambda_{r}(a)\delta_{e}, \delta_{e}\rangle\neq 0$. Hence, $\Phi(a)$ is a non-zero element in $J_{N,\mathcal{X}}$.

Proposition \ref{prp:resttosg} implies $J_{\Gamma,\mathcal{X}}\cap \mathbb{C}[N] = J_{N,\mathcal{X}}\cap\mathbb{C}[N]$ and this proves the ``if'' direction of the second half of proposition.

The ``only if'' direction follows from $\Phi(J_{\Gamma,\mathcal{X}}\cap\mathbb{C}[\Gamma]) = J_{N,\mathcal{X}}\cap \mathbb{C}[N]$ and the fact (proven above) that if $a$ is positive and non-zero in $J_{\Gamma,\mathcal{X}}\cap\mathbb{C}[\Gamma]$, then $\Phi(a)\neq 0$.
\end{proof}

\begin{cor}
\label{cor:redtonsg}
    Let $\Gamma$ be a discrete group and $\mathcal{X}$ a closed set of subgroups invariant under conjugation. Then, $(\Gamma, \mathcal{X})\in \mathcal{I}$ if and only if $(N,\mathcal{X})\in\mathcal{I}$, where $N$ is the normal subgroup generated from $X\in\mathcal{X}$.
\end{cor}

The first part of the following result is a lemma which is likely well known in group theory (and is proven also in \cite[Theorem~4.7]{BGHL25}), but we prove it for completeness (the proof follows that in the aforementioned citation).

\begin{lem}
\label{lem:finite}
Let $\Gamma$ be a discrete group and $\mathcal{X}$ a finite set of finite subgroups invariant under conjugation. Then the normal subgroup $N$ generated by $X\in\mathcal{X}$ is finite. Therefore, $(\Gamma,\mathcal{X})\in\mathcal{I}$.
\end{lem}

\begin{proof}
    Let $N$ be the subgroup generated by $X\in\mathcal{X}$. The proof that $N$ is finite is the same as in the first paragraph of \cite[Theorem~4.7]{BGHL25}, but we recall it for completeness. Write $\mathcal{X} = \{X_{1},...X_{m}\}$. We claim every element $g\in N$ can be written as $g = g_{1}...g_{k}$, where $g_{i}\in X_{n_{i}}\setminus\bigcup_{j\neq i}X_{n_{j}}$. This forces $k\leq m$ so that $N$ is finite.
    
    To prove this, for $g\in N$, let $g = g_{1}...g_{k}$ be a minimal factorization of $g$, where $g_{i}\in \bigcup^{m}_{i=1}X_{i}$. Suppose for the sake of contradiction that there are $g_{j}, g_{j'}\in X_{l}$ for some $j < j'$ and $l\leq m$. For $i < j$, let $g'_{i} := g_{i}$, and for $i = j$, let $g'_{j}:= g_{j}g_{j'}$. For $j <i< j'$, let $g'_{i}:= g^{-1}_{j'}g_{i}g_{j'}$, and for $j'\leq i\leq k-1$, let $g'_{i} = g_{i+1}$. Then, we have $g = g'_{1}...g'_{k-1}$. Note $g'_{j}\in X_{l}$ since $X_{l}$ is closed under multiplication, and $g'_{i}\in \bigcup^{m}_{i=1}X_{i}$ since this set is closed under conjugation. Therefore, this factorization contradicts minimality of $k$, thus proving the claim that $N$ is finite.
\end{proof}

Now we show that finite invariant collections of subgroups belong in $\mathcal{I}$. First, we will prove a lemma which will be useful to us in multiple arguments.

\begin{lem}
\label{lem:prodelementconstr}
    Let $\Gamma$ be a discrete group and $\mathcal{X}$ a closed set of subgroups such that either

    \begin{enumerate}
        \item $\mathcal{X}$ is finite and every $X\in\mathcal{X}$ is infinite or
        \item every $X\in\mathcal{X}$ is normal and torsion free.
    \end{enumerate}
    Then, for every non-zero $a\in\mathbb{C}[\Gamma]$, there is $b\in J_{\Gamma,\mathcal{X}}\cap\mathbb{C}[\Gamma]$ such that $ab\neq 0$.
\end{lem}
\begin{proof}
    In both cases, the normalizer $N = \{g\in\Gamma: gXg^{-1} = X\text{ }\forall X\in\mathcal{X}\}$ has the property that $N\cap X$ is infinite, for all $X\in\mathcal{X}$. In case (1), write $N\cap\mathcal{X} = \{Y_{1},...,Y_{n}\}$. For case (2), $N = \Gamma$ and we know (by compactness of $\mathcal{X})$ there is a finite set of torsion free elements $\{f_{1},...,f_{n}\} = F\subseteq \Gamma\setminus\{e\}$ such that $F\cap X\neq \emptyset$ for all $X\in\mathcal{X}$. Let $Y_{i}$ be the subgroup generated by $f_{i}$, for $i\leq n$.

    For a non-zero $a\in\mathbb{C}[\Gamma]$, we claim there is $h_{i}\in Y_{i}$ such that $a\Pi^{n}_{i=1}(\delta_{e} - \delta_{h_{i}})\neq 0$. To prove this, it suffices by an induction argument to show for any non-zero element $c\in\mathbb{C}[\Gamma]$ and $i\leq n$, there is $h_{i}\in Y_{i}$ such that $c(\delta_{e} - \delta_{h_{i}})\neq 0$.

    Suppose this was not true, and let $F\subseteq G$ be the finite set consisting of $g\in G$ with $c(g)\neq 0$. By assumption, $c = c\delta_{h}$ for all $h\in Y_{i}$, which implies $Fh = F$ and hence $Y_{i}\subseteq F^{-1}F$, contradicting the fact that $Y_{i}$ is infinite. 

    So, let $b := \Pi^{n}_{i=1}(\delta_{e} - \delta_{h_{i}})$ such that $ab\neq 0$. By Proposition \ref{prp:resttosg}, to prove the lemma, it suffices to show $b\in J_{N, N\cap\mathcal{X}}\cap \mathbb{C}[N]$.

    For $g\in N$ and $Y\in N\cap\mathcal{X}$, we have (by normality of $Y$ in $N$)
     $$b(gY) = \sum_{g_{1}\cdot\cdot\cdot g_{n} = g}(\delta_{e}-\delta_{h_{1}})(g_{1}Y)\cdot...\cdot (\delta_{e}-\delta_{h_{n}})(g_{n}Y).$$

    We consider the summands above. By construction, there is $i\leq n$ such that $Y_{i}\subseteq Y$ and note that $(\delta_{e} - \delta_{h_{i}})(kY_{i}) = 0$, for all $k\in N$. Hence, $(\delta_{e} - \delta_{h_{i}})(g_{i}Y) = 0$.  It follows that $b(gY) = 0$, proving the lemma. 
\end{proof}

Now, we apply this lemma to get one of our main results of this section.

\begin{thm}
\label{thm:finitenumberofsgs}
    Let $\Gamma$ be a discrete group and $\mathcal{X}$ a finite set of subgroups invariant under conjugation. Then, $(\Gamma, \mathcal{X})\in \mathcal{I}$.
\end{thm}

\begin{proof}
    If $\Gamma$ is finite, then $(\Gamma,\mathcal{X})\in\mathcal{I}$ follows from Lemma \ref{lem:finite}.
    Assume $J_{\Gamma,\mathcal{X}}\neq \{0\}$. By Proposition \ref{prp:gideal=rgideal}, it follows that  $J_{\Gamma, \mathcal{X}_{<\infty}}\neq 0$ and $ J_{\Gamma,\mathcal{X}_{\infty}}\neq 0$ where $\mathcal{X}_{<\infty}, \mathcal{X}_{\infty}$ consist of the subgroups in $\mathcal{X}$ which are finite, infinite, respectively. It follows by Lemma \ref{lem:finite} that there is a non-zero $a\in \mathbb{C}[N]\cap J_{\Gamma,\mathcal{X}_{<\infty}}$, where $N$ is the finite subgroup generated by $X\in \mathcal{X}_{<\infty}$.

    By Lemma \ref{lem:prodelementconstr}, there is $b\in J_{\Gamma, \mathcal{X}_{\infty}}\cap\mathbb{C}[\Gamma]$ such that $0\neq ab\in (J_{\Gamma,\mathcal{X}_{<\infty}}\cap\mathbb{C}[\Gamma])(J_{\Gamma,\mathcal{X}_{\infty}}\cap\mathbb{C}[\Gamma]\subseteq J_{\Gamma, \mathcal{X}_{<\infty}}\cap J_{\Gamma,\mathcal{X}_{\infty}}\cap\mathbb{C}[\Gamma]= J_{\Gamma,\mathcal{X}}\cap\mathbb{C}[\Gamma]$, proving the theorem.
\end{proof}

\begin{cor}
\label{cor:finitetooneproj}
    If $G$ is an \'etale groupoid with $\pi_{ess}:\tilde{G}^{0}_{ess}\to G^{0}$ finite to one, then $J\cap\mathscr{C}_{c}(G)\neq \{0\}$ if and only if $J\neq \{0\}$.
\end{cor}
\begin{proof}
    This is an immediate application of Theorem \ref{thm:valgideal=videal} and Theorem \ref{thm:finitenumberofsgs}.
\end{proof}

Here is another application of Lemma \ref{lem:prodelementconstr}.

\begin{prp}
\label{prp:freeabelianAI}
    Every countable torsion free abelian group $\Gamma$ satisfies Property $AI$.
\end{prp}
\begin{proof}
This is an immediate application of case (2) in Lemma \ref{lem:prodelementconstr}.
\end{proof}

Actually, the above result is a special case of a more general permanence result for the intersection properties, stated and proven below.

\begin{prp}
\label{prp:freeabelianextAI}
    Let $\Gamma$ be a discrete group with a normal subgroup $N$ that satisfies Property $AI$ and suppose $\Gamma/N$ is torsion free and abelian. Then, $\Gamma$ satisfies Property $AI$.
\end{prp}

\begin{proof}
    Let $q:\Gamma\to \Gamma/N = \Lambda$ be the quotient map. Since $\Lambda$ is an increasing union $\bigcup_{n}\Lambda_{n}$ of finitely generated torsion free and abelian subgroups, $\Gamma$ is an increasing union of subgroups $\Gamma_{n} = q^{-1}(\Lambda_{n})$ with $N\subseteq \Gamma_{n}$ and $\Gamma_{n}/N \simeq \mathbb{Z}^{m_{n}}$. Therefore, by Corollary \ref{cor:pAI&ctbleunions}, it suffices to prove the proposition in the case that $\Gamma/N = \mathbb{Z}^{n}$. Moreover, it suffice to prove the case where $n=1$. This follows from the fact that the subgroup $N_{k} = q^{-1}(\mathbb{Z}^{k-1}\oplus 0_{n - k +1})$ for $k\leq n$ is normal in $N_{k+1}$ and $N_{k+1}/N_{k} = \mathbb{Z}$, so a simple induction argument once the $n=1$ case is established would prove for the case of general $n\in\mathbb{N}$.

    So, assume $\Gamma/N = \mathbb{Z}$ and $\mathcal{X}$ is a closed set of subgroups invariant under conjugation with $\{e\}\notin \mathcal{X}$. Consider $\mathcal{X}' = \{X\in\mathcal{X}:X\cap N = \{e\}\}$. If $\mathcal{X}' = \emptyset$, then $\{e\}\notin N\cap\mathcal{X}$. Since $N$ is assumed to have Property $AI$, it follows that $0\neq J_{N, N\cap\mathcal{X}}\cap\mathbb{C}[N]\subseteq J_{\Gamma, \mathcal{X}}\cap\mathbb{C}[\Gamma]$, which proves the proposition in this case.

    Now, assume $\mathcal{X}'\neq \emptyset$. It is easy to see $\mathcal{X}'$ is closed and conjugate invariant (since $N$ is normal).
    
    For each $X\in\mathcal{X}'$, the map $x\in X\to q(x)\in q(X)$ is injective, so it follows from $\Gamma/N = \mathbb{Z}$ that $X$ is singly generated; write $X = \langle x\rangle$ for some $x\in X$. Since $\{e\}\notin\mathcal{X}'$, by compactness, there is a finite set $F\subseteq\Gamma\setminus \{e\}$ such that $F\cap X\neq \emptyset$ for all $X\in\mathcal{X}$. Therefore, for every $X$, there is $f(X)\in F$ such that $\langle f(X)\rangle \subseteq \langle x\rangle$. Moreover, for every $X\in\mathcal{X}'$, there is $n(X)\in\mathbb{Z}$ such that $x^{n(X)} = f(X)$. Therefore $q(f(X)) = q(x)^{n(X)}$ and since $q(x)\neq 0$ for all $X\in\mathcal{X}'$ and $F$ finite, it follows that there is $n\in\mathbb{N}$ such that $n(X)\leq n$ for all $X\in\mathcal{X}'$. 
    
    Now, for $g\in \Gamma$, we have either  $gf(g^{-1}Xg)g^{-1} = x^{n(g^{-1}Xg)}$ or $gf(g^{-1}Xg)g^{-1} = x^{ - n(g^{-1}Xg)}$. Therefore, $F(X) = \bigcap_{g\in \Gamma}\langle gf(g^{-1}Xg)g^{-1}\rangle = \langle x^{m(X)}\rangle $ for some $0 \neq m(X)\leq n^{n} $. Moreover, since the group $F(X)$ has uniformly (in $X$) bounded index in $\langle f(X)\rangle $, the set $\mathcal{F} = \{F(X):X\in \mathcal{X}'\}$ is finite. By construction, $hF(X)h^{-1} = F(hXh^{-1})$ for all $h\in \Gamma$ and $X\in\mathcal{X}$. Therefore, $\mathcal{F}$ is conjugate invariant.

    Write $\mathcal{F} = \{\langle y_{1}\rangle,..., \langle y_{l}\rangle \}$. By conjugate invariance, the finite set $K = \bigcup^{l}_{i=1}\{y_{i},y_{i}^{-1}\}$ satisfies  $gKg^{-1} = K$ for all $g\in \Gamma$. For $k\in K$, let $Z(k) = \{X\in \{0,1\}^{\Gamma}: k\in X\}$. Then, the set $Z(K) = \bigcup_{k\in K} Z(k)$ is clopen, conjugation invariant and $\mathcal{X}'\subseteq Z(K)$, so $\mathcal{Y} = \mathcal{X}\setminus Z(K) \subseteq\{X\in\mathcal{X}: X\cap N\neq \{e\}\}$ is closed and conjugation invariant, and is of the type considered in the first case. Therefore, there is a non-zero element $a\in J_{N, N\cap\mathcal{Y}}\cap\mathbb{C}[N]\subseteq J_{\Gamma, \mathcal{Y}}\cap\mathbb{C}[\Gamma]$.

    The closed and conjugate invariant set $\mathcal{Z} = \mathcal{X}\cap Z(K)$ has the property that for every $Z\in\mathcal{Z}$, there is $Y\in\mathcal{F}$ such that $Y\subseteq Z$. Therefore, $J_{\Gamma, \mathcal{F}}\cap\mathbb{C}[\Gamma]\subseteq J_{\Gamma, \mathcal{Z}}\cap\mathbb{C}[\Gamma]$. By Lemma \ref{lem:prodelementconstr} case (1), there is $b\in J_{\Gamma, \mathcal{F}}\cap\mathbb{C}[\Gamma]$ such that $ab\neq 0$. It follows from Corollary \ref{cor:gidealunion=int} that $ab\in J_{\Gamma, \mathcal{Y}\cup\mathcal{Z}}\cap\mathbb{C}[\Gamma]=J_{\Gamma, \mathcal{X}}\cap\mathbb{C}[\Gamma]$. This proves $\Gamma$ satisfies Property $AI$.
\end{proof}

Now we prove another large class of closed invariant sets of subgroups satisfying a finiteness condition that is, in some sense, perpendicular to that in Theorem \ref{thm:finitenumberofsgs} belongs to $\mathcal{I}$. It will also be useful to us as a lemma to prove the next permanence result for the intersection properties.
\begin{prp}
\label{prp:finitesgs}
    Let $\Gamma$ be a discrete group and $\mathcal{X}$ a closed invariant set of subgroups such that every $X\in\mathcal{X}$ is finite. Then, there is a finite conjugate invariant subset $\mathcal{F}\subseteq \mathcal{X}$ such that for every $X\in\mathcal{X}$, there is $Y\in\mathcal{F}$ with $Y\subseteq X$.
    
    Consequently, $(\Gamma, \mathcal{X})\in\mathcal{I}$.
\end{prp}
\begin{proof}
    If $\{e\}\in \mathcal{X}$, then we can take $\mathcal{F} = \{\{e\}\}$. Let's assume $\{e\}\notin \mathcal{X}$.
     For each $n\in\mathbb{N}$, let $\mathcal{X}_{n} = \{X\in\mathcal{X}: |X|\leq n\}$. Then, $\mathcal{X}_{n}$ is closed and conjugate invariant. Moreover, $\mathcal{X}_{2}$ is finite; if $(X_{n} = \{e, a_{n}\})$ is an infinite and pairwise distinct collection in $\mathcal{X}_{2}$, then we can extract a subsequence such that for every finite set $F\subseteq \Gamma$ $a_{n}\notin F$ eventually. Then, $X_{n}$ converges to $\{e\}$, which would be a contradiction.

    Set $Z_{2} = \{X\in \mathcal{X}: Y\subseteq X \text{ for some } Y\in\mathcal{X}_{2}\}$. Since $\mathcal{X}_{2}$ is a finite collection of finite subgroups, $Z_{2}$ is clopen in $\mathcal{X}$.

    Now, set $\mathcal{Y}_{2} = \mathcal{X}_{2}$ and define inductively for $n> 2$ $$Z_{n-1} = \{X\in \mathcal{X}: Y\subseteq X\text{ for some }Y\in\bigcup_{k\leq n-1}\mathcal{Y}_{k}\}\text{ and } \mathcal{Y}_{n} = \mathcal{X}_{n}\setminus Z_{n-1}.$$ We claim for every $n\in\mathbb{N}$, $\mathcal{Y}_{n}$ is finite (therefore $Z_{n}$ is clopen in $\mathcal{X}$), conjugate invariant and $\mathcal{X}_{n}\subseteq Z_{n}$.

    Suppose we know this is true for $k\leq n-1$, and let's prove it for $n$. Conjugate invariance of $Z_{n-1}$ follows from conjugate invariance of $\bigcup_{k\leq n-1}\mathcal{Y}_{k}$. Therefore, $\mathcal{Y}_{n} = \mathcal{X}_{n}\setminus Z_{n-1}$ is conjugate invariant. Suppose $(Y_{n})\subseteq \mathcal{Y}_{n}$ is an infinite sequence with pairwise distinct elements. Since $Z_{n-1}$ is clopen, and $\mathcal{X}_{n}$ is closed, we know that $\mathcal{Y}_{n}$ is closed, so we can extract a sub-sequence $(Y_{n_{k}})$ converging to $Y\in\mathcal{Y}_{n}$. Since $Y$ is finite, we have $Y\subseteq Y_{n_{k}}$ eventually. Therefore $|Y| < |Y_{n_{k}}|\leq n$ eventually. By the inductive hypothesis, we have $Y\in Z_{n-1}$. But $Y\in\mathcal{Y}_{n}\subseteq \mathcal{X}\setminus Z_{n-1}$, a contradiction. Therefore $\mathcal{Y}_{n}$ is finite. Since $Z_{n-1}\subseteq Z_{n}$ and $\mathcal{X}_{n}\setminus Z_{n-1}\subseteq Z_{n}$, it follows that $\mathcal{X}_{n}\subseteq Z_{n}$, proving the claim by induction.

    Now, since $Z_{n}\subseteq Z_{n+1}$ for all $n\in\mathbb{N}$ and $\mathcal{X}\subseteq \bigcup_{n}Z_{n}$, compactness implies there is $n\in\mathbb{N}$ such that $\mathcal{X}\subseteq Z_{n}$. Therefore, there is a conjugate invariant finite set $\mathcal{F} = \bigcup^{n}_{k=1}\mathcal{Y}_{k}$ such that for every $X\in\mathcal{X}$, there is $Y\in\mathcal{F}$ such that $Y\subseteq X$. Therefore, $J_{\Gamma,\mathcal{F}}\cap \mathbb{C}[\Gamma]\subseteq J_{\Gamma, \mathcal{X}}\cap\mathbb{C}[\Gamma]$. This containment, together with $J_{\Gamma, \mathcal{X}}\subseteq J_{\Gamma,\mathcal{F}}$ and Theorem \ref{thm:finitenumberofsgs} implies $(\Gamma,\mathcal{X})\in\mathcal{I}$.
\end{proof}

Now, we prove our second major permanence result for the intersection properties.

\begin{prp}
\label{prp:finiteindexext}
    If $\Gamma$ is a discrete group containing a finite index Property $AI$ group, then $\Gamma$ has Property $I$. If $\Gamma$ is, additionally, torsion free, then $\Gamma$ has Property $AI$.
\end{prp}

\begin{proof}
    Let $N$ be a torsion free Property $AI$ subgroup of finite index in $\Gamma$ and let $K = \{g\in \Gamma: ghN = hN,\text{ for all }h\in \Gamma\}$. Then, $K\subseteq N$ is a normal subgroup of $\Gamma$ such that $|\Gamma/K| <\infty$. Let $\mathcal{X}$ be a closed set of subgroups invariant under conjugation by $\Gamma$ such that $\{e\}\notin\mathcal{X}$. Consider $\mathcal{X}' = \{X\in\mathcal{X}: X\cap K = \{e\}\}$. Note that each $X\in\mathcal{X}'$ is finite.

    If $\mathcal{X}' = \emptyset$ (for instance if $\Gamma$ is torsion free), then $\{e\}\notin N\cap\mathcal{X}$, so we can apply Property $AI$ of $N$ and Proposition \ref{prp:resttosg} to show $J_{\Gamma,\mathcal{X}}\cap\mathbb{C}[\Gamma]\neq \{0\}$.

    So, assume $\mathcal{X}'\neq \emptyset$. Then, $\mathcal{X}'$ is closed and each $X\in \mathcal{X}'$ is finite. Proposition \ref{prp:finitesgs} implies there is a conjugate invariant finite set $\mathcal{F}\subseteq \mathcal{X}'$ such that for every $X\in\mathcal{X}$, there is $Y\in\mathcal{F}$ such that $Y\subseteq X$. Therefore, $J_{\Gamma, \mathcal{F}}\cap\mathbb{C}[\Gamma] = J_{\Gamma,\mathcal{X}'}\cap\mathbb{C}[\Gamma]$. Similarly, if we let $Z(Y)$, for $Y\in \mathcal{F}$, be the clopen set  $\{X\in\mathcal{X}: Y\subseteq X\}$ and $\mathcal{Z} = \bigcup_{Y\in\mathcal{F}}Z(Y)$, then $\mathcal{Z}\subseteq \mathcal{X}$ is a clopen and conjugate invariant subset which satisfies $J_{\Gamma, \mathcal{F}}\cap\mathbb{C}[\Gamma] = J_{\Gamma, \mathcal{Z}}\cap\mathbb{C}[\Gamma]$.
    
    By assumption, we know that $\{0\}\neq J_{\Gamma,\mathcal{X}}\subseteq J_{\Gamma, \mathcal{F}}$, so Theorem \ref{thm:finitenumberofsgs} and Corollary \ref{cor:redtonsg} combined implies there is a non-zero element in $a\in J_{\Gamma, \mathcal{F}}\cap\mathbb{C}[F] = J_{\Gamma, \mathcal{Z}}\cap\mathbb{C}[F]$, where $F$ is the finite subgroup generated by $Y\in\mathcal{F}$.

    Now, $\mathcal{Y} = \mathcal{X}\setminus \mathcal{Z}$ is a closed and conjugate invariant set of subgroups such that $\{e\}\notin K\cap \mathcal{Y}$. Note that $K\cap \mathcal{Y}$ is a closed set of subgroups of $N$ that is invariant under conjugation by elements in $N$. It follows from Property $AI$ for $N$ and Proposition \ref{prp:resttosg} that there is a non-zero $b\in J_{N, K\cap \mathcal{Y}}\cap\mathbb{C}[N]\subseteq J_{N, N\cap \mathcal{Y}}\cap\mathbb{C}[N]\subseteq J_{\Gamma, \mathcal{Y}}\cap\mathbb{C}[\Gamma]$. Moreover, from Corollary \ref{cor:redtonsg} and the fact that the normal subgroup generated by elements in $K\cap\mathcal{Y}$ is inside $K$, we can choose $b\in \mathbb{C}[K]$.

    Now, let's show $ab\neq 0$. by applying $\delta^{-1}_{g}$ for some $g\in F$ such that $a(g)\neq 0$, we can assume $a(e)\neq 0$. Let $\Phi:\mathbb{C}[\Gamma]\to \mathbb{C}[K]$ be the restriction map. Since $F\cap K = \{e\}$ we must have $\Phi(ab) = a(e)b\neq 0$. Therefore, $ab\neq 0$. Since $\mathcal{Y}\cup \mathcal{Z} = \mathcal{X}$, it follows that $ab\in J_{\Gamma,\mathcal{X}}\cap\mathbb{C}[\Gamma]$.
\end{proof}

\begin{cor}
\label{cor:abelianAI}
    Let $\Gamma$ be a torsion free discrete group with a normal subgroup $N$ that satisfies Property $AI$ and suppose $\Gamma/N$ is abelian. Then, $\Gamma$ satisfies Property $AI$.
\end{cor}

\begin{proof}
    Write $q:\Gamma\to \Gamma/N = \bigcup_{n}\Lambda_{n}$, where $(\Lambda_{n})$ is an increasing sequence of finitely generated abelian groups. By Corollary \ref{cor:ctbleunion}, it suffices to show $q^{-1}(\Lambda_{n})$ has Property $AI$. Therefore, we can assume, without loss of generality, that $\Gamma/N = \mathbb{Z}^{n}\oplus F$, where $n\in\mathbb{N}$ and $F$ is a finite abelian group. By Proposition \ref{prp:freeabelianAI}, $q^{-1}(\mathbb{Z}^{n}) = N'$ satisfies Property $AI$. Since $\Gamma$ is torsion free, $N'$ is normal and $\Gamma/N'$ is finite, it follows from Proposition \ref{prp:finiteindexext} that $\Gamma$ satisfies Property $AI$. 
\end{proof}

As a simple corollary to the permanence properties we have established for Property $I$ and $AI$ groups, we have the following.

\begin{thm}
\label{thm:propertyAI&I}
    Direct limits of torsion free virtually solvable groups satisfy Property $AI$. Direct limits of virtually torsion free solvable groups satisfy Property $I$. 
    
    In particular, every group of polynomial growth or amenable matrix group (over a characteristic zero field) satisfies Property $I$ and moreover Property $AI$ if the group is torsion free.
\end{thm}

\begin{proof}
    Since the class of virtually torsion free solvable groups is preserved by quotients and subgroups, we can replace a direct limit of such groups with a countable increasing union. Similarly for torsion free virtually solvable groups. Therefore, by Corollary \ref{cor:pAI&ctbleunions}, it suffices to prove virtually torsion free solvable groups satisfy Property $I$ and torsion free virtually solvable groups satisfy Property $AI$.
    
    Suppose $\Gamma$ is solvable and torsion free. By definition, there exists a sequence $\{e\} = N_{0}\subseteq N_{1}\subseteq ...\subseteq N_{k} = \Gamma$ of subgroups such that $N_{i-1}$ is normal in $N_{i}$ for all $1\leq i\leq k$ and $N_{i}/N_{i-1}$ is abelian. Therefore, $N_{1} = N_{1}/\{e\}$ is torsion free and abelian, and satisfies Property $AI$ by
    Proposition \ref{prp:freeabelianAI}. Suppose, for the sake of induction that $N_{i}$ satisfies Property $AI$ for all $i\leq j < k$. Since $N_{j+1}/N_{j}$ is abelian and $N_{j+1}$ is torsion free, Corollary \ref{cor:abelianAI} implies $N_{j+1}$ has Property $AI$. Therefore, by induction, $\Gamma$ has Property $AI$.

    The fact that $\Gamma$ satisfies Property $I$ if it is virtually torsion free and solvable, and additionally Property $AI$ if it is torsion free, follows immediately from what we have just proven and Proposition \ref{prp:finiteindexext}. The fact that countable increasing unions of these groups satisfies the respective conditions follows immediately from Corollary \ref{cor:pAI&ctbleunions}.

    By Gromov's theorem \cite{G81}, every finitely generated group of polynomial growth is virtually nilpotent, and hence virtually solvable. Moreover, by \cite[Theorem~2.1]{B71}, every finitely generated nilpotent group has a torsion free subgroup of finite index. Since nilpotency passes to subgroups, it follows from what we have proven above that every finitely generated group of polynomial growth satisfies Property $I$ and Property $AI$ if it is torsion free.

    By Tits alternative \cite{T71}, every amenable matrix group is virtually solvable. By Selberg's lemma, every matrix group over a characteristic zero field is virtually torsion free \cite{S60}.
\end{proof}

Now, let's characterize when an abelian group satisfies Property $AI$. We first note the following characterization of Property $AI$ for finite groups.

\begin{prp}
\label{prp:finitechar}
    A finite group $\Gamma$ has Property $AI$ if and only if the collection $\mathcal{X}_{min}$ of non-trivial minimal subgroups satisfies $J_{\Gamma, \mathcal{X}_{min}}\neq \{0\}$ if and only if the $\Gamma$ invariant subspace $\text{span}_{\mathbb{C}}\{\delta_{gX}: g\in \Gamma, X\in\mathcal{X}_{min}\}\neq \mathbb{C}[\Gamma]$.
\end{prp}
\begin{proof}
    The first ``if and only if'' follows from the fact that, for an arbitrary conjugate invariant set of subgroups $\mathcal{X}$ of $\Gamma$ not containing $\{e\}$, the collection of non-trivial minimal subgroups $\mathcal{X}'$ that are contained in elements of $\mathcal{X}$ satisfies $J_{\Gamma, \mathcal{X}_{min}}\subseteq J_{\Gamma,\mathcal{X}'}\subseteq J_{\Gamma,\mathcal{X}}$ (the sums of characteristic functions on co-sets of $\mathcal{X}'$ contain the characteristic functions on co-sets of $\mathcal{X}$). The second ``if and only if'' follows from the first and Lemma \ref{lem:alggroupidealchar}.
\end{proof}

Recall that the minimal subgroups in a finite group are the cyclic subgroups with prime order. We now prove our characterization of Property $AI$ for abelian groups.

\begin{thm}
\label{thm:abeliancharAI}
    A discrete abelian group $\Gamma$ satisfies Property $AI$ if and only if for every prime $p$, there is at most one element $g\in \Gamma$ with cyclic order $p$.
\end{thm}
\begin{proof}
    By Corollary \ref{cor:pAI&ctbleunions}, it suffices to prove the characterization for finitely generated abelian groups. We prove the `if'' direction first. By the fundamental theorem for finitely generated abelian groups, we can write $\Gamma\simeq \mathbb{Z}^{m}\oplus \mathbb{Z}/p^{n_{1}}_{1}\mathbb{Z}\oplus....\oplus\mathbb{Z}/p_{k}^{n_{k}}\mathbb{Z} = \mathbb{Z}^{m}\oplus G$. Suppose $\mathcal{X}$ is a closed set of subgroups with $\{e\}\notin \mathcal{X}$. Let $\mathcal{Z}$ be the collection of $X\in\mathcal{X}$ which contain a finite subgroup. Since there are only finitely many finite subgroups in $\mathbb{Z}^{m}\oplus \mathbb{Z}/p^{n_{1}}_{1}\mathbb{Z}\oplus....\oplus\mathbb{Z}/p_{k}^{n_{k}}\mathbb{Z}$, we have that $\mathcal{Z}$ is clopen in $\mathcal{X}$, so that $\mathcal{Y} := \mathcal{X}\setminus \mathcal{Z} = \{X\in \mathcal{X}: X\subseteq \mathbb{Z}^{m}\}$ is also clopen in $\mathcal{X}$. By Proposition \ref{prp:freeabelianAI} and Corollary \ref{cor:redtonsg}, if $\mathcal{Y}\neq \emptyset$, then there is $0\neq a\in J_{\Gamma, \mathcal{Y}}\cap\mathbb{C}[\mathbb{Z}^{m}]$. This also  proves the ``if'' direction in the case that $\mathcal{Z} = \emptyset$.

    Suppose $\mathcal{Z}\neq\emptyset$. Let's show $J_{\Gamma,\mathcal{Z}}\cap\mathbb{C}[G]\neq \{0\}$. If $x$ has prime cyclic order, then since $|G| = \Pi^{k}_{i=1}p^{n_{i}}$ and $|\langle x\rangle|$ divides $|G|$, we have $p = p_{i}$, for some $i\leq n$. It follows by the hypothesis that $\langle x\rangle  = \langle p^{n_{i}-1}1_{i}\rangle $, where $1_{i}$ is the canonical cyclic generator for the factor $\mathbb{Z}/p_{i}^{n_{i}}\mathbb{Z}$. Let $b = \Pi^{k}_{i=1}(\delta_{e} - \delta_{p^{n_{i}-1}1_{i}})$. Under the identification $\mathbb{C}[G]\simeq \mathbb{C}[\mathbb{Z}/p^{n_{1}}_{1}]\otimes....\otimes \mathbb{C}[\mathbb{Z}/p^{n_{k}}_{k}] $, the element $b$ corresponds to the basic tensor $(\delta_{e} - \delta_{p^{n_{1}-1}1_{1}})\otimes...\otimes (\delta_{e} - \delta_{p^{n_{k}-1}1_{k}})$ and therefore $b\neq 0$. Letting $\mathcal{X}_{min}$ be the minimal subgroups of $G$, it is easy to see $b\in J_{G,\mathcal{X}_{min}}$, following the same argument as in Lemma \ref{lem:prodelementconstr} and using the fact that $(\langle p^{n_{i}-1}1_{i}\rangle)^{k}_{i=1}$ are all the minimal subgroups of $G$. Since every $Z\in\mathcal{Z}$ has $X\subseteq Z$ for some $X\in\mathcal{X}_{min}$, it follows that $\{0\}\neq J_{G,\mathcal{X}_{min}}\subseteq J_{\Gamma, \mathcal{Z}}\cap\mathbb{C}[G]$. This also proves the ``if'' direction in the case $\mathcal{Y} = \emptyset$.

    Now, suppose that $\mathcal{Y}\neq\emptyset$ and $\mathcal{Z}\neq\emptyset$. We have $a\in \mathbb{C}[\mathbb{Z}^{m}]$ and $b\in \mathbb{C}[G]$. Since $\Gamma = \mathbb{Z}^{m}\oplus G$, we have $\mathbb{C}[\Gamma]\simeq \mathbb{C}[\mathbb{Z}^{m}]\otimes \mathbb{C}[G]$. Under this identification, $ab$ is the basic tensor $a\otimes b\neq 0$. Therefore, $0\neq ab\in J_{\Gamma, \mathcal{Y}}\cap J_{\Gamma, \mathcal{Z}}\cap\mathbb{C}[\Gamma] = J_{\Gamma, \mathcal{X}}\cap\mathbb{C}[\Gamma]$. This proves the ``if'' direction of the theorem.

    Now, we prove the converse. Suppose there are two distinct elements $x_{1},x_{2}$ each with prime cyclic order $p$. Then, $(\mathbb{Z}/p\mathbb{Z})^{2}\simeq \langle x_{1},x_{2}\rangle \subseteq \Gamma$. To prove $\Gamma$ does not satisfy Property $AI$ it suffices, by Proposition \ref{prp:videalredtonsg} and Proposition \ref{prp:finitechar}, to show the minimal subgroups $\mathcal{X}_{min}$ of $(\mathbb{Z}/p\mathbb{Z})^{2}$ satisfy $V:=  \text{span}_{\mathbb{C}}\{\delta_{g + X}: g\in (\mathbb{Z}/p\mathbb{Z})^{2}, X\in\mathcal{X}_{min}\} =  \mathbb{C}[(\mathbb{Z}/p\mathbb{Z})^{2}]$.

    Note that $(\mathbb{Z}/p\mathbb{Z})^{2} = \mathbb{F}^{2}_{p}$, where $\mathbb{F}^{2}_{p}$ is the two-dimensional vector space over the finite field $\mathbb{Z}/p\mathbb{Z} = \mathbb{F}_{p}$. Under this identification, the group structure is the additive structure of the vector space and so $\{e\}$ corresponds to the origin $0$. Moreover, the minimal subgroups $\mathcal{X}_{min}$ correspond to the set of one-dimensional subspaces (lines) of $\mathbb{F}^{2}_{p}$, while the co-sets of the subgroups $\mathcal{X}_{min}$ correspond to the collection of affine lines $\mathcal{L}$ (i.e. lines not necessarily centered at the origin). For $0\neq x\in \mathbb{F}^{2}_{p}$, let $n(x)$ be the number of affine lines $L$ containing $x$ but not $0$. Since the group of invertible linear operators on a finite dimensional vector (over an arbitrary field) acts transitively on non-zero vectors, we have that $n(x) = n(y)$ for all $x,y\in \mathbb{F}^{2}_{p}\setminus \{0\}$. Therefore,
    $$\sum_{L\in \mathcal{L}:0\notin L}\delta_{L} = \sum_{x\in \mathbb{F}^{2}_{p}\setminus \{0\}}n(x)\delta_{x} = n\delta_{\mathbb{F}^{2}_{p}\setminus \{0\}}.$$
    Since $\delta_{\mathbb{F}^{2}_{p}}\in V$ (it is the sum of the characteristic functions of the cosets of any fixed line), we have $\delta_{0} = \delta_{\mathbb{F}^{2}_{p}} - \frac{1}{n}\sum_{L\in \mathcal{L}:0\notin L}\delta_{L}\in V$. From invariance ($\delta_{x}V = V$ for all $x\in \mathbb{F}^{2}_{p}$), it follows that $\delta_{x}\delta_{0} = \delta_{x}\in V$, for all $x\in\mathbb{F}^{2}_{p}$. Hence, $V = \mathbb{C}[\mathbb{F}^{2}_{p}]$.
\end{proof}
We leave the reader with a question.

\begin{qtn}
    Does every amenable group satisfy Property $I$, and moreover Property $AI$ if the group is torsion free?
\end{qtn}

\end{document}